\documentclass[12pt]{amsart}
\usepackage[dvips]{graphics}
\usepackage[latin1]{inputenc}
\usepackage{amssymb, amsmath, amscd, mathrsfs,marvosym, psfrag, picins}
\input xy
\xyoption{all}
\DeclareMathAlphabet{\mathpzc}{OT1}{pzc}{m}{it}

\newtheorem{theorem}{Theorem}[section]

\newtheorem{proposition}[theorem]{Proposition}
\newtheorem{corollary}[theorem]{Corollary}
\newtheorem{conjecture}[theorem]{Conjecture}
\newtheorem{lemma}[theorem]{Lemma}

\newtheorem*{MConj}{Conjecture}

\theoremstyle{definition}
\newtheorem{definition}[theorem]{Definition}

\theoremstyle{remark}
\newtheorem{remark}[theorem]{Remark}

\def\varle{\leqslant}

\newcommand{\CA}{{\mathcal A}}

\newcommand{\CC}{{\mathcal C}}

\newcommand{\CE}{{\mathcal E}}

\newcommand{\CG}{{\mathcal G}}
\newcommand{\CH}{{\mathcal H}}
\newcommand{\CI}{{\mathcal I}}

\newcommand{\CK}{{\mathcal K}}

\newcommand{\CM}{{\mathcal M}}

\newcommand{\CO}{{\mathcal O}}

\newcommand{\CR}{{\mathcal R}}
\newcommand{\CS}{{\mathcal S}}
\newcommand{\CT}{{\mathcal T}}

\newcommand{\CV}{{\mathcal V}}
\newcommand{\CW}{{\mathcal W}}

\newcommand{\CZ}{{\mathcal Z}}

\newcommand{\SA}{{\mathscr A}}
\newcommand{\SB}{{\mathscr B}}

\newcommand{\SF}{{\mathscr F}}
\newcommand{\SG}{{\mathscr G}}

\newcommand{\SO}{{\mathscr O}}

\newcommand{\ScW}{{\mathscr W}}

\newcommand{\fh}{{{\mathfrak h}}}

\newcommand{\fg}{{{\mathfrak g}}} 
\newcommand{\fb}{{{\mathfrak b}}} 
 
\newcommand{\fl}{{{\mathfrak l}}}

\newcommand{\hCW}{{\widehat\CW}}

\newcommand{\hCG}{{\widehat\CG}}

\newcommand{\hCS}{{\widehat\CS}}
\newcommand{\hCT}{{\widehat\CT}}
\newcommand{\hCZ}{{\widehat\CZ}}

\newcommand{\hS}{{\widehat S}}

\newcommand{\hT}{{\widehat T}}

\newcommand{\hV}{{\widehat V}}
\newcommand{\hR}{{\widehat R}}
\newcommand{\hX}{{\widehat X}}

\newcommand{\hH}{{\widehat H}}

\newcommand{\hw}{\widehat w}

\newcommand{\tP}{\widetilde{P}}
\newcommand{\tZ}{\widetilde{Z}}

\newcommand{\tCM}{\widetilde{\CM}}
\newcommand{\tCC}{\widetilde{\CC}}

\newcommand{\DC}{{\mathbb C}}

\newcommand{\DR}{{\mathbb R}}
\newcommand{\DZ}{{\mathbb Z}}

\newcommand{\DN}{{\mathbb N}}
\newcommand{\DQ}{{\mathbb Q}}
\newcommand{\DV}{{\mathbb V}}

\newcommand{\bM}{{\mathbf M}}
\newcommand{\bW}{{\mathbf W}}
\newcommand{\bH}{{\mathbf H}}

\newcommand{\bs}{{\mathbf s}}

\newcommand{\dercat}{{\textsf{D}}}

\newcommand{\height}{{\operatorname{ht}}}
\newcommand{\ch}{{\operatorname{char}\, }}
\newcommand{\im}{{\operatorname{im}}}
\newcommand{\Spec}{{\operatorname{Spec}}}

\newcommand{\Hom}{{\operatorname{Hom}}}

\newcommand{\Ind}{{\operatorname{Ind}}}

\newcommand{\Fl}{{\operatorname{{\mathcal{F}{l}}}}}

\newcommand{\Res}{{\operatorname{Res}}}
\newcommand{\supp}{{\operatorname{supp}}}
\newcommand{\catmod}{{\operatorname{-mod}}}

\newcommand{\rk}{{{\operatorname{rk}}}}
\newcommand{\grk}{{{\operatorname{\underline{rk}}}}}

\newcommand{\ol}{\overline}

\newcommand{\pt}{{pt}}

\newcommand{\llinie}{{\text{---\!\!\!---\!\!\!---}}}

\newcommand{\Hyp}{{\operatorname{\mathbb H}}}
\newcommand{\hFl}{{\widehat{\Fl}}}

\newcommand{\GL}{{\operatorname{GL}}}
\newcommand{\Aff}{{\operatorname{Aff}}}
\newcommand{\IC}{{\operatorname{IC}}}

\newcommand{\on}{\vartheta_{on}}

\newcommand{\son}{{\vartheta}^s_{on}} 
\newcommand{\sout}{{\vartheta}^s_{out}}

\newcommand{\tson}{\widetilde{\vartheta}^s_{on}} 
\newcommand{\tsout}{\widetilde{\vartheta}^s_{out}}

\newcommand{\CTon}{\CT_{on}}
\newcommand{\CTout}{\CT_{out}}

\newcommand{\CTson}{{\CT^s_{on}}}
\newcommand{\CTsout}{{\CT^s_{out}}}

\newcommand{\out}{\vartheta_{out}}
\newcommand{\ul}{\underline}
\newcommand{\ttheta}{\tilde\theta}

\begin{document}

\pagenumbering{arabic}
\title[]{Sheaves on affine Schubert varieties, modular  representations
  and Lusztig's conjecture} 
\author[]{Peter Fiebig}
\address{Department Mathematik, Universit{ä}t Erlangen-N\"urnberg, Germany}
\email{fiebig@mi.uni-erlangen.de}

\begin{abstract} We relate a certain category of sheaves of $k$-vector spaces on
  a complex affine Schubert variety to modules over the  $k$-Lie
  algebra  (for $\ch k>0$) or to modules over the small quantum group (for $\ch k=0$)  associated to the Langlands dual root datum. As an application we give a  new proof of Lusztig's
  conjecture on quantum  characters and on modular characters for almost all characteristics. Moreover, we relate the geometric and representation theoretic sides to sheaves on the underlying moment graph, which allows us to extend the known instances of Lusztig's modular conjecture in two directions: We give an upper bound on the exceptional characteristics and verify its multiplicity one case for all relevant primes. 
\end{abstract}

\maketitle

\section{Introduction}
One of the fundamental problems in representation theory is the
calculation of the simple characters of a given group. This problem often turns out to be difficult and there is  an
abundance of situations in which a solution is out of reach.   In the
case of  algebraic groups over 
fields of positive characteristic we have a partial, but not yet a full answer.  

In 1979 George Lusztig conjectured a formula for the simple characters
of a reductive algebraic group defined over a field of  characteristic bigger than the associated  Coxeter number, cf.~\cite{MR604598}. Lusztig outlined in 1990 a program that led, in a
combined effort of several authors, to a proof of the conjecture for
almost all characteristics. This means that for a given root system $R$  there exists a number $N=N(R)$ such that
the conjecture holds for all algebraic groups associated to the root system $R$ if the underlying field is of characteristic bigger than $N$. This number, however, is
unknown in all but low rank cases.

One of the essential steps in Lusztig's program was the construction
of a functor between the category of 
intersection cohomology sheaves with complex coefficients on an affine
flag manifold and the category of representations of a quantum group
(this combines results of  Kashiwara--Tanisaki, cf.~\cite{MR1317626}, and Kazhdan--Lusztig, cf.~\cite{KLequiv}). This  led to
a proof of the quantum (i.e.~ characteristic $0$) analog of the
conjecture. Andersen, Jantzen and Soergel then showed that the
characteristic zero case implies the characteristic $p$ case for
almost all $p$ (cf.~\cite{MR1272539}). 

One of the principal functors utilized in Lusztig's program was the affine version of
the Beilinson--Bernstein localization functor. It amounts to realizing
an affine Kac--Moody algebra inside the space of global differential
operators on an affine flag manifold. A characteristic $p$ version of
this functor is a fundamental ingredient in Bezrukavnikov's program
for modular representation theory (cf.~\cite{BMR08}), and recently Frenkel and Gaitsgory
used the Beilinson--Bernstein localization idea in order to study the
critical level representations of an affine Kac--Moody algebra (cf.~\cite{FG06}).

There is, however, an alternative approach that links  the geometry of an algebraic variety to representation
theory. It was originally developed in the case of finite dimensional
complex simple Lie algebras by Soergel (cf.~\cite{Soe90}). The
idea was to give a ``combinatorial description'' of both the
topological and the representation theoretic categories in terms of
the underlying root system using Jantzen's translation functors. This
approach gives a new proof of the Kazhdan--Lusztig
conjecture, but it is also important in its own right: when taken together with
the Beilinson--Bernstein localization it establishes the celebrated
Koszul duality for simple finite dimensional complex Lie algebras
(cf.~\cite{Soe90,BGS}).

In this paper we develop the combinatorial approach for quantum and modular
representations. We relate a certain category of sheaves of $k$-vector
spaces on an affine flag manifold to representations of the $k$-Lie
algebra or the quantum group associated to Langlands dual root datum (the occurence of
Langlands duality is typical for this type of approach).  As a
corollary we obtain Lusztig's conjecture for quantum groups and for modular representations for large enough
characteristics.

The main tool that we use is the theory of {\em sheaves on moment
  graphs}, which originally appeared in the work on the localization
theorem for equivariant sheaves on topological spaces by Goresky,
Kottwitz and MacPherson (cf.~\cite{MR1489894}) and Braden and MacPherson (cf.~\cite{MR1871967}). In
particular, we state a conjecture in terms of moment graphs that implies Lusztig's quantum and modular conjectures
 for all relevant characteristics. 

Although there is
no general proof of this moment graph conjecture yet, some important
instances are known: The {\em smooth locus} of a moment graph is
determined in \cite{fiebig:math0607501}, which yields the multiplicity one case of
Lusztig's conjecture in full generality. Moreover, by developing a
Lefschetz theory on a moment graph we obtain in \cite{fiebig:math0811.1674} an upper bound
on the exceptional primes, i.e.~an  upper
bound for the number $N$ referred to above. Although this bound is
huge (in particular, much bigger than the Coxeter number), it can
be calculated by an explicit formula in terms of the underlying root
system.

In the remainder of this introduction we explain our approach and the results in more detail.

\subsection{The basic data}

Let $G=G_\DC$ be a connected,  simply connected complex
algebraic group. We fix a Borel subgroup $B\subset G$ and a maximal
torus $T\subset B$. Let $X=\Hom(T,\DC^\times)$ be the character
lattice of $T$ and $R^+\subset R\subset X$ the sets of roots of $B$
and of $G$. We denote by $X^\vee=\Hom(\DC^\times,T)$ the cocharacter lattice,
and by $R^\vee\subset X^\vee$ the dual root system. We let $\CW\subset\GL(X^\vee)$ be the
Weyl group, $\hCW=\CW\ltimes \DZ R^\vee\subset\Aff(X^\vee)$ the affine Weyl group and $h$ 
the Coxeter number of our data, i.e.~the height of the highest root $+1$. By $\hCS\subset\hCW$ we denote the set of
simple affine reflections.

We fix an algebraically closed field $k$  of characteristic
$p>h$ and let $G^\vee=G^\vee_k$ be the connnected simple algebraic
group over $k$ that is the Langlands dual of $G$. We let  $T^\vee\subset
G^\vee$ be a maximal torus and identify
$\Hom(T^\vee,\DC^\times)$ with the cocharacter lattice $X^\vee$. We let $\fg^\vee$ and $\fh^\vee$ be the Lie algebras of $G^\vee$ and $T^\vee$. 

\subsection{Multiplicities of modular representations}

It is known how to calculate the characters of the simple rational
representations of the reductive algebraic groups with root system $R^\vee$  from  the characters of simple objects in
a certain category $\CC$ of restricted representations of $\fg^\vee$ that
carry an additional action of $T^\vee$ (i.e., an
$X^\vee$-grading), such that  $\fh^\vee\subset\fg^\vee$ acts via the differential of the $T^\vee$-action. 
The simple objects in $\CC$ are parametrized by their highest
weights, so for $\lambda\in X^\vee$ denote by $L(\lambda)$ the
corresponding simple object. It can be constructed as
the unique simple quotient of the {\em standard} (or {\em baby Verma})
module $Z(\lambda)$ with highest weight $\lambda$. 

The characters of
the standard modules are easy to compute, and in order to get the
characters of the simple modules it is enough to know the number $[Z(\lambda):L(\mu)]$ of occurrences of $L(\mu)$ as a
subquotient in a Jordan-H\"older filtration of $Z(\lambda)$ for all
$\lambda,\mu\in X^\vee$.
From the linkage and translation principles it follows that it is
sufficient to consider  the simple and standard modules in
the principal block of $\CC$, i.e.~those objects which correspond to
weights of the form  $\lambda=x\cdot_p 0$ and $\mu=y\cdot_p 0$ for $x,y\in\hCW$ (here
``$\cdot_p$'' denotes a shifted and elongated action of $\hCW$ on $X^\vee$).

We  identify the set $\hCW$ with the set $\SA$ of alcoves for the affine action of $\hCW$ on $X^\vee\otimes_\DZ \DR$ by fixing a base alcove $A_e$ (so that $w$ corresponds to the alcove $w.A_e$).
In \cite{MR1445511} a polynomial $p_{A,B}$ is associated to each pair of alcoves (following  \cite{MR591724}). For $x,y\in\hCW$ we write $p_{x,y}$ for $p_{x.A_e,y.A_e}$. 
Let us denote  by $w_0\in\CW$ the longest element in the finite
Weyl group. 

\begin{MConj} Suppose that $\ch k>h$. For $x,y\in\hCW$ we have 
$$
[Z(x\cdot_p 0):L(y\cdot_p 0)]=p_{w_0x,w_0y}(1).
$$
\end{MConj}

The above conjecture is known as the {\em generic Lusztig
conjecture}. In \cite{Fie07} we prove the (well-known) fact that it is
equivalent to Lusztig's original conjecture on the characters of
simple rational representations (cf.~\cite{MR604598}). 

Using an inherent symmetry one can show that it is  enough to verify this conjecture in the case that  $y$
is an element in the ``anti-fundamental box'', i.e.~in the case that $y$ satisfies 
 $-p<\langle
  \alpha,y\cdot_p 0\rangle <0$ for each simple root $\alpha$. Denote
  by $\hCW^{res,-}\subset\hCW$ the set of such elements. 

\subsection{The quantum analog} Let us just quickly state, without
giving details, that there is a characteristic $0$ analog of the representation theory
described above (cf.~ the overview articles \cite{MR1403973} and \cite{MR1357204}). Lusztig associated to the root system $R$, its set
of positive roots $R^+\subset R$, an integer $l$ prime to all entries
of the Cartan matrix, and a primitive $l$-th root of unity $\zeta$ the {\em
  restricted quantum group $\bf u$}, which is a finite dimensional
$\DZ R$-graded algebra over $k=\DQ(\zeta)$. For each $\lambda\in X$
one constructs a standard object  $Z_q(\lambda)$ and a simple object
$L_q(\lambda)$ in the category of $\DZ R$-graded $\bf u$-modules and one can state a conjecture in complete analogy to the conjecture above. Let us refer to these two conjectures as {\em Lusztig's
 modular conjecture} and {\em Lusztig's  quantum conjecture}.

\subsection{Results}
The main application of the relations that we construct between sheaves on affine Schubert varieties, sheaves on the underlying moment graph and representation theory is a proof  of the following theorem. 

\begin{theorem}\label{theorem-MainTh}
\begin{enumerate} 
\item Lusztig's quantum conjecture holds.
\item The multiplicity one case of Lusztig's modular conjecture holds, i.e.~ for all fields $k$ with $\ch k=p>h$ we have $[Z(x\cdot_p 0):L(y\cdot_p 0)]=1$ if and only if $p_{w_0x,w_0y}(1)=1$. 
\item There is an explicit number $U(\hw_0)$, defined in
  terms of the root system, such that Lusztig's modular conjecture holds for
  all fields $k$ with  $\ch k>U(\hw_0)$. 
\end{enumerate}
\end{theorem}

Part (1) of the above theorem is already proven by combining results of Kashiwara--Tanisaki and Kazhdan--Lusztig which form major steps in Lusztig's program. Our proof is independent and uses only the combinatorial description of the quantum  category given by Andersen, Jantzen and Soergel. 
Part (2) is an application of the result on the ``smooth locus'' of a moment graph given in \cite{fiebig:math0607501}. Part (3) finally uses the main result in  \cite{fiebig:math0811.1674}, where we develop a Lefschetz theory on a moment graph in order to calculate  the number $U(\hw_0)$.

In the remainder of the introduction we want to describe the main
ideas in the proofs of the above results. 

\subsection{Deformed representation theory}
Each simple object $L(\mu)$ admits a projective cover
$P(\mu)$ in $\CC$. Moreover, each $P(\mu)$ has a finite
filtration with subquotients that are isomorphic to various
$Z(\lambda)$. The corresponding multiplicity, denoted by
$(P(\mu):Z(\lambda))$, is independent of the particular
filtration, and the following Brauer-type reciprocity formula was
shown by Humphreys (cf.~\cite{MR0453884}):
$$
[Z(\lambda):L(\mu)]=(P(\mu):Z(\lambda)).
$$

Let $S=S(\fh^\vee)$ be the symmetric algebra of $\fh^\vee$, and
denote by $\tilde S$ the completion of $S$ at the maximal ideal
generated by $\fh^\vee$. In
order to compute the above multiplicities  Andersen, Jantzen and
Soergel construct in \cite{MR1272539}  deformed
versions $\tP(\mu)$ and $\tZ(\lambda)$ of $P(\mu)$ and
$Z(\lambda)$, that appear in a deformed
 version $\tCC$ of
$\CC$. Again, each
$\tP(\mu)$ has a finite filtration with subquotients isomorphic to
various $\tZ(\lambda)$, and for the multiplicities we have 
$$
(\tP(\mu):\tZ(\lambda))=(P(\mu):Z(\lambda)).
$$

The objects $\tP(y\cdot_p 0)$ can be constructed using {\em translation
  functors}: to each simple affine reflection $s\in\hCS$ one associates a
translation functor $\theta^s$ on the principal block of
$\tCC$. We show that for each $y\in\hCW^{res,-}$ the object $\tP(y\cdot_p 0)$ appears as a direct
summand in $\theta^t\circ\dots\circ\theta^s \tZ(0)$ for a suitable sequence
$s,\cdots,t\in\hCS$. This serves as a  motivation to define $\CR$ as the smallest full subcategory of $\tCC$ that is stable under taking direct sums and direct summands, contains
$\tZ(0)$ and with each object $M$ and each $s\in\hCS$ 
 the object $\theta^s M$. 

Let us just mention that in the quantum case  we find analogous structures and results and, in particular, we have a category $\CR$ over the field $k=\DQ(\zeta)$ that encodes the structure of the quantum category. 

In the following we relate the category $\CR$ to equivariant sheaves of $k$-vector spaces on a complex affine flag variety.

\subsection{Equivariant sheaves}

Denote by $G((t))$ the loop group associated to $G$, and by $I\subset G((t))$
the Iwahori subgroup corresponding to $B$. The quotient $\hFl=G((t))/I$
carries a natural complex ind-variety structure and is called the
affine flag variety. Denote by $\hT=T\times\DC^\times$ the extended
torus and let it act on $G((t))$ such that the first factor acts by
left multiplication and the second by rotating the loops. Then $\hT$
also acts on $\hFl$. 

Now let $k$ be a field of characteristic $\ne 2$ and let
$\dercat_{\hT}(\hFl,k)$ be the bounded equivariant derived category of
sheaves of $k$-vector spaces on $\hFl$.  We define the category
$\CI$  of {\em special equivariant
sheaves} on $\hFl$ following \cite{MR1784005}. It is the smallest full
subcategory of $\dercat_{\hT}(\hFl,k)$ that  is stable under taking direct sums and direct summands and under shifting, contains the constant equivariant sheaf $F_{e}$ of rank
one on the base point of $\hFl$, and with each $F$ and each $s\in\hCS$
the sheaf $\pi_s^\ast\pi_{s\ast} F$ in $\dercat_{\hT}(\hFl,k)$ (here
$\pi_s\colon\hFl\to\hFl_s$ is the canonical map onto the partial affine
flag variety corresponding to $s$). In \cite{FW} we give another, more intrinsic definition of the category $\CI$ as the category of certain equivariant {\em parity sheaves} on $\hFl$.

The set of $\hT$-fixed points in $\hFl$ is discrete and can  canonically be identified
with the set $\hCW$. So for $x\in\hCW$ we denote by $i_x\colon
\{\pt\}\to\hFl$ the corresponding inclusion. The following result is
the most important step in the proof of Theorem \ref{theorem-MainTh}.
\begin{theorem}\label{Phi} Suppose $k=\DQ(\zeta)$ or  that $k$ is of characteristic  $p>h$.
\begin{enumerate}
\item
There exists an additive functor 
$$
\Phi\colon\CI\to\CR
$$
such that
$\Phi(F_e)\cong \tZ(0)$ and such that 
$\Phi(\pi_s^\ast\pi_{s\ast} F)\cong\theta^s\Phi(F)$ for all
$s\in\hCS$ and $F\in\CI$.
\item Each $\Phi(F)$ has a filtration by deformed standard modules, and the ranks of the local equivariant hypercohomologies of $F$ on  $\hT$-fixed points yield multiplicities for $\Phi(F)$, i.e.~we have 
$$
\rk\, \Hyp_{\hT}^\ast(i_x^\ast F)=\left(\Phi(F): \tZ(x\cdot_p 0)\right)
$$
for all $x\in\hCW$.
\end{enumerate}
\end{theorem}
(Here, ``$\rk$'' refers to the rank of a free 
$H_{\hT}^\ast(\pt,k)$-module.)
Now we explain how to derive part (1) of Theorem \ref{theorem-MainTh} from the statement  above.

\subsection{Intersection cohomology sheaves}

We identify the set of $I$-orbits in $\hFl$  with
the affine Weyl group, so we denote by $\CO_y\subset\hFl$ the orbit
corresponding to $y\in\hCW$ and let
$\IC_{\hT,y}\in\dercat_{\hT}(\hFl,k)$ be the $\hT$-equivariant intersection cohomology complex on the Schubert variety $\ol{\CO_y}$ with coefficients in $k$.

If $k$ is a field of characteristic zero, then the decomposition
theorem and some orbit combinatorics show that $\CI$ is the
category of direct sums of shifted $\hT$-equivariant intersection
cohomology sheaves on Schubert varieties in $\hFl$.  The
equivariant analogue of a
result of Kazhdan and
Lusztig in \cite{MR560412} is that we have  
$\rk\, \Hyp_{\hT}^\ast(i^\ast_x\IC_{\hT,y})=h_{x,y}(1)$
for all $x,y\in\hCW$, where $h_{x,y}$ is an
affine Kazhdan-Lusztig polynomial (cf.~Theorem \ref{self-dual elts}).

Now fix an arbitrary sequence  $s,\dots,t \in\hCS$. The object
$\pi_t^\ast\pi_{t\ast}\cdots\pi_s^\ast\pi_{s\ast} F_{e}$ can be defined over
$\DZ$ and  it decomposes in almost all characteristics as it
does in characteristic zero. We deduce that
$\pi_t^\ast\pi_{t\ast}\cdots\pi_s^\ast\pi_{s\ast} F_{e}$ is a direct
sum of shifted intersection cohomology sheaves on Schubert varieties
for all fields $k$ of  big enough characteristic (the notion of big
enough now depends on the chosen sequence). 

Denote by $\hw_0\in\hCW$ the largest element in $\hCW^{res,-}$ (with
respect to the Bruhat order), and
set $\hCW^\circ=\{w\in\hCW\mid w\le \hw_0\}$.  Define
$\CI^\circ\subset \CI$ as the full subcategory
that consists of direct sums of shifted direct summands of the objects
$\pi_t^\ast\pi_{t\ast}\cdots\pi_s^\ast\pi_{s\ast} F_{e}$, where
$s\cdots t$ is a reduced expression of an element  in $\hCW^\circ$. Since for
the construction of $\CI^\circ$ we have to consider only
finitely many sequences, we can deduce the following:
 
\begin{theorem}\label{ICpos} Suppose that $k$ is a field of characteristic $0$ or $p\gg 0$. 
\begin{enumerate}
\item  $\CI^\circ$ is the full subcategory of $\dercat_{\hT}(\hFl,k)$ that consists of objects isomorphic to a direct sum of shifted intersection cohomology sheaves $\IC_{\hT,y}$ with $y\in\hCW^\circ$. 
\item If $y\in\hCW^\circ$ and $y=s\cdots t$ is a reduced expression, then $\IC_{\hT,y}$ occurs as the unique indecomposable direct summand in $\pi_t^\ast\pi_{t\ast}\cdots\pi_s^\ast\pi_{s\ast} F_{e}$ that is supported on $\ol{\CO_{y}}$. 
\item For $y\in\hCW^\circ$ and $x\in\hCW$ we have $\rk\, \Hyp_{\hT}^\ast(i_x^\ast \IC_{\hT,y})=h_{x,y}(1)$. 
\end{enumerate}
\end{theorem}
Again one hopes that each prime above the Coxeter number is big
enough for the statements of the theorem to hold. 

In
Section \ref{sec-mod} we show that we have  $\Phi(\IC_{\hT,y})\cong
\tP(y\cdot_p 0)$ under the assumption of Theorem
\ref{ICpos}. Moreover, the following mysterious relation between the
Kazhdan--Lusztig polynomials $h$ and the periodic polynomial $p$
holds. Although $h_{x,y}\ne p_{w_0x,w_0y}$ in general, we have
$h_{x,y}(1)= p_{w_0x,w_0y}(1)$ for all $y\in\hCW^{res,-}$ and all $x\in\hCW$ (this phenomenon is explained by the fact that the functor $\Psi$ defined below does not preserve gradings). So part (1) of Theorem \ref{theorem-MainTh}  is a
consequence of Theorem \ref{Phi} and Theorem \ref{ICpos}. In addition, we can already deduce Lusztig's modular conjecture for almost all primes. Let us now
discuss  the main steps of the definition $\Phi$.

\subsection{Categorifications of combinatorial data}

 Let $\bH$ be the affine Hecke
algebra associated to $R$  and $\bM$ its
periodic module (cf.~\cite{MR1445511}).  Denote by $A_e$ the basis element in $\bM$
corresponding to the fundamental alcove, and consider the map 
$\psi\colon \bH\to\bM$ that is obtained by letting $\bH$ act on $A_e$. 

For each field $k$ with characteristic different from $2$ and different from $3$ if $R$ is of type $G_2$ we construct categories $\CH$ and $\CM$ and a functor
$\Psi\colon\CH\to\CM$ together with character maps
$h_\CH\colon\CH\dashrightarrow \bH$ and $h_\CM\colon\CM\dashrightarrow
\bM$ (i.e.~maps to the Grothendieck groups), such that the diagram 

\centerline{
\xymatrix{
\CH \ar@{-->}[d]_-{h_{\CH}}\ar[r]^-{\Psi} & \CM\ar@{-->}[d]^-{h_{\CM}}\\
 \bH \ar[r]^-{\psi} & \bM
}
} 
\noindent
commutes.

The category $\CH$ is a full subcategory of the category of modules
over the equivariant cohomology of $\hFl$ (with coefficients in $k$), and hypercohomology
$\Hyp^\ast_{\hT}$ yields a functor from $\CI$ to $\CH$. The
category $\CM$ appears in the work of Andersen, Jantzen and
Soergel. There an equivalence $\DV$ between $\CR$  and an $\tilde
S$-linear version $\tCM$ of $\CM$ is constructed in the case that $k$ is an algebraically closed field
of characteristic $p>h$. Our functor is then the composition
$$
\Phi\colon\CI\stackrel{\Hyp^\ast_{\hT}}\longrightarrow \CH
\stackrel{\Psi}\longrightarrow \CM\xrightarrow{\cdot\otimes_S\tilde S}\tCM
\stackrel{\DV^{-1}}\longrightarrow \CR.
$$ 
We now relate  $\CH$ to yet another category.

\subsection{Sheaves on moment graphs}
In the papers \cite{fiebig:math0607501} and
\cite{fiebig:math0811.1674} we study the theory of sheaves on moment
graphs, which gives the following alternative description of the indecomposable
objects in $\CH$ (cf.~\cite{Fie05b}).

To the action of $\hT$ on $\hFl$ one can associate a  {\em moment graph} 
over the lattice $\hX=\Hom(\hT,\DC^\times)$
that plays a prominent role in equivariant algebraic topology
(cf.~\cite{MR1489894,MR1871967}). Recall that an (ordered) moment graph over a lattice $Y$ is a
graph together with a partial order on its set of vertices and a map that associates to each edge a non-zero
element in $Y$. 

In our situation such a moment graph $\hCG$ is
obtained as follows. Its vertices are the 
$\hT$-fixed point in $\hFl$  and its edges are the one-dimensional
$\hT$-orbits (the closure of a one-dimensional orbit  contains exactly two fixed
points). Each edge is labelled by the positive affine root
associated to the rotation action of $\hT$ on the corresponding
orbit. The  partial order is induced by the closure
relation on the Iwahori-orbits in $\hFl$ (each such orbit
contains a unique $\hT$-fixed point). 
For the applications in this
article only the full subgraph $\hCG^\circ\subset\hCG$ that consists of
the vertices corresponding to elements in $\hCW^\circ$ plays a
role. It has the advantage of being finite.

Suppose that $k$ is a field of characteristic $\ne 2$ and set $\hV_k:=\hX\otimes_\DZ k$.
To $\hCG^\circ$
one can associate a category of  $k$-sheaves on
$\hCG^\circ$, and to each vertex $y$ 
one assigns the  {\em intersection} (or  {\em
  canonical}) sheaf $\SB_y$. Denote by
$\SB_y^x$ its stalk at the vertex $x$. It is a graded free module over the symmetric algebra $\hS:=S(\hV_k)$.  

Now the restriction on the characteristic of $k$ is as follows. A
large part of the theory of $k$-sheaves on a moment graph  $\CG$ behaves well 
only in case that the pair $(\CG, k)$ satisfies the  GKM-property, which means that the
labels on two different edges meeting at a common vertex
are linearly independent over $k$ (cf.~\cite{Fie05a}). We will show that for $\hCG^\circ$
this is the case if  $\ch k$ is at least the Coxeter
number. 

We define a certain subcategory $\CH^\circ$ of $\CH$ that
contains all relevant objects. From the results in   \cite{Fie05b} we deduce that for $\ch k\ge h$ the category $\CH^\circ$ can be
interpreted as the category of direct sums of the spaces of global sections of the intersection sheaves
$\SB_y$ on
$\hCG^\circ$. This, together with the functor $\Psi$ and the Andersen--Jantzen--Soergel equivalence, gives us a way to link the moment graph theory to modular representation theory, and in Section \ref{subsec-MomGraconj} we state a conjecture that implies Lusztig's conjecture. The main results in \cite{fiebig:math0607501} and
\cite{fiebig:math0811.1674} prove certain instances of this conjecture and yield part (2) and part (3) of  Theorem \ref{theorem-MainTh}. 

\subsection{Contents} In Section \ref{sec-cat} we discuss the
affinization of a root system and define an algebra $\hCZ$
over the subring $\DZ^\prime=\DZ[1/2]$ of $\DQ$. For each simple affine 
reflection we construct a translation functor
on the category of $\hCZ$-modules, and we define the category $\CH$ of {\em
  special $\hCZ$-modules} as the full subcategory generated from a unit
object by repeatedly applying translation functors. 

In Section \ref{sec-Loc} we discuss various localizations of special
modules. This is used in Section \ref{sec-Fil} to define a functorial
filtration on the category $\CH$ for each partial order on the set $\hCW$ that is compatible with the involutions given by the simple affine reflections. To such a filtration corresponds a character map from $\CH$ to the free $\DZ[v,v^{-1}]$-module with basis $\hCW$. It turns out that the Bruhat order naturally yields characters
in the affine Hecke algebra, whereas the generic Bruhat order yields
characters in its periodic module.

In Section \ref{sec-AJS} we recall the definition of the category $\CM$ introduced by
Andersen, Jantzen and Soergel. In analogy to the
definition of $\CH$, also $\CM$ is generated inside a certain category
$\CK$ by repeatedly applying translation functors to a unit
object. This allows us to define the functor $\Psi\colon
\CH\to\CM$. Although the functor itself is easy to define, it is quite
tedious to check that it intertwines the translation functors.

The definition of translation functors on $\CK$ that we use is
different from the definition of Andersen, Jantzen and Soergel (it is
a version ``without constants''). In Section \ref{sec-AJStrans} we
show that both definitions lead to equivalent categories.

In Section \ref{sec-She} we define the category $\CI$ of special
sheaves on the affine flag variety and show that the hypercohomology
functor can be considered as a functor from $\CI$ to $\CH$. We use the
decomposition theorem to deduce multiplicity formulas for the objects in
$\CH$ for almost all characteristics.

In Section \ref{sec-mod} we recall the main result of \cite{MR1272539},
which relates the category $\CM$ to the category $\CR$ of
representations
of a Lie algebra or a quantum group. The multiplicity formulas that we
gained in Section \ref{sec-She} for $\CH$, together with the functor
$\Psi\colon \CH\to \CM$, give multiplicity formulas for objects in
$\CR$.

We interpret our main category $\CH$ as a category of sheaves on
moment graphs in Section \ref{sec-momgra}, and state a conjecture on the multiplicities of stalks of the intersection sheaves on the graph. In  \cite{fiebig:math0607501} and
\cite{fiebig:math0811.1674} two instances of this conjecture are proven. In a last step we apply this to  Lusztig's conjecture.

\subsection{Acknowledgements} I thank Henning Haahr Andersen, Jens
Carsten Jantzen, Masaharu Kaneda, Olaf Schnürer, Wolfgang Soergel and Geordie Williamson for helpful remarks on various versions of this paper.

\section{A category associated to a root system} \label{sec-cat}

Let $V$ be a finite dimensional $\DQ$-vector space, let $V^\ast=\Hom_\DQ(V,\DQ)$ be its dual space and denote by $\langle\cdot,\cdot\rangle\colon V\times V^\ast\to\DQ$ the canonical pairing. Let $R\subset V$ be a reduced and irreducible root system. For a root $\alpha\in R$ denote by $\alpha^\vee\in V^\ast$ its coroot. Let $R^\vee=\{\alpha^\vee\mid \alpha\in R\}$ be the coroot system. 

Denote by $s_\alpha\in \GL(V^\ast)$ the reflection associated to $\alpha\in R$, i.e.~the linear map given by $s_\alpha(v)=v-\langle\alpha,v\rangle\alpha^\vee$. Let $\CW\subset\GL(V^\ast)$ be the Weyl group, i.e.~the subgroup generated by the $s_\alpha$ with $\alpha\in R$.

\subsection{The affine Weyl group}

 For $\alpha\in R$ and $n\in\DZ$ define the affine hyperplane
$$
H_{\alpha,n}:=\{v\in V^\ast\mid \langle\alpha,v\rangle=n\}.
$$
We denote by $s_{\alpha,n}$ the reflection at $H_{\alpha,n}$, i.e.~the affine transformation on $V^\ast$ that maps $v\in V^\ast$ to
$$
s_{\alpha,n}(v):=v-(\langle\alpha,v\rangle-n)\alpha^\vee.
$$
The {\em affine Weyl group} $\hCW\subset\Aff(V^\ast)$ is the group of affine transformations on $V^\ast$ that is generated by the set $\hCT$ of all reflections $s_{\alpha,n}$ with $\alpha\in R$ and  $n\in\DZ$.

Since $s_\alpha=s_{\alpha,0}$, the finite Weyl group $\CW$ appears as a subgroup of $\hCW$. Let $\DZ R^\vee\subset V^\ast$ be the coroot lattice. To $x\in \DZ R^\vee$ we associate the translation $t_x\colon V^\ast\to V^\ast$, $v\mapsto v+x$. For $\alpha^\vee\in R^\vee$ we have $t_{\alpha^\vee}=s_{\alpha,1}\circ s_{\alpha,0}$, hence $\hCW$ contains the abelian group $\DZ R^\vee$. We even have $\hCW=\CW\ltimes \DZ R^\vee$ and we denote by  $\ol{\cdot}\colon \hCW\to \CW$, $w\mapsto \ol w$,  the corresponding quotient map. Then $\ol{s_{\alpha,n}}=s_\alpha$.

\subsection{A linearization}

Set $\hV:=V\oplus\DQ$, so
$\hV^\ast=V^\ast\oplus\DQ$, and define for $\alpha\in R$ and
$n\in\DZ$ a {\em linear} action of $s_{\alpha,n}$ on  $\hV^\ast$ by 
$$
s_{\alpha,n}(v,\mu) := (v-(\langle\alpha,v\rangle-\mu  n)\alpha^\vee,\mu).
$$
This extends to a linear action of $\hCW$ on $\hV^\ast$ which leaves
the level spaces $V^\ast_\kappa=\{(v,\kappa)\mid v\in V^\ast\}\cong
V^\ast$ for $\kappa\in\DQ$ stable. On $V^\ast_1$ we recover the affine
action of $\hCW$ and on $V^\ast_0$ the affine Weyl group $\hCW$ acts via the ordinary action of its finite quotient $\CW$.

The hyperplane in $\hV^\ast$ fixed by $s_{\alpha,n}$ is  
$$
\hH_{\alpha,n}:=\{(v,\mu)\in\hV^\ast\mid \langle\alpha,v\rangle=\mu n\}.
$$ 
Set $\delta:=(0,1)\in \hV=V\oplus\DQ$. Then 
$$
\alpha_n:=(\alpha,0)-n\delta\in\hV
$$
is an equation of $\hH_{\alpha,n}$. We call $\alpha_n$ a {\em (real) affine root}, and we set 
$$
\hR:=\{\alpha_n\mid \alpha\in R,n\in\DZ\}=R\times \DZ\delta\subset \hV.
$$

Let us choose a system $R^+\subset R$ of positive (finite) roots. 
The corresponding set of positive affine roots is
$$
\hR^+:=\{\alpha+n\delta\mid \alpha\in R, n>0\}\cup
\{\alpha\mid \alpha\in R^+\}.
$$
Then $\hR=\hR^+\dot\cup -\hR^+$ and for any reflection $t\in\hCT$ there is a unique positive affine  root $\alpha_t=\alpha-n\delta$ such that $t=s_{\alpha,n}$.

Now consider the dual action of $\hCW$ on $\hV$ that is given by $w.\phi=\phi\circ w^{-1}$ for $w\in\hCW$ and $\phi\in \hV$. More explicitely, it is given by 
\begin{eqnarray*}
w(0,\nu) & = & (0,\nu), \\
s_{\alpha,n}(\lambda,0) & = & (s_\alpha(\lambda), n\langle\lambda,\alpha^\vee\rangle).
\end{eqnarray*}
Denote by $\ol\cdot\colon\hV\to V$ the map $(\lambda,\nu)\mapsto\lambda$. Then the following is immediate.
\begin{lemma}\label{Wquot}
For each $w\in\hCW$ and $x\in \hV$ we have $\ol{w(x)}=\ol w(\ol x)$.
\end{lemma}

Recall that the set of positive roots $R^+$ determines a set of simple roots $\Delta\subset R^+$ and a set of simple reflections $\CS\subset\CW$. The corresponding set of simple affine reflections is 
$$
\hCS:=\CS\cup \{s_{\gamma,1}\}\subset\hCW,
$$
where $\gamma\in R^+$ is the highest root. Then $(\hCW,\hCS)$ is a Coxeter system and we denote by $l\colon \hCW\to \DN$ the associated length function.

\subsection{The associated moment graph}
 Let $Y\cong\DZ^r$ be a lattice. An (unordered) moment graph $\CG$ over $Y$ is given by a graph $(\CV,\CE)$ with vertices $\CV$ and edges $\CE$ and a map $\alpha\colon \CE\to Y\setminus\{0\}$ which is called the labeling. 

Let $X:=\{\lambda\in V\mid \langle \lambda,\alpha^\vee\rangle\in\DZ\quad\forall \alpha\in R\}$ be the weight lattice and 
$\hX:=X\oplus\DZ\delta\subset\hV$ 
the {\em affine weight lattice}. The latter  contains the affine root lattice $\DZ \hR$. The {\em affine Bruhat graph} $\hCG=\hCG_R$ associated to $R$ 
is given as follows. Its underlying lattice is $\hX$. The set of vertices is the affine Weyl group $\hCW$ and $x,y\in\hCW$ are connected by an edge if there is a reflection $t\in\hCT$ with $tx=y$. This edge is labelled by the positive affine root $\alpha_t\in \hX$ corresponding to $t$.

For a parabolic subgroup $\hCW_I\subset \hCW$ that is given by a subset $I$ of $\hCS$ we obtain another graph $\hCG^I$ as follows. Its set of vertices is $\hCW^I:=\hCW/ \hCW_I$ and $x,y\in\hCW^{I}$ are connected by an edge if there exists $t\in \hCT$ such that $tx=y$. This edge is then labelled by $\alpha_t$.

\subsection{The structure algebra}\label{subsec-Struc}  We denote by $\DZ^\prime$ the subring $\DZ[2^{-1}]$ of $\DQ$. Let
 $S=S_{\DZ^\prime}=S(X\otimes_\DZ\DZ^\prime)$ and $\hS=\hS_{\DZ^\prime}=S(\hX\otimes_\DZ\DZ^\prime)$ be the symmetric algebras over the free $\DZ^\prime$-modules associated to the lattices $X$ and $\hX$.  Both $S$ and $\hS$ are $\DZ$-graded algebras. The grading we
choose is determined by setting $X\otimes_\DZ \DZ^\prime$ and $\hX\otimes_\DZ\DZ^\prime$ in degree $2$ (this
will become important only when we  consider graded multiplicities).
In this article many objects are considered to be $\DZ$-graded. For $n\in\DZ$ we denote by $(\cdot)\langle n\rangle$ the functor that shifts gradings by $n$, i.e.~for a  $\DZ$-graded module $M=\bigoplus_{i\in\DZ} M_i$ we have $M\langle n\rangle_i=M_{i+n}$.

The {\em affine structure algebra} associated to the root system $R$ is 
$$
\hCZ=\hCZ_{\DZ^\prime}:=\left\{(z_w)\in\prod_{w\in\hCW} \hS\left|\begin{matrix} z_w\equiv z_{s_{\alpha,n}w}\mod\alpha_n \\ 
\text{ for all $\alpha\in R^+$, $n\in\DZ$, $w\in\hCW$}
\end{matrix}\right.\right\}.
$$ 
The infinite product $\prod_{w\in\hCW} \hS$ appearing in the definition above should be understood as the product in the category of graded $\hS$-modules, i.e.~the product is taken degree-wise. As the defining relations are homogeneous, $\hCZ$ is a $\DZ$-graded algebra. It is an $\hS$-algebra via the diagonal inclusion  $\hS\subset\hCZ$. 

Let $k$ be a ring in which $2$ is invertible. For any $\DZ^\prime$-module $M$ we set $M_k:=M\otimes_{\DZ^\prime}k$. In particular, we have  $S_k=S(X\otimes_{\DZ} k)$ and $\hS_k=S(\hX\otimes_{\DZ}k)$, and the structure algebra over $k$ is
$$
\hCZ_k=\hCZ\otimes_{\DZ^\prime} k.
$$
Note that this does not necessarily coincide with the subalgebra of $\prod_{w\in\hCW} \hS_k$ defined in analogy with $\hCZ$ by congruence relations. 

Let us construct some elements in $\hCZ$. For $\lambda\in \hX$ and $x\in\hCW$ define $c(\lambda)_x:=x(\lambda)$. As $s_{\alpha,n}(\mu)\equiv \mu\mod\alpha_n$ for all $\alpha,n,\mu$, we have  that $c(\lambda):=(c(\lambda)_x)_{x\in\hCW}$ is contained in $\hCZ$. This yields a homomorphism 
$$
c\colon \hX\to\hCZ
$$ 
of abelian groups. Note that the image of $c$ is contained in degree $2$. 


For $\beta\in R^+$  let $\hCW^\beta\subset\hCW$ be the subgroup generated by all reflections $s_{\beta,n}$ with $n\in\DZ$.  The next lemma will be used  in the proof of Theorem \ref{MTheo}.
\begin{lemma}\label{dinZ} For each $\beta\in R^+$ and $w\in\hCW$ there is an element $z=(z_x)_{x\in\hCW}\in\hCZ$ of degree $2$ with the following properties:
\begin{enumerate}
\item  $z_w=0$,
\item  for all $x\in\hCW^\beta$ we have
$z_{xw}\in \DZ^\prime\delta$, if $l(x)$ is even, and $z_{xw}\in\beta+\DZ^\prime\delta$, if $l(x)$ is odd.
\end{enumerate}
\end{lemma}
\begin{proof} For all $x\in\hCW$ set $\tilde z_{x}:=\beta-xw^{-1}(\beta)\in \hX\subset \hS$. Then $\tilde z:=(\tilde z_x)_{x\in\hCW}=\beta\cdot 1_{\hCZ}+c(w^{-1}(\beta))$ is an element in $\hCZ$. Obviously, $\tilde z_w=0$. For all $x\in\hCW^\beta$ we calculate (using Lemma \ref{Wquot}):
\begin{eqnarray*}
\tilde z_{xw} & = & \beta-x(\beta) \\
& \equiv & \beta-\ol{x}(\beta)\mod\delta \\
& \equiv & \beta-(-1)^{l(x)}\beta \mod\delta
\end{eqnarray*}
and hence $2^{-1}\tilde z$ serves our purpose. 
\end{proof}

\subsection{Invariants} 
Choose  $t\in\hCT$ and consider the involution $w\mapsto wt$ on the set $\hCW$.  We denote by $\sigma_t$  the algebra involution on $\prod_{w\in\hCW} \hS$ that is given by switching coordinates: $\sigma_t(z_w)=(z^\prime_w)$, where $z^\prime_w=z_{wt}$. Then the subalgebra $\hCZ\subset\prod_{w\in\hCW}\hS$ is $\sigma_t$-stable. Denote by $\hCZ^t\subset\hCZ$ the subalgebra of $\sigma_t$-invariants. For the ring $k$ we set $\hCZ_k^t=(\hCZ^t)_k$. We denote  by $\hCZ^{-t}\subset\hCZ$ the $\hCZ^t$-module of $\sigma_t$-anti-invariants. As we can divide by $2$ in $\hCZ$ we have $\hCZ=\hCZ^t\oplus \hCZ^{-t}$. 
Consider the element $c(\alpha_t)\in\hCZ$. Then $\sigma_t(c(\alpha_t))=-c(\alpha_t)$ as $t(\alpha_t)=-\alpha_t$. So $c(\alpha_t)\in\hCZ^{-t}$.

\begin{lemma}\label{lemma-decZ} We have $\hCZ=\hCZ^t\oplus c(\alpha_t)\hCZ^t$ and $\hCZ_k=\hCZ_k^t\oplus c(\alpha_t)\hCZ_k^t$.
\end{lemma}
\begin{proof} The second claim is an immediate consequence of the first. We clearly have $c(\alpha_t)\hCZ^t\subset\hCZ^{-t}$. We now show that this inclusion is a bijection.
Let $z\in\hCZ^{-t}$. Then $z_w=-z_{wt}$ and $z_w\equiv z_{wt}\mod\alpha_{wtw^{-1}}$, hence $z_w$ is divisible by $\alpha_{wtw^{-1}}=\pm w(\alpha_t)=\pm c(\alpha_t)_w$ in $\hS$. Hence $z^\prime=c(\alpha_t)^{-1}z$ is a well-defined element in $\prod_{w\in\hCW}\hS$. Moreover, $z^\prime$ is $\sigma_t$-invariant as $c(\alpha_t)$ and $z$ are anti-invariant. We have to show that  $z^\prime\in\hCZ$.

Let $w\in\hCW$ and $u\in\hCT$. Then  
\begin{align*}
c(\alpha_t)_{uw}c(\alpha_t)_{w}(z^\prime_{uw}-z^\prime_w) & =  c(\alpha_t)_{w}z_{uw}-c(\alpha_t)_{uw}z_w \\
& =  (c(\alpha_t)_{w}-c(\alpha_t)_{uw})z_{uw} + c(\alpha_t)_{uw}(z_{uw}-z_w)
\end{align*}
is divisible by $\alpha_u$, since the expressions in brackets on the
right hand side are. If $u\ne wtw^{-1}$ then neither $c(\alpha_t)_{uw}$ nor
$c(\alpha_t)_{w}$ is divisible by $\alpha_u$, so we deduce
$z^\prime_{uw}\equiv z^\prime_w\mod\alpha_u$. If $u=wtw^{-1}$, then
$z^\prime_{uw}=z^\prime_w$ as $z^\prime$ is $\sigma_t$-invariant. So $z^\prime_{uw}\equiv
z^\prime_w\mod\alpha_u$ in any case, hence $z^\prime\in\hCZ$.
\end{proof}

\subsection{Translation functors and special modules}

Let $\hCZ_k\catmod^f$ be the category of $\DZ$-graded $\hCZ_k$-modules that
are torsion free over  $\hS_k$ and finitely generated over $\hS_k$. For a simple affine reflection  $s\in\hCS$ define $\hCZ_k^s\catmod^f$ analogously.
The restriction functor $\son:=\Res_{\hCZ_k}^{\hCZ_k^s}$ is a functor from $\hCZ_k\catmod^f$ to $\hCZ_k^s\catmod^f$. By the previous lemma, the induction functor $\sout:=\Ind_{\hCZ_k^s}^{\hCZ_k}=\hCZ_k\otimes_{\hCZ_k^s}\cdot$ is a functor from $\hCZ_k^s\catmod^f$ to $\hCZ_k\catmod^f$. We denote by  $\theta^s:=\sout\circ\son\colon\hCZ_k\catmod^f\to\hCZ_k\catmod^f$ the composition. Let $B_e\in\hCZ_k\catmod^f$ be the free $\hS_k$-module of rank one on which
$z=(z_w)$ acts by multiplication with $z_e$.  

\begin{definition}\begin{enumerate}
\item The category of {\em special $\hCZ_k$-modules} is the full
  subcategory $\CH_k$ of $\hCZ_k\catmod^f$ that consists of all  objects that
  are isomorphic to a direct summand of a direct sum of modules of the
  form $\theta^s\circ\cdots\circ\theta^t(B_e)\langle n\rangle$, where   $s,\dots,t$ is an
  arbitrary sequence in $\hCS$ and $n$ is in $\DZ$. 
\item Let $s\in\hCS$. The category of {\em special $\hCZ_k^s$-modules} is the full subcategory $\CH_k^{s}$ of $\hCZ_k^s\catmod^f$ that consists of all objects that are isomorphic to a direct summand of  $\son(M)$ for some $M\in\CH_k$.
\end{enumerate}
\end{definition}

By definition, we have induced functors
$\son\colon\CH_k\to\CH_k^{s}$, $\sout\colon\CH_k^{s}\to\CH_k$ and
$\theta^s\colon\CH_k\to\CH_k$ for all $s\in\hCS$.  For $M\in \CH_{\DZ^\prime}$ and $N\in \CH_{\DZ^\prime}^s$ we have $(\son M)_k=\son(M_k)$ and $(\sout M)_k=\sout (M_k)$, hence base change yields functors from $\CH_{\DZ^\prime}$ to $\CH_k$ and from $\CH_{\DZ^\prime}^s$ to $\CH^s_k$.

 For the applications we need the ring $k$ to be a field, preferably of positive characteristic. But for some of the arguments in the following we would need the GKM-restriction on $k$ (cf.~Section \ref{subsec-GKM}). As we prefer to consider, for the moment, the whole graph $\hCG$, this restricts ourselves to the characteristic zero case. A good compromise is to work over the ring $\DZ^\prime$. So in the following, when we leave out the index $k$, we mean objects over  $\DZ^\prime$. Later (cf.~Section \ref{subsec-basechange}) we use the base change functor  to establish the crucial results for fields $k$.

\subsection{Finiteness of special modules}
The next lemma shows that we can consider each special module as a module over a finite version of the structure algebra. This will be  important for the localization constructions in the next section. 

For $\Omega\subset\hCW$ define
$$
\hCZ(\Omega):= \left\{(z_w)\in\prod_{w\in\Omega} \hS\left|
\,
\begin{matrix} z_w\equiv z_{tw}\mod\alpha_t \\ 
\text{for all $t\in\hCT$, $w\in\Omega$ with $tw\in\Omega$}
\end{matrix}\right.\right\}.
$$
The projection $\prod_{w\in\hCW}\hS\to\prod_{w\in\Omega}\hS$ with
kernel $\prod_{w\in\hCW\setminus\Omega}\hS$ induces a map
$\hCZ\to\hCZ(\Omega)$. Note that this map need not be surjective. (If it
was, this article, as well as many others, would be rather pointless.)

 Choose $s\in\hCS$. In the following we call a subset $\Omega$ of $\hCW$ {\em $s$-invariant}, if $x\in\Omega$ implies $xs\in\Omega$. For an $s$-invariant $\Omega$  we can define, as before, an algebra involution $\sigma_s$ on $\hCZ(\Omega)$. Then we let $\hCZ(\Omega)^s\subset\hCZ(\Omega)$ be the subalgebra of invariants. There is a canonical map $\hCZ^s\to\hCZ(\Omega)^s$. The proof of Lemma \ref{lemma-decZ} carries over and yields a decomposition $\hCZ(\Omega)=\hCZ(\Omega)^s\oplus c(\alpha_s)\hCZ(\Omega)^s$.

\begin{lemma}\label{Zqufin} 
\begin{enumerate}
\item Let $M$ be an object in $\CH$. Then there exists  a finite subset $\Omega$ of $\hCW$ and an action of $\hCZ(\Omega)$ on $M$, such that $\hCZ$ acts on $M$ via the canonical map $\hCZ\to\hCZ(\Omega)$. 
\item Let $s\in\hCS$ and let $N$ be an object in $\CH^{s}$. Then  there exists  a finite, $s$-invariant subset  $\Omega$ of $\hCW$ and an action of $\hCZ(\Omega)^s$ on $N$, such that $\hCZ^s$ acts on $N$ via the canonical map $\hCZ^s\to\hCZ(\Omega)^s$. 
\end{enumerate}
\end{lemma}
\begin{proof} Suppose we have proven (1). Then (2) follows, since it is enough to prove (2) for modules of the form $\vartheta_{s,on}(M)$ with $M\in\CH$, and if $M$ is acted upon by $\hCZ(\Omega)$, then $\vartheta_{s,on}(M)$ is acted upon by $\hCZ(\Omega\cup\Omega s)^s$.   

Now (1) certainly holds for
  $M=B_e$ with $\Omega=\{e\}$. It is enough to show that if (1)
  holds for $M$, then it also holds for $\theta^s(M)$ for all
  $s\in\hCS$. So assume (1) holds for $M$ and choose $s\in\hCS$. Choose
$\Omega\subset \hCW$ with the desired property for $M$. After enlarging $\Omega$ if necessary, we can assume that $\Omega$ is $s$-invariant. Then the action of $\hCZ^s$ on $\son M$ factors over an action of $\hCZ(\Omega)^s$, and since $\theta^s(M)=\hCZ\otimes_{\hCZ^s}M = \hCZ(\Omega)\otimes_{\hCZ(\Omega)^s} M$
we get an action of $\hCZ(\Omega)$ on $\theta^s(M)$ with the claimed property.
\end{proof}

\subsection{Sheaves on moment graphs}\label{sec-momgrasheaves} We need another, more local, construction of the special $\hCZ$-modules as global sections of certain sheaves on the moment graph  associated to our data.  Let us recall the notion of a sheaf on such a graph. 
So let $\CG$ be a moment graph over the lattice $Y$.  Let $k$ be a ring and denote by $S_k(Y):=S(Y\otimes_\DZ k)$ the symmetric algebra over the free $k$-module associated to $Y$, again graded in such a way that $\deg Y\otimes_\DZ k=2$. A {\em $k$-sheaf} $\SF=(\{\SF^x\},\{\SF^E\},\{\rho_{x,E}\})$ on $\CG$ is given by the following data:
\begin{itemize}
\item an $S_k(Y)$-module $\SF^x$ for any vertex $x$,
\item an $S_k(Y)$-module $\SF^E$ for any edge $E$ with the property $\alpha(E)\cdot\SF^E=0$,
\item an $S_k(Y)$-module homomorphism $\rho_{x,E}\colon\SF^x\to\SF^E$ for any vertex $x$ adjacent to the edge $E$.
\end{itemize}
A homomorphism $f\colon\SF\to\SG$ between $k$-sheaves on $\CG$ is given by $S_k(Y)$-module homomorphisms $f^x\colon \SF^x\to\SG^x$ for any vertex $x$ and $f^E\colon\SF^E\to\SG^E$ for any edge $E$ that are compatible with all $\rho$-maps. We denote by $\CG\catmod_{k}$ the resulting category of $k$-sheaves on $\CG$. Recall that we assume that each $S_k(Y)$-module and each homomorphism is graded. This yields a graded structure on $\CG\catmod_{k}$ with the obvious shift functor.

Let $\SF$ be a $k$-sheaf on $\CG$ and let $\Omega$ be a subset of the set $\CV$ of vertices  of $\CG$. The {\em sections} of $\SF$ over $\Omega$ are defined as
$$
\Gamma(\Omega,\SF):=\left\{(m_x)\in\prod_{x\in\Omega}\SF^x\left|
\begin{matrix}
\rho_{x,E}(m_x)=\rho_{y,E}(m_y)\\
\text{ for any $x,y\in\Omega$ that are }\\
\text{ connected by the edge $E$ }
\end{matrix}
\right\}\right..
$$
We denote by $\Gamma(\SF):=\Gamma(\CV,\SF)$ the space of global sections of $\SF$. 

\subsection{Translation functors on sheaves}\label{sec-transonsheaves}

In the following we construct an analog of the functors $\son$ and $\sout$ for the categories of sheaves on the moment graphs $\hCG$ and $\hCG^s$ associated to the root system $R$ and a simple affine reflection $s$.  Again we restrict ourselves to the case $k=\DZ^\prime$ and leave out the index $k$.  The word sheaf means $\DZ^\prime$-sheaf.

So let $\SF$ be a sheaf on $\hCG$. We define the sheaf $\SG=\tson\SF$ on $\hCG^s$ as follows. For $\ol x=\{x,xs\}\in\hCW^s=\hCW/\{1,s\}$  set
$\SG^{\ol x}:=\Gamma(\{x,xs\},\SF)$. Let $\ol E$ be an edge connecting $\ol x=\{x,xs\}$ and $\ol y=\{y,ys\}$. Then there is $t\in\hCT$ with $t\ol x=\ol y$. We can assume that $y=tx$, so there is an edge $E$ connecting $y$ and $x$ and another edge $Es$ connecting $ys$ and $xs$, both labelled by $\alpha_t$. We set 
$\SG^{\ol E}:=\SF^E\oplus \SF^{Es}$.
For a vertex $\ol x$ adjacent to $\ol E$ we let
$\rho_{\ol x,\ol E}$ be the composition of the inclusion $\Gamma(\{x,xs\},\SF)\subset \SF^x\oplus\SF^{xs}$ with the direct sum $\rho_{x,E}\oplus\rho_{xs,Es}\colon \SF^{x}\oplus\SF^{xs}\to \SF^{E}\oplus\SF^{Es}$. These data define a sheaf $\SG$ on $\hCG^s$. One immediately checks that this yields a functor $\tson\colon\hCG\catmod\to\hCG^s\catmod$.

Now let $\SG$ be a sheaf on $\hCG^s$. We define a sheaf $\SF=\tsout\SG$ on $\hCG$. For a vertex $x$ on $\hCG$ let $\ol x=\{x,xs\}$ be the corresponding vertex on $\hCG^s$. Then we set 
$\SF^x:=\SG^{\ol x}$. 
Let $E$ be an edge of $\hCG$ connecting $x$ and $y$. In the case $y=xs$, so $\ol x=\ol y$, we set $
\SF^E:=\SG^{\ol x}/\alpha(E)\SG^{\ol x}$
and we let $\rho_{x,E}$ and $\rho_{y,E}$ be the canonical maps $\SG^{\ol x}\to \SG^{\ol x}/\alpha(E)\SG^{\ol x}$. If $y\ne xs$, then $\ol x$ and $\ol y$ are connected by an edge $\ol E$ and we set
$\SF^E:=\SG^{\ol E}$
and $\rho_{x,E}:=\rho_{\ol x,\ol E}$, $\rho_{y,E}:=\rho_{\ol y,\ol E}$. This defines a sheaf $\SF$ on $\hCG$ and it is immediate that we obtain a functor $\tsout\colon\hCG^s\catmod\to\hCG\catmod$. We  denote by $\ttheta^s=\tsout\circ\tson\colon \hCG\catmod\to\hCG\catmod$ the composition.

Let $\Omega$ be a subset of $\hCW$. Note that coordinatewise multiplication yields a $\hCZ(\Omega)$-module structure on $\Gamma(\Omega,\SF)$ for any sheaf $\SF$ on $\hCG$. Likewise, if $\Omega$ is $s$-invariant and if $\ol\Omega\in\hCW^s$ is its image, then $\Gamma(\ol \Omega,\SG)$ is a $\hCZ(\Omega)^s$-module for each sheaf $\SG$ on $\hCG^s$. 

\begin{lemma} \label{lemma-inttwintrans} Suppose that $\Omega\subset\hCW$ is $s$-invariant and denote by $\ol \Omega\in\hCW^s$ its canonical image. For any sheaf $\SF$ on $\hCG$ we have $\Gamma(\ol\Omega,\tson \SF)=\son\Gamma(\Omega,\SF)$. For any sheaf $\SG$ on $\hCG^s$ such that for any edge $\ol E$ the stalk $\SG^{\ol E}$ is a free $\hS/\alpha(E)\hS$-module, we have $\Gamma(\Omega,\tsout\SG)=\sout\Gamma(\ol\Omega,\SG)$.
\end{lemma}

\begin{proof} Note that from the definition of $\tson$ it immediately follows that we have an identification $\Gamma(\Omega,\SF)=\Gamma(\ol\Omega,\tson\SF)$  as $\hCZ(\Omega)^s$-modules, which yields the first claim.

In order to prove the second claim, let us set $\SF=\tsout\SG$. We then have $\SF^x=\SF^{xs}$ for any $x\in\Omega$. Hence we can define an involution $\sigma_s$ on $\prod_{x\in\Omega}\SF^x$ by $\sigma_s(m_x)=(m^\prime_x)$, where $m^\prime_x=m_{xs}$ for any $x$. It follows from the definitions  that $\Gamma(\Omega,\SF)\subset\prod_{x\in\Omega}\SF^x$ is $\sigma_s$-stable. Hence, we have a decomposition
$\Gamma(\Omega,\SF)=\Gamma(\Omega,\SF)^s\oplus\Gamma(\Omega,\SF)^{-s}$
into $\sigma_s$-invariants and $\sigma_s$-anti-invariants. 

We now prove that $\Gamma(\Omega,\SF)^{-s}=c(\alpha_s)\Gamma(\Omega,\SF)^{s}$. The arguments are similar to the ones used for Lemma \ref{lemma-decZ}. Let $x\in\Omega$ and let $E$ be the edge connecting $x$ and $xs$.
We have $\SF^x=\SF^{xs}=\SG^{\ol x}$ and $\SF^E=\SG^{\ol x}/\alpha(E)\SG^{\ol x}$. If $m=(m_x)\in \Gamma(\Omega,\SF)^{-s}$, then $m_{x}=-m_{xs}$ and $m_x\equiv m_{xs}\mod \alpha(E)$. Hence $m_x$ and $m_{xs}$ are divisible by $\alpha(E)$ in $\SG^{\ol x}$. Now $\alpha(E)$ is the positive root corresponding to the reflection $xsx^{-1}$, i.e.~$\alpha(E)=\pm x(\alpha_s)=\pm c(\alpha_s)_x$. So $m^\prime=c(\alpha_s)^{-1} m$ is a well-defined element in $\prod_{x\in\Omega}\SF^x$, which is, moreover, $\sigma_s$-invariant. Hence it remains to show that $m^\prime$ is a section in $\Gamma(\Omega,\SF)$.

So let $E$ now be an edge connecting $x$ and $y$. We want to show that $\rho_{x,E}(m^\prime_x)=\rho_{y,E}(m^\prime_y)$. In the case $y=xs$ this is clear. So suppose that $y\ne xs$. As $m$ is a section, we have $\rho_{x,E}(m_x)=\rho_{y,E}(m_y)$. By our assumption, $\SF^E=\SG^{\ol E}$ is a free $\hS/\alpha(E)\hS$-module. Now  $c(\alpha_s)_x$ and $c(\alpha_s)_y$ are equivalent modulo $\alpha(E)$, hence they yield the same element in $\hS/\alpha(E)\hS$, which is non-zero (here the GKM-restriction is needed in case we worked over a field of positive characteristic). We deduce that $\rho_{x,E}(m^\prime_x)=\rho_{y,E}(m^\prime_y)$. Hence $m^\prime$ is a section of $\SF$. So we showed that $\Gamma(\Omega,\SF)^{-s}=c(\alpha_s)\Gamma(\Omega,\SF)^{s}$, hence we have $\Gamma(\Omega,\SF)=\Gamma(\Omega,\SF)^{s}\oplus c(\alpha_s)\Gamma(\Omega,\SF)^{s}$.

From the definition of $\tsout$ it is immediate that $\Gamma(\Omega,\SF)^s$ can be identified with $\Gamma(\ol\Omega,\SG)$ as a $\hCZ^s$-module. Using Lemma  \ref{lemma-decZ} we get an identification $\Gamma(\Omega,\SF)=\Gamma(\ol\Omega,\SG)\oplus c(\alpha_s)\Gamma(\ol\Omega,\SG)=\sout\Gamma(\ol\Omega,\SG)$ of $\hCZ$-modules.
\end{proof}

\subsection{Special sheaves on moment graphs}

We let $\SB_e$ be the following sheaf on $\hCG$. We set $\SB_e^e=\hS$ and $\SB_e^x=0$ for $x\ne e$. For any edge $E$ we set $\SB^E_e=0$ and we let all $\rho_{x,E}$'s be the zero homomorphism.

\begin{definition}\begin{enumerate}
\item The category of {\em special sheaves on $\hCG$} is the full
  subcategory $\CA$ of $\hCG\catmod_{\DZ^\prime}$ that consists of all  objects that
  are isomorphic to a direct summand of a direct sum of sheaves of the
  form $\ttheta^s\circ\cdots\circ\ttheta^t(\SB_e)\langle n\rangle$, where   $s,\dots,t$ is an
  arbitrary sequence in $\hCS$ and $n$ is in $\DZ$. 
\item Let $s\in\hCS$. The category of {\em special sheaves on $\hCG^s$} is the full subcategory $\CA^{s}$ of $\hCG^s\catmod_{\DZ^\prime}$ that consists of all objects that are isomorphic to a direct summand of  $\tson(\SF)$ for some $\SF\in\CA$.
\end{enumerate}
\end{definition}

By definition, we have induced functors
$\tson\colon\CA\to\CA^{s}$, $\tsout\colon\CA^{s}\to\CA$ and
$\ttheta^s\colon\CA\to\CA$ for all $s\in\hCS$. From the definitions of the translation functors it follows that for each $\SF\in\CA$ and each $\SG\in\CA^s$ the stalks $\SF^x$ and $\SG^{\ol x}$ are free $\hS$-modules of finite rank,  and $\SF^E$ and $\SG^{\ol E}$ are free $\hS/\alpha(E)\hS$-modules of finite rank.

\subsection{The global sections of special sheaves}

Note that we use a simplified definition of the translation functors.  In general, the special sheaves that we obtain are not generated by global sections in the sense of \cite{Fie05a}, so are not Braden--MacPherson sheaves. However, this definition nevertheless serves our purpose which is to interpret the special $\hCZ$-modules as global sections of moment graph sheaves. This is used in the proof of Lemma \ref{lemma-subquotfree} in an essential way. 

\begin{proposition}\label{prop-inttwintrans}   \begin{enumerate} 
\item For $\SF\in\CA$ we have $\Gamma(\SF)\in\CH$. For any $s\in\hCS$ and $\SG\in\CA^s$ we have $\Gamma(\SG)\in\CH^s$. 
\item Let $s\in\hCS$. Then the following diagrams of functors commute:

\centerline{
\xymatrix{
\CA \ar[d]_-{\tson} \ar[r]^-{\Gamma} & \CH \ar[d]^-{\son}\\
 \CA^s \ar[r]^-{\Gamma} & \CH^s,
}\quad  
\xymatrix{
\CA^s \ar[d]_-{\tsout} \ar[r]^-{\Gamma} & \CH^s \ar[d]^-{\sout}\\
 \CA \ar[r]^-{\Gamma} & \CH.
}
}
\end{enumerate}
\end{proposition}

\begin{proof} It follows immediately from the definitions that $\Gamma(\SB_e)\cong B_e$ as $\hCZ$-modules. In order to prove our claims it is enough to show that for $\SF\in\CA$ with $\Gamma(\SF)\in\CH$ we have  $\Gamma(\tson\SF)\cong\son\Gamma(\SF)$ and $\Gamma(\tsout\tson \SF)\cong\sout\Gamma(\tson\SF)$. For the first claim we can directly apply Lemma \ref{lemma-inttwintrans}. As each stalk of a special sheaf  is free (by construction), we can apply the same lemma to prove the second claim. 
\end{proof}

\section{Localization of special modules}\label{sec-Loc}

Each special $\hCZ$-module is naturally a module over the symmetric algebra $\hS$. In this section we collect some results on what happens if we invert some of the affine roots. In previous papers we used to work with certain localizations at prime ideals of the underlying symmetric algebra. But since the main aim of this paper is to relate the category of special modules to the category defined by Andersen, Jantzen and Soergel, it is convenient to use the following subalgebras of the quotient field of $\hS$. Let us fix $\beta\in R^+$ and let us set
$$
\hS^{\emptyset}:=\hS[\alpha_n^{-1}\mid\alpha\in R^+, n\in\DZ]\text{ and } \hS^{\beta}:=\hS[\alpha_n^{-1}\mid \alpha\in R^+, \alpha\ne\beta, n\in\DZ].
$$
  For a subset $\Omega$ of $\hCW$ set
$\hCZ^\emptyset(\Omega):=\hCZ(\Omega)\otimes_{\hS}\hS^\emptyset$ and $\hCZ^\beta(\Omega):=\hCZ(\Omega)\otimes_{\hS}\hS^\beta$.

Recall that we defined the subgroup $\hCW^\beta$ of $\hCW$ that is generated by all $s_{\beta,n}$, $n\in\DZ$.
We denote by $\hCW^\beta\backslash\hCW$ the set of orbits in $\hCW$ of the left $\hCW^\beta$-action. 

\begin{lemma}\label{lemma-Zloc} For a finite subset $\Omega$ we have canonical isomorphisms
\begin{align*}
\hCZ^\emptyset(\Omega)&=\bigoplus_{w\in\Omega}\hS^\emptyset,\\
\hCZ^\beta(\Omega)&=\left\{(z_w)\in\bigoplus_{w\in\Omega} \hS^\beta\left|\,\begin{matrix} z_w\equiv z_{s_{\beta,n}w}\mod\beta_n \\ 
\text{ for all $w\in\Omega$,  $n\in\DZ$ with $s_{\beta,n}w\in\Omega$}
\end{matrix}\right.\right\}\\
&=\bigoplus_{\Theta\in\hCW^\beta\backslash\hCW}\hCZ^\beta(\Omega\cap\Theta).
\end{align*}
\end{lemma}

\begin{proof}  The inclusion $\hCZ(\Omega)\subset\bigoplus_{w\in\Omega}\hS$ yields an inclusion $\hCZ^\emptyset(\Omega)\subset\bigoplus_{w\in\Omega}\hS^\emptyset$. We want to show that this is a bijection.  Let $d\in \hS$ be the product of all $\alpha_t$, where $t\in\hCT$ is a reflection connecting two elements in $\Omega$. From the definition of $\hCZ(\Omega)$ it is immediate that   $\bigoplus_{w\in\hCW} d\hS\subset \hCZ(\Omega)$, and since $d$ is invertible in $\hS^\emptyset$ we can deduce our first claim.

For the second claim we argue analogously. Let us denote by $L$ the set on the right hand side of the second claimed equation. It is clear that $\hCZ^\beta(\Omega)$ is contained in $L$. Hence it is enough to show that $L\cap \bigoplus_{w\in\Omega}\hS$ is contained in $\hCZ^\beta(\Omega)$. Now let $d\in\hS$ be the product of all $\alpha_t$ for reflections $t$ that are not of type $\beta$ (i.e.~$t\ne s_{\beta,n}$ for all $n$), but connect two elements in $\Omega$. Then we have $L\cap\bigoplus_{w\in\Omega}d\hS\subset\hCZ(\Omega)$. As $d$ is invertible in $\hS^\beta$, we can deduce the second equation. The third is a direct consequence.
\end{proof}

Let $M\in\CH$ and define $M^{\emptyset}:=M\otimes_{\hS_k} \hS^{\emptyset}$ and $M^{\beta}:=M\otimes_{\hS} \hS^{\beta}$. 
Recall that $M$ is a module for
$\hCZ(\Omega)$ for some finite subset $\Omega$ of $\hCW$ by Lemma
\ref{Zqufin}. So $M^\emptyset$ and $M^\beta$ are modules over
$\hCZ^\emptyset(\Omega)$ and $\hCZ^\beta(\Omega)$, resp.

\begin{lemma}\label{lemma-decompreg} The decompositions in Lemma \ref{lemma-Zloc} induce decompositions
$$
M^\emptyset  =  \bigoplus_{w\in\hCW} M^{\emptyset, w}, \quad 
M^\beta   =   \bigoplus_{\Theta\in\hCW^\beta\backslash\hCW} M^{\beta,\Theta}.
$$
\end{lemma}

Now fix $s\in\hCS$.
We want to prove a similar statement for objects $N\in\CH^{s}$. Suppose that $\Omega\subset\hCW$ is $s$-invariant. Then the involution $\sigma_s$ induces an involution on $\hCZ^\emptyset(\Omega)$ and on $\hCZ^\beta(\Omega)$ for any $\beta\in R^+$ that we denote, for simplicity, by $\sigma_s$ as well. We let $\hCZ^{\emptyset}(\Omega)^s$ and $\hCZ^{\beta}(\Omega)^s$ be the corresponding sets of invariants.  

Recall that $\hCW_s=\{1,s\}$ is the parabolic subgroup corresponding to $s$ and that  $\hCW^s=\hCW/\hCW_s$.  For $w\in\hCW$ we denote by $\ol w\in\hCW^s$
its
canonical image. 
The group $\hCW^\beta$ still acts on $\hCW^s$ from the left and the quotient map $\hCW\to \hCW^s$ is $\hCW^\beta$-equivariant. For  $\Theta\in\hCW^\beta\backslash\hCW$ we denote by
$\ol\Theta\in\hCW^\beta\backslash\hCW^s$ its canonical image.
From Lemma \ref{lemma-Zloc} we get decompositions
\begin{align*}
\hCZ^\emptyset(\Omega)^s=\bigoplus_{\ol w\in\ol\Omega}\hS^\emptyset, \quad\hCZ^\beta(\Omega)^s=\bigoplus_{\ol\Theta\in\hCW^\beta\backslash\hCW^s}\hCZ^\beta(\Omega\cap\ol\Theta)^s.
\end{align*}
Hence we obtain the following analog of Lemma \ref{lemma-decompreg}:

\begin{lemma}\label{decompsing} Let $N\in\CH^{s}$. Then there are canonical decompositions
$$
N^\emptyset  =  \bigoplus_{\ol w\in\hCW^s} N^{\emptyset, \ol w}, \quad 
N^\beta   =   \bigoplus_{\ol \Theta\in\hCW^\beta\backslash\hCW^s} N^{\beta,\ol \Theta}.
$$
\end{lemma}

\subsection{Localization of translation functors}

Now we study the behaviour of the translation functors under base change. 
Let us fix $s\in\hCS$. 

\begin{lemma}\label{lemma-gendecomptrans} For $M\in\CH$, $N\in\CH^{s}$ and $w\in\hCW$ we have
$$
(\son M)^{\emptyset, \ol w}=M^{\emptyset,w}\oplus M^{\emptyset,ws}, \quad (\sout N)^{\emptyset,w}=N^{\emptyset,\ol w}.
$$
\end{lemma}
\begin{proof} 
Note that for a finite $s$-invariant subset $\Omega$ of $\hCW$ the inclusion $\hCZ^\emptyset(\Omega)^s\subset\hCZ^\emptyset(\Omega)$ equals $\left(\bigoplus_{w\in\Omega}\hS^\emptyset\right)^s\subset\bigoplus_{ w\in\Omega}\hS^\emptyset$ and the claims follow immediately.
\end{proof}

Now fix $\beta\in R^+$. 
For $\Theta\in\hCW^\beta\backslash\hCW$ we
have two cases to distinguish. Either we have $\Theta=\Theta s$, or $\Theta\ne
\Theta s$. In the latter case the map that sends $w$ to $\ol w$ yields a
bijection  $\Theta\stackrel{\sim}\to \ol \Theta$.

\begin{lemma}\label{loctransone} Choose $\Theta\in\hCW^\beta\backslash\hCW$
  and let $\ol \Theta\in\hCW^\beta\backslash\hCW^s$ be its image.  Then 
\begin{enumerate}
\item  $(\son M)^{\beta,\ol\Theta}=M^{\beta,\Theta}\oplus M^{\beta,\Theta s}$, if $\Theta\ne \Theta s$, and $(\son M)^{\beta,\ol\Theta}=M^{\beta,\Theta}$ if $\Theta=\Theta s$. 
\item  $(\sout N)^{\beta,\Theta}= N^{\beta,\ol\Theta}$, if $\Theta\ne \Theta s$, and $(\sout N)^{\beta,\Theta}= \hCZ^\beta(\Theta)\otimes_{\hCZ^\beta(\Theta)^s} N^{\beta,\ol\Theta}$ if $\Theta=\Theta s$. 
\end{enumerate}
\end{lemma}
\begin{proof} In the case $\Theta=\Theta s$ the inclusion $\hCZ^\beta(\Omega)^s\subset\hCZ^\beta(\Omega)$ contains $\hCZ^\beta(\Theta)^s\subset\hCZ^\beta(\Theta)$ as a direct summand. In the case $\Theta\ne\Theta s$ we have an analogous direct summand $\hCZ^\beta(\Theta\cup\Theta s)^s\subset\hCZ^\beta(\Theta)\oplus\hCZ^\beta(\Theta s)$, and this inclusion is an isomorphism on each direct summand. From these facts the claims follow easily.
\end{proof}

Let $\Omega\subset\hCW$ be a subset and $\Lambda:=\hCW\setminus\Omega$ its complement. Let $M$ be an object in $\CH$. We consider the canonical inclusion $M\subset M^\emptyset=\bigoplus_{w\in\hCW}M^{\emptyset,w}$ and define
\begin{align*}
M_\Lambda&:=M\cap \bigoplus_{w\in\Lambda} M^{\emptyset,w},\\
M^\Omega&:=M/M_\Lambda=\im\left(M\to M^{\emptyset}\to \bigoplus_{w\in\Omega}M^{\emptyset,w}\right).
\end{align*}
If $\Omega^\prime\subset\Omega$ are subsets in  $\hCW$ with complements $\Lambda\subset\Lambda^\prime$, then we have an inclusion $M_{\Lambda}\subset M_{\Lambda^\prime}$ and a corresponding surjection $M^{\Omega}\to M^{\Omega^\prime}$. 

Let $N\in\CH^s$. For any subset $\ol\Omega$ of $\hCW^s$ with complement $\ol\Lambda$ we can analogously define a submodule $N_{\ol \Lambda}$ and a quotient $N^{\ol \Omega}$ of $N$. The following is an immediate consequence of Lemma \ref{lemma-gendecomptrans}.

\begin{lemma}\label{lemma-transsubquot} Let $s\in\CS$ and let $\Omega$ be an $s$-invariant subset of $\hCW$ with complement $\Lambda$. Denote by $\ol\Omega,\ol\Lambda\in\hCW^s$ the images. For $M\in\CH$ and $N\in\CH^s$ we have
\begin{align*}
(\son M)_{\ol\Lambda}=\son (M_\Lambda),\quad (\sout N)_\Lambda=\sout( N_{\ol\Lambda}),\\
(\son M)^{\ol\Omega}=\son (M^\Omega),\quad (\sout N)^\Omega=\sout (N^{\ol\Omega}).
\end{align*}
\end{lemma}

\section{Filtrations and characters} \label{sec-Fil}

In this section we describe how certain partial orders on the set
$\hCW$  induce functorial filtrations on the objects of  $\CH$ indexed by $\hCW$. The
subquotients of these filtrations turn out to be graded free
$\hS$-modules of finite rank and hence give rise to the definition of
a character map from the Grothendieck group of $\CH$ to the free
$\DZ[v,v^{-1}]$-module $\bW$ with basis $\hCW$.  We are particularly interested in the characters associated to the
ordinary Bruhat order ``$\varle$''
and the ``generic'' Bruhat order
``$\succeq$'' on $\hCW$ (cf.~Sections \ref{Bruord} and \ref{Brinf}). 

We identify $\bW$ with
the $\DZ[v,v^{-1}]$-module underlying both the affine Hecke algebra $\bH$
and its periodic module $\bM$, and we show that the translation
combinatorics on $\CH$ forms a categorification of $\bH$ via the
character associated to ``${\varle}$'' (Lemma \ref{hvarle}), and of $\bM$ via the character
associated to ``${\succeq}$'' (Lemma \ref{hsucceq}).

\subsection{Partial orders}\label{partord}
Let ``$\unlhd$'' be a partial order on the set $\hCW$. For $x,y\in\hCW$
denote by $[x,y]=\{z\in\hCW\mid x\unlhd z\unlhd y\}$ the interval between $x$ and $y$. Suppose that
``$\unlhd$'' has the
following properties:  
\begin{enumerate} 
\item The elements $w$ and $tw$ are comparable for all $w\in\hCW$ and
  $t\in\hCT$. The relations between all such pairs $w,tw$ generate the partial order.
\item We have $[w,ws]=\{w,ws\}$ for all $w\in\hCW$ and $s\in\hCS$ such
  that  $w \unlhd ws$.  
\item For $x,y\in\hCW$ such that $x\unlhd xs$ and $y\unlhd xs$ we have $ys\unlhd xs$.  For $x,y\in\hCW$ such that $xs\unlhd x$ and $xs\unlhd y$ we have $xs\unlhd ys$.
\end{enumerate}

We call a  subset $\Omega$  of $\hCW$ {\em open}, if $x\in\Omega$, $y\in\hCW$ and $x\unlhd y$ imply $y\in\Omega$. Note that assumption (3) implies that if $\Omega$ is open and if $s\in\CS$, then $\Omega\cup\Omega s$ is open as well and that for each $x\in\hCW$ with $xs\unlhd x$ the sets $\{y\in\hCW\mid y\unlhd x\}$ and $\{y\in\hCW\mid xs\unlhd y\}$ are $s$-invariant. 
For $s\in\hCS$  each coset in $\hCW^s=\hCW/\{1,s\}$ has a smallest
representative by assumption. Hence comparing these representatives gives us an
induced partial order on $\hCW^s$ that we denote by 
``$\unlhd$'' as well.

\subsection{Local sections of moment graph sheaves}

Let $\SF$ be an object in $\CA$. By Proposition \ref{prop-inttwintrans} we have $\Gamma(\SF)\in\CH$. For any subset $\Omega$ of $\hCW$ the restriction map $\Gamma(\SF)\to\Gamma(\Omega,\SF)$ factors over a map $\Gamma(\SF)^\Omega\to\Gamma(\Omega,\SF)$. We now prove that these maps are bijections for open sets $\Omega$. 

\begin{lemma}\label{lemma-specmodglobsec} \begin{enumerate}
\item For a sheaf $\SF$ in $\CA$ and an  open subset $\Omega$ of $\hCW$ we have $\Gamma(\SF)^{\Omega}=\Gamma(\Omega,\SF)$. 
\item For a simple affine reflection $s$, a sheaf $\SG$ in $\CA^s$ and an open subset $\ol\Omega$ of $\hCW^s$ we have $\Gamma(\SG)^{\ol\Omega}=\Gamma(\ol\Omega,\SG)$.
\end{enumerate}
\end{lemma}
\begin{proof} 
Statement (1) is obvious in the case $\SF=\SB_e$. So using induction it is enough to consider the following situation. Suppose that  (1) holds for  $\SF^\prime\in\CA$. We then have to show that (2) holds for $\SG:=\tson\SF^\prime$ and (1) holds for $\SF=\tsout\SG$. We repeatedly use the Lemmas \ref{lemma-inttwintrans} and \ref{lemma-transsubquot} in the following.

So suppose that $\ol\Omega$ is open in $\hCW^s$. Then its preimage $\Omega$ is $s$-invariant and open in $\hCW$. We have $\Gamma(\tson\SF^\prime)^{\ol\Omega}=(\son\Gamma(\SF^\prime))^{\ol\Omega}=\son(\Gamma(\SF^\prime)^\Omega)=\son(\Gamma(\Omega,\SF^\prime))=\Gamma(\ol\Omega,\tson\SF^\prime)$, and we have shown property (2) for $\SG=\tson\SF^\prime$.

Now  consider the sheaf $\SF=\tsout\SG$.  Let $\Omega$ be open in $\hCW$ and suppose first
 that $\Omega$ is $s$-invariant. Then we have $\Gamma(\tsout\SG)^\Omega=(\sout\Gamma(\SG))^\Omega=\sout(\Gamma(\SG)^{\ol\Omega})=\sout\Gamma(\ol\Omega,\SG)=\Gamma(\Omega,\tsout\SG)$, hence we proved (1) for $s$-invariant $\Omega$.  
 
 Note that the claim in (1) is equivalent to the surjectivity of the restriction $\Gamma(\SF)\to \Gamma(\Omega,\SF)$. So we have proven this under the $s$-invariance assumption. For an arbitrary open subset $\Omega$, the set $\Omega\cup \Omega s$ is open as well. Using induction it is hence enough to show that $\Gamma(\Omega\cup\{x\},\SF)\to\Gamma( \Omega,\SF)$ is surjective for any $x\in\Omega s$, $x\not\in\Omega$ such that $\Omega\cup \{x\}$ is open as well. So let us fix such an $x$ and let us choose a section $m=(m_y)_{y\in \Omega}$  in $\Gamma(\Omega,\SF)$. We want to find some $m_x\in\SF^x$ such that $((m_y)_{y\in\Omega},m_x)$ is a section of $\SF$ over $\Omega\cup\{x\}$.
  
  Now any two connected vertices in the graph are comparable, hence in order to find $m_x$ it suffices to consider only those $m_y$ with $x\lhd y$.  We have $x\lhd xs$ and hence the set $\{y\mid x\lhd y\}\setminus\{xs\}$ is $s$-invariant. As we have proven the surjectivity statement already for $s$-invariant open sets,  we can assume that $m_y=0$ for all $y\ne xs$ with  $x\lhd y$. But in this case it immediately follows from the definition of $\tsout$  that setting $m_x:=m_{xs}$ actually yields  a section $((0)_{x\lhd y, y\ne xs}, m_x,m_{xs})$  of $\SF$ over $\{y\mid x\unlhd y\}$, hence an extension of $m$ to the vertex $x$. Hence $\Gamma(\Omega\cup\{x\},\SF)\to \Gamma(\Omega,\SF)$ is a surjective map.
  \end{proof}

\subsection{Subquotients}
\label{subsec-subquot}

Let $\SF\in\CA$ and let $M=\Gamma(\SF)\in\CH$ .  Let $\Omega\subset\hCW$ be open and suppose that $x$ is a minimal element in $\Omega$, so $\Omega^\prime:=\Omega\setminus\{x\}$ is open as well. We let $K=K_{\Omega,x}$ be the kernel of the surjection $M^\Omega\to M^{\Omega^\prime}$. From Lemma \ref{lemma-specmodglobsec} we deduce that $K$ can be identified with the kernel of the restriction map $\Gamma(\Omega,\SF)\to \Gamma(\Omega^\prime,\SF)$, hence with the elements in $\Gamma(\Omega,\SF)$ supported on $\{x\}$. Hence this is the set of elements $m_x$ in $\SF^x$ with the property that $\rho_{x,E}(m_x)=0$ for any edge $E$ connecting $x$ to some $y$ with $x\lhd y$. Now the latter description is independent of the choice of $\Omega$ (but certainly depends on ``$\lhd$''). Hence we obtain that $M_{[x]}=M_{[x,\unlhd]}:=K_{\Omega,x}$ is a well-defined subquotient of  $M$.

Suppose that $\Lambda\subset \hCW$ is the complement of $\Omega$. Then $\Lambda^\prime=\Lambda\cup\{x\}$ is the complement of $\Omega^\prime$. Clearly, the cokernel of $M_{\Lambda}\to M_{\Lambda^\prime}$ identifies with the kernel of $M^\Omega\to M^{\Omega^\prime}$, hence with $M_{[x]}$. In particular, we can identify $M_{[x]}$ with the image of the map $M_{\{y\mid y\unlhd x\}}\subset \bigoplus_{y} M^{\emptyset, y}\to M^{\emptyset, x}$, where the last map is the projection along the decomposition.

Let $s$ be a simple affine reflection and let $x\in\hCW$. Then $\{x,xs\}$ is an interval in the partial order ``$\unlhd$''. Suppose that $x\lhd xs$ and set $\Omega:=\{y\mid x\unlhd y\}$, $\Omega^\prime:= \Omega\setminus\{x\}$ and $\Omega^{\prime\prime}=\Omega^\prime\setminus\{xs\}$. Then $\Omega$, $\Omega^\prime$ and $\Omega^{\prime\prime}$  are open. We let $M_{[x,xs]}$ be the kernel of the surjection $M^\Omega\to M^{\Omega^{\prime\prime}}$. As $M_{[x]}$ and $M_{[xs]}$ can be identified with the kernels of $M^\Omega\to M^{\Omega^\prime}$ and $M^{\Omega^\prime}\to M^{\Omega^{\prime\prime}}$ we obtain a short exact sequence 
$$
0\to M_{[x]}\to M_{[x,xs]}\to M_{[xs]}\to 0.
$$

\begin{lemma}\label{lemma-subquotfree} Let $x\in\hCW$, $s\in\hCS$ and
  suppose that $x\unlhd xs$. Then we have isomorphisms
\begin{eqnarray*}
(\theta^s M)_{[x]} &\cong & M_{[x]}\langle -2\rangle \oplus
M_{[xs]}\langle -2\rangle, \\
(\theta^s M)_{[xs]}& \cong &  M_{[x]}\oplus M_{[xs]} 
\end{eqnarray*}
of graded $\hS$-modules. 
In particular, for each $M\in\CH$ and all $y\in\hCW$ the space $M_{[y]}$ is a graded free $\hS$-module of finite rank.
\end{lemma}

\begin{proof}
Since $\Omega$ and $\Omega^{\prime\prime}$ are $s$-invariant we have, by Lemma \ref{lemma-transsubquot}, 
$\theta^s M^\Omega=(\theta^s M)^\Omega$ and $\theta^s M^{\Omega^{\prime\prime}}=(\theta^s M)^{\Omega^{\prime\prime}}$. As $\theta^s$ is an exact functor on $\hCZ\catmod^f$, we can deduce $\theta^s M_{[x,xs]}=(\theta^s M)_{[x,xs]}$. Now
$
\theta^s M_{[x,xs]} = \hCZ(\{x,xs\})\otimes_{\hCZ(\{x,xs\})^s} M_{[x,xs]}=\hCZ(\{x,xs\})\otimes_{\hS} M_{[x,xs]}$.
Moreover,  $\hCZ(\{x,xs\})_{[x]}\cong \hS\langle -2\rangle$ and $\hCZ(\{x,xs\})_{[xs]}\cong \hS$. From the short exact sequence above we deduce that the $\hS$-modules $M_{[x,xs]}$ and $M_{[x]}\oplus M_{[xs]}$ are isomorphic, hence we get the isomorphisms that are stated in the lemma. 
\end{proof}

\subsection{The $k$-linear version}\label{subsec-basechange}

Let $k$ be a ring in which $2$ is invertible. Recall that we defined  $\hCZ_k=\hCZ\otimes_{\DZ^\prime}k$ and $\hCZ^s_k=\hCZ^s\otimes_{\DZ^\prime}k$, hence any module in $\CH_k$ occurs as a direct summand of $M_k=M\otimes_{\DZ^\prime} k$ for some object $M$ of $\CH$. We now translate the previous results on the category $\CH$ to $\CH_k$ using base change. 

For a finite subset $\Omega$ of $\hCW$ we set $\hCZ_k(\Omega)=\hCZ(\Omega)_k$. If $\Omega$ is $s$-invariant, we set $\hCZ_k(\Omega)^s=(\hCZ(\Omega)^s)_k$. We get canonical maps $\hCZ_k\to \hCZ_k(\Omega)$, $\hCZ^s_k\to \hCZ_k(\Omega)^s$ and the statements of Lemma \ref{Zqufin} carry over to the $k$-linear versions. Now the statement of  Lemma \ref{lemma-inttwintrans} does not necessarily carry over, as we used an analog of the GKM-property in its proof (but compare Proposition \ref{prop-globsecBM}), which is why we defined the category $\CA$ only over $\DZ^\prime$. 

We set $\hCZ^\emptyset_k(\Omega)=\hCZ^\emptyset(\Omega)_k$ and $\hCZ^\beta_k(\Omega)=\hCZ^\beta(\Omega)_k$ for $\beta\in R^+$. From Lemma \ref{lemma-Zloc} we deduce  $\hCZ^\emptyset_k(\Omega)=\bigoplus_{w\in\Omega}\hS_k^\emptyset$ and $\hCZ_k^\beta(\Omega)=\bigoplus_{\Theta\in\hCW^\beta\backslash\hCW}\hCZ_k^\beta(\Omega\cap\Theta)$. So we get a decomposition as in Lemma \ref{lemma-decompreg} also for objects in $\CH_k$.  
  For a simple affine reflection $s$ and an $s$-invariant $\Omega$ we set
$\hCZ_k^{\emptyset}(\Omega)^s=(\hCZ^{\emptyset}(\Omega)^s)_k$, etc. Then the statements of Lemma \ref{decompsing} carry over. For $M\in \CH_k$ or $N\in \CH_k^s$ we define $M^\Omega$, $M_\Lambda$, $N^{\ol\Omega}$ and $N_{\ol\Lambda}$ as before using the generic decompositions. These decompositions are  compatible with base change, i.e.~for $M\in \CH$ and $w\in \hCW$ we have $M_k^{\emptyset,w}=(M^{\emptyset,w})_k$, etc., so Lemmas \ref{lemma-gendecomptrans}, \ref{loctransone} and \ref{lemma-transsubquot} carry over. 

Now we define, for $M\in\CH_k$ and $x\in\hCW$ the subquotients $M_{[x]}$ and $M_{[x,xs]}$ as the kernels of the maps $M^{\unrhd x}\to M^{\rhd x}$ and $M^{\unrhd x}\to M^{\rhd x,\ne xs}$ if $xs\rhd x$.  Then the freeness result of Lemma \ref{lemma-subquotfree} implies that for $M\in \CH$ we have $(M_k)_{[x]}=(M_{[x]})_k$, etc. Hence also Lemma \ref{lemma-subquotfree} carries over.

From now on we consider for a given $M\in\CH_k$ the quotients $M^\Omega$ only for open subsets $\Omega$, and the submodules $M_\Lambda$ only for their closed complements $\Lambda$.

\subsection{Character maps associated to partial orders}
Suppose that $L$ is a $\DZ$-graded, graded free $\hS_k$-module of finite
rank, i.e.~$L\cong \bigoplus_{i=1,\dots,n} \hS_k\langle k_i\rangle$ for some
$k_i\in\DZ$. The multiset of numbers $\{k_i\}$ is well-defined, so we
set 
$$
\grk\, L:=\sum_{i=1,\dots,n} v^{-k_i}\in\DN[v,v^{-1}].
$$ 

Choose a length function $l\colon\hCW\to\DZ$ with respect to
``$\unlhd$'', i.e.~a function with the property that $l(xs)=l(x)+1$ for
all $x\in\hCW$ and $s\in\hCS$ with $x\unlhd xs$. Assume that $l(e)=0$. 
Let $\bW$ be the free
$\DZ[v,v^{-1}]$-module with basis $\{ W_x\}_{x\in\hCW}$. 
By Lemma \ref{lemma-subquotfree}  we can define for each $M\in\CH_k$ the element 
$$
h_{\unlhd,l}(M)  :=  \sum_{x\in\hCW} v^{l(x)}\grk\, M_{[x,\unlhd]}W_x\in\bW.
$$
This yields a map from the Grothendieck group of $\CH_k$ to $\bW$ that
we notate as $h_{\unlhd,l}\colon \CH_k \dashrightarrow \bW$.
For  $s\in\hCS$  we define the $\DZ[v,v^{-1}]$-linear map $\rho_{s,\unlhd}\colon\bW\to\bW$ by 
$$
\rho_{s,\unlhd}(W_x) := 
\begin{cases} 
W_{xs}+ v^{-1}W_x,  &\text{if $xs\unlhd x$}, \\
W_{xs}+ v W_x, & \text{if $x\unlhd xs$}.
\end{cases}
$$

\begin{proposition}\label{chartrans} For each $M\in\CH_k$ and $s\in\hCS$ we have
$$
h_{\unlhd,l}(\theta^s M\langle1\rangle) =  \rho_{s,\unlhd}(h_{\unlhd,l}(M)).
$$
\end{proposition}

\begin{proof} 
For  $x\in\hCW$ we have, by Lemma \ref{lemma-subquotfree},
$$
\grk(\theta^s M)_{[x]}  
= 
\begin{cases} 
\grk\, M_{[x]}+\grk\, M_{[xs]}, & \text{if $ xs\unlhd x$}, \\
v^2\cdot \grk\, M_{[x]}+v^2\cdot \grk\, M_{[xs]}, & \text{if $x\unlhd xs$.}
\end{cases}
$$
Hence
\begin{eqnarray*}
h_{\unlhd,l}(\theta^s M\langle1\rangle) & = & \sum_{x\in\hCW}v^{l(x)-1}\grk (\theta^s M)_{[x]}  W_x \\
& = & \sum_{x\in\hCW, xs\unlhd x}  v^{l(x)-1}  \left(\grk\, M_{[x]}+\grk\, M_{[xs]}\right)W_x \\
& & + \sum_{x\in\hCW, x\unlhd xs}  v^{l(x)+1} \left(\grk\, M_{[xs]}+\grk\, M_{[x]}\right) W_x.
\end{eqnarray*}
Moreover,
\begin{eqnarray*}
\rho_{s,\unlhd}(h_{\unlhd,l}(M)) & = & \sum_{x\in\hCW}(v^{l(x)} \grk\, M_{[x]})\rho_{s,\unlhd}(W_x)\\ 
& = & \sum_{x\in\hCW, xs\unlhd x}  ( v^{l(x)}\grk\, M_{[x]} )(W_{xs}+v^{-1} W_x) \\
&& + \sum_{x\in\hCW, x\unlhd xs}  (v^{l(x)}\grk\, M_{[x]})(W_{xs}+vW_{x}) \\
& = & \sum_{x\in\hCW,xs\unlhd x}  (v^{l(x)-1}\grk\, M_{[x]} +v^{l(x)-1}\grk\, M_{[xs]}) W_x \\
&& + \sum_{x\in\hCW, x\unlhd xs}  (v^{l(x)+1}\grk\, M_{[x]}+v^{l(x)+1}\grk\, M_{[xs]})W_x\\ 
& = & h_{\unlhd,l}(\theta^s M\langle 1 \rangle).
\end{eqnarray*}
Hence the proposition is proven.
\end{proof}
In the following we consider more closely the character maps
associated to the Bruhat order and the generic Bruhat order.

\subsection{The Bruhat order and the affine Hecke algebra} \label{Bruord}
In the following we will denote by ``$\varle$'' the Bruhat order on $\hCW$ with respect to $\hCS$. It has all the properties listed in Section \ref{partord} (for property $(3)$, see \cite[Proposition 5.9]{Hum}). Denote by $l\colon\hCW\to\DN$ the associated length function. 

Let $\bH=\bigoplus_{x\in\hCW} \DZ[v,v^{-1}]\cdot T_x$ be the affine Hecke algebra. Its multiplication is given by the formulas
\begin{eqnarray*}
{ T}_x\cdot { T}_{y} & = & { T}_{xy}\quad\text{if $l(xy)=l(x)+l(y)$}, \\
{ T}_s^2 & = & v^{-2}T_e+(v^{-2}-1)T_s \quad\text{for $s\in\hCS$}.
\end{eqnarray*}
Then ${ T}_e$ is a unit in $\bH$ and for any $x\in\hCW$ there exists an inverse of ${ T}_x$ in $\bH$. For $s\in\hCS$ we have ${ T}_s^{-1}=v^2T_s+(v^2-1)$. There is a duality (i.e.~a $\DZ$-linear involution) $d\colon\bH\to\bH$, given by $d(v)=v^{-1}$ and $d(T_x)  =  T_{x^{-1}}^{-1}$ for $x\in\hCW$. 

Set $\tilde T_x:=v^{l(x)}T_x$. Recall the following result:
\begin{theorem}[\cite{MR560412,MR1445511}]\label{self-dual elts} 
For any $x\in\hCW$ there exists a unique element $\ul{H}_x=\sum_{y\in\hCW} h_{y,x}(v)\cdot { \tilde T}_y\in\bH$ with the following properties:
\begin{enumerate}
\item \label{prop of H: duality} $\ul{H}_x$ is self-dual, i.e.~$d(\ul{H}_x)=\ul{H}_x$. 
\item \label{prop of H: support} $h_{y,x}(v)=0$ if $y\not\leq x$, and $h_{x,x}(v)=1$,
\item \label{prop of H: norm} $h_{y,x}(v)\in v\DZ[v]$ for $y<x$.
\end{enumerate}
\end{theorem} 
For example, we have $\ul{H}_e=\tilde T_e$ and $\ul{H}_s=\tilde T_s+v\tilde T_e$ for each $s\in\hCS$. 

Now we identify the free $\DZ[v,v^{-1}]$-modules $\bH$ and $\bW$ by means of the map $\tilde T_x\mapsto W_x$ for all $x\in\hCW$. In the following we will consider $h_{\varle,l}$ as a character map from $\CH_k$ to $\bH$. A short and simple calculation yields the following result.
\begin{lemma}\label{hvarle} Choose $s\in\hCS$. Under the above identification the right multiplication with $\ul{H}_s$ on $\bH$ corresponds to the map $\rho_{s,\varle}\colon \bH\to\bH$. Hence 
$$
h_{\varle,l}(\theta^s M\langle 1\rangle)=h_{\varle,l}( M)\cdot \ul{H}_s
$$
for all $M\in\CH_k$.
\end{lemma}

\subsection{The generic Bruhat order and the periodic Hecke module}\label{Brinf}
Denote by ``$\succeq$'' the {\em generic} Bruhat order on $\hCW$. It can be
defined as follows. For $x,y\in\hCW$ set $x\succeq y$ if
$t_{\lambda}y\varle t_{\lambda}x$ for sufficiently large 
$\lambda\in \DN R^{\vee,+}$.  This order has all the properties we listed at the beginning of this section (for property (3), see \cite[Proposition 3.2]{MR591724}).  Let $\delta\colon\hCW\to\DZ$ be a length function, normalized such that
 $\delta(e)=0$. 
 
 Let $\bM$ be the free $\DZ[v,v^{-1}]$-module with basis $\{A_x\}_{x\in\hCW}$, and identify $\bM$ with $\bW$ by means of the map $A_x\mapsto W_x$.

\begin{theorem} \cite{MR1445511}\label{Mstruc}
 There is a unique right $\bH$-module structure on $\bM$, denoted by
 $\ast\colon \bM\times\bH\to\bM$, such that for all $s\in\hCS$ the right multiplication with
 $\ul{H_s}$ is given by the map $\rho_{s,\succeq}\colon
 \bM\to\bM$, i.e.~
$$
X\ast \ul{H_s} = \rho_{s,\succeq}(X)
$$
for all $X\in \bM$ and $s\in\hCS$. 
\end{theorem}

In the following we will consider $h_{\succeq,\delta}$ as a character map from $\CH_k$ to $\bM$. From Proposition \ref{chartrans} and Theorem \ref{Mstruc} we deduce the following:
\begin{lemma}\label{hsucceq} For all $M\in\CH_k$  we have
$$
h_{\succeq,\delta}(\theta^s M\langle 1\rangle)=h_{\succeq,\delta}(M)\ast \ul{H}_s.
$$
\end{lemma}

The following theorem is part of our main result.
\begin{theorem}\label{charmaps} The following diagram commutes:

\centerline{
\xymatrix{
& \CH_k \ar@{-->}[ld]_-{h_{\varle,l}}\ar@{-->}[dr]^-{h_{\succeq,\delta}}\\
 \bH \ar[rr]^-{A_e\ast\cdot}& & \bM
}
} 
\end{theorem}

\begin{proof} We have $h_{\succeq,\delta}(B_e)=A_e=A_e\ast
  \ul{H}_e=A_e\ast h_{\varle,l}(B_e)$. Suppose that $h_{\succeq,\delta}(M)=A_e\ast
  h_{\varle,l}(M)$ for $M\in\CH_k$ and choose $s\in\hCS$. Then Lemma \ref{hvarle} and
  Lemma \ref{hsucceq} yield
\begin{eqnarray*}
h_{\succeq,\delta}(\theta^s M\langle 1\rangle) & = &
h_{\succeq,\delta}(M) \ast \ul{H}_s \\
& = & (A_e\ast h_{\varle,l}(M))\ast  \ul{H}_s  \\
& = & A_e\ast(h_{\varle,l}(M))\cdot  \ul{H}_s )\\
& = & A_e\ast h_{\varle,l}(\theta^s M\langle 1\rangle).
\end{eqnarray*}
The theorem follows by induction.
\end{proof}

\section{The category of Andersen, Jantzen and Soergel}\label{sec-AJS}

In this section we define the combinatorial category $\CM$ that
appears in the work of Andersen, Jantzen and Soergel on  Lusztig's conjecture (cf.~\cite{MR1272539,MR1357204}). Its
definition is very similar to the definition of $\CH$: It is
constructed by applying translation functors to a unit object $P_0$ inside a certain category $\CK$. 
We construct a functor $\Psi\colon \CH\to\CK$ with
$\Psi(B_e)\cong P_0$ and show that $\Psi$ intertwines the translation
functors on both sides. From this we deduce that the image of $\Psi$
is contained in $\CM\subset\CK$. 

The translation functors that we define on $\CK$ differ from  those
used in \cite{MR1272539}. In Section \ref{sec-AJStrans} we construct an autoequivalence on $\CK$ that
intertwines both sets of translation functors (Theorem \ref{twtrans}), hence both sets lead to equivalent categories of special objects.

\subsection{Alcoves and walls}

Note that in the constructions so far the affine Weyl group $\hCW$ appears in many respects merely as an index set (more precisely, as an index set acted upon by a set of involutions $\hCS$). For the following it is convenient to replace $\hCW$ by the set $\SA$ of  alcoves for the affine action of $\hCW$ on $V^\ast$. So before we state the definition of $\CK$, we collect some basic facts and notions from alcove geometry.

For  $\alpha\in R^+$ and $n\in \DZ$ define
$$
H_{\alpha,n}^+:=\{v\in V^\ast\mid \langle \alpha,v\rangle>n\},\quad 
H_{\alpha,n}^-:=\{v\in V^\ast\mid \langle \alpha, v\rangle<n\}.
$$
Then $V^\ast=H_{\alpha,n}^-\dot\cup H_{\alpha,n}\dot\cup H_{\alpha,n}^+$. 
 Consider the coarsest partition of $V^\ast$ that refines the set of all partitions thus obtained. The components of this partition are called {\em facets}. Let $\SF$ be the set of all facets. Facets that are open in $V^\ast$ are called {\em alcoves}.  Let $\SA\subset \SF$ be the set of alcoves. A facet $B$ that is a subset of exactly one hyperplane  $H_{\alpha,n}$ is called a {\em wall}. We set $\alpha_B:=\alpha$ in this case and call $\alpha_B$ the {\em type} of $B$. Let $\ScW\subset\SF$ be the set of walls. A wall that is a subset of the closure of an alcove  $A$ is called a
wall of $A$.

Let $C=\{v\in V^\ast\mid \langle\alpha,v \rangle > 0\text{ for all
  $\alpha\in R^+$}\}$ be the dominant chamber, and denote by
$A_e\subset C$ the fundamental dominant alcove, i.e.~the unique
alcove in $C$ whose closure contains the origin $0$ of  $V^\ast$.

The affine Weyl group acts on $\SF$ and leaves the subsets $\SA$ and $\ScW$ stable.  This induces a bijection
\begin{eqnarray*}
\hCW & \xrightarrow{\sim} & \SA \\
w & \mapsto & A_w:=w(A_e)
\end{eqnarray*}
whose inverse we denote by $A\mapsto w_A$. By transport of structure we get a right action of $\hCW$
on $\SA$. It is given explicitly by $(A_w).x=A_{wx}$.

The walls of $A_e$ form a fundamental domain for the $\hCW$-orbits in $\ScW$.  Moreover, the walls of $A_e$ are naturally in bijection with
the simple affine reflections $\hCS$: associate to $s\in\hCS$ the
unique wall $A_{e,s}$ of $A_e$ that lies in the closure of $s.A_e$ (this is the
wall stabilized by $s$). Denote by $\SA^s\subset\ScW$ its $\hCW$-orbit. We get another bijection
\begin{eqnarray*}
\hCW^s=\hCW/\{1,s\} & \xrightarrow{\sim} & \SA^s, \\
\ol w & \mapsto & \ol w(A_{e,s}).
\end{eqnarray*}

We denote by 
\begin{eqnarray*}
\SA&\to&\SA^s,\\
A & \mapsto & \ol A
\end{eqnarray*} 
the unique $\hCW$-equivariant map that sends the alcove $A_e$ to its wall $A_{e,s}$, and hence each alcove $A$ to its unique wall $\ol A$ in $\SA^s$. So the diagram

\centerline{
\xymatrix{
\hCW \ar[r]^{\sim} \ar[d] & \SA  \ar[d]\\
\hCW^s\ar[r]^{\sim}  & \SA^s.
}
} 
\noindent
commutes.
We define two maps that are splittings of $\SA\to \SA^s$. For $B\in\SA^s$ we denote by $B_-$ and $B_+$ the unique alcoves with wall $B$ such that $B_-\subset H_{\alpha,n}^-$ and
$B_+\subset H_{\alpha,n}^+$, where  $H_{\alpha,n}$ is the hyperplane containing $B$. 

Now choose $\beta\in R^+$.
For a facet ${F}\in \SF$ define $\beta\uparrow
{F}:=s_{\beta,n}.{F}$, where $n\in\DZ$ is the smallest integer such
that ${F}$ lies on or below $H_{\beta,n}$ (i.e.~in $H_{\beta,n}^-\cup
H_{\beta,n}$). Then $\beta\uparrow\cdot\colon \SF\to \SF$ is a
bijection whose inverse we denote by $\beta\downarrow\cdot$. 
Under our identification $\hCW\cong\SA$, an equivalence
class of alcoves under the relation generated by
$\beta\uparrow\cdot$ corresponds to a coset under the
left action of the subgroup $\hCW^\beta$ that is generated by all $s_{\beta,n}$ with $n\in\DZ$.  
 
In the following we list some statements which are easy to prove. \begin{lemma}\label{wallcomb} Let $s\in\hCS$ and $\beta\in R^+$. Choose $A\in\SA$ and $B\in \SA^s$. 
\begin{enumerate}
\item If $\beta\uparrow B=B$, then $\beta\uparrow B_-=B_+$. 
\item If $\beta\uparrow B\ne B$, then $\{\beta\uparrow B_-, \beta\uparrow B_+\}=\{(\beta\uparrow B)_-, (\beta\uparrow B)_+\}$.
\item If $\beta\uparrow \ol A=\ol A$ and $A=\ol A_-$, then $\ol{\beta\uparrow A}=\ol A$ and $\beta\uparrow A=\ol A_+$.
\item If $\beta\uparrow \ol A\ne \ol A$, then $\ol{\beta\uparrow A}=\beta\uparrow \ol A$.
\end{enumerate}
\end{lemma}

\subsection{The category of Andersen, Jantzen and Soergel} Let $k$ be a field of characteristic $\ne 2$, and suppose that its characteristic also differs from $3$ if $R$ is of type $G_2$. Recall that we defined $S_k$ as the symmetric algebra $S(X\otimes_\DZ k)$ over the $k$-vector space associated to $X$.  
We define the following sub-$S_k$-algebras of the quotient field of $S_k$: Set
$$
S_k^{\emptyset}:=S_k[\alpha^{-1}\mid\alpha\in R^+]\text{ and }S_k^{\beta}:=S_k[\alpha^{-1}\mid \alpha\in R^+, \alpha\ne\beta]
$$
for $\beta\in R^+$. 
These algebras naturally inherit a $\DZ$-grading. For an $S_k$-module $M$ set $M^{\beta}:=M\otimes_{S_k} S_k^{\beta}$ and $M^{\emptyset}:=M\otimes_{S_k} S_k^{\emptyset}$.

\begin{definition}[\cite{MR1272539,MR1357204}]\label{DefK} For a $\hCW$-orbit ${\SO}\subset\SF$  of facets let $\CK_k({\SO})$ be the category that consists of objects $
M=\left(\{M(F)\}_{F\in{\SO}}, \{M(F,\beta)\}_{F \in {\SO},\beta\in R^+}\right)$, where
\begin{enumerate}
\item  $M(F)$ is an $S_k^{\emptyset}$-module  for each $F \in {\SO}$ and
\item for $F \in{\SO}$ and $\beta\in R^+$, $M(F ,\beta)$ is an $S_k^\beta$-submodule of $M(F )\oplus M(\beta\uparrow F )$, if $\beta\uparrow F \neq F $, and of $M(F)$, if $\beta\uparrow F =F $.
\end{enumerate}
A morphism 
$f\colon M\to N$ in $\CK_k(\SO)$
is given by a collection $(f_F)_{F\in{\SO}}$ of homomorphisms $f_F\colon M(F)\to N(F)$ of  $S_k^\emptyset$-modules, such that for all $F\in{\SO}$ and $\beta\in R^+$, $f_F\oplus f_{\beta\uparrow F}$ ($f_F$, resp.) maps $M(F,\beta)$ into $N(F,\beta)$.
\end{definition}
 We define a degree shift functor $\langle n\rangle\colon\CK_k(\SO)\to\CK_k(\SO)$  for each $n\in \DZ$ by applying the degree shift to all constituents $M(F)$ and $M(F,\beta)$ for $M\in\CK_k(\SO)$.

\subsection{Translation functors on $\CK$}\label{sec-transonK}

We fix a simple reflection $s\in\hCS$ and write  $\CK_k$ for $\CK_k(\SA)$ and $\CK_k^s$ for $\CK_k(\SA^s)$. In this section we define translation functors  $\CTon=\CTson\colon\CK_k\to\CK_k^s$ and
$\CTout=\CTsout\colon\CK_k^s\to\CK_k$. Recall that for
an alcove $A$ we denote by $\ol A$ its unique wall that lies in the
$\hCW$-orbit $\SA^s$, and to each wall $B\in\SA^s$ we associate the
neighbouring alcoves $B_-$ and $B_+$. 

For $M\in\CK_k$, $B\in\SA^s$ and
$\beta\in R^+$ set
\begin{eqnarray*}
\CTon M({B}) &:=  &  M({B}_-)\oplus M({B}_+), 
\\
\CTon M({B},\beta) &:=  &
\begin{cases}
M({B}_-,\beta), &  \text{if $\beta\uparrow B=B$}, \\
M({B}_-,\beta)\oplus M({B}_+,\beta), & \text{if $\beta\uparrow B\ne B$}.
\end{cases} 
\end{eqnarray*}
If $\beta\uparrow B=B$ we have $\beta\uparrow B_-=B_+$, and if $\beta\uparrow B\ne B$ we have $\{(\beta\uparrow B)_-,(\beta\uparrow B)_+\}=\{\beta\uparrow B_-,\beta\uparrow B_+\}$. Hence we can naturally view $\CTon M({B}, \beta)$ as a subspace in $\CTon M(B)$ or $\CTon M(B)\oplus \CTon M(\beta\uparrow B)$, resp. 
 
For $N\in\CK_k^s$, $A\in\SA$ and $\beta\in R^+$ set
\begin{eqnarray*}
\CTout N({A}) &:=  &  N(\ol{A}), 
\\
\CTout N({A} ,\beta) &:=  &
\begin{cases}  
\left\{(\beta x+y,y)\mid x,y\in N(\ol{A},\beta)\right\}, & \text{if $\beta\uparrow \ol{A}=\ol{A}$}\\
&\quad\text{ and $A=\ol{A}_-$},  \\
\beta\cdot N(\ol{A},\beta)\oplus N(\ol{\beta\uparrow {A}},\beta), & \text{if $\beta\uparrow \ol{A}=\ol{A}$}\\
&\quad\text{and $A=\ol{A}_+$}, \\
N(\ol{A},\beta), & \text{if $\beta\uparrow \ol{A}\ne \ol{A}$}.
\end{cases}
\end{eqnarray*}
If $\beta\uparrow \ol A=\ol A$ and $A=\ol A_-$, then $\ol{\beta\uparrow A}=\ol A$. If $\beta\uparrow \ol A\ne \ol A$, then $\ol{\beta\uparrow A}= \beta\uparrow \ol A$. Hence $\CTout N(A,\beta)$ is naturally a subspace in $\CTout N(A)\oplus \CTout N(\beta\uparrow A)$ in each case.
We denote by  $\CT^s:=\CTsout\circ\CTson\colon\CK_k\to\CK_k$ the composition.

\subsection{Special objects in $\CK_k$}

Let $P_0\in\CK_k$ be the object defined by
\begin{eqnarray*}
P_0(A) & := & 
\begin{cases} S_k^\emptyset, & \text{if $A=A_e$}, \\
0, & \text{else},
\end{cases}
 \\
P_0(A,\beta) & := & 
\begin{cases} S_k^\beta, & \text{if $A=A_e$ or $A=\beta\downarrow A_e$},\\ 
0, & \text{else},
\end{cases}
\end{eqnarray*}
with the obvious inclusions $P_0(A,\beta)\subset P_0(A)\oplus P_0(\beta\uparrow A)$ for all $A\in\SA$ and $\beta\in R^+$.

\begin{definition}\label{defM}
\begin{enumerate}
\item
The category of {\em special objects in $\CK_k$} is the full subcategory
$\CM_k$ of $\CK_k$ that consists of all objects that are isomorphic to a
direct summand of a direct sum of objects of the form
$\CT^s\circ\dots\circ \CT^t(P_0)\langle n\rangle$ for arbitrary
sequences $s,\dots, t\in\hCS$ and $n\in\DZ$.
\item Let $s\in\hCS$. The category of {\em special objects in $\CK_k^s$} is the full subcategory $\CM_k^s$ of $\CK_k^s$ that consists of all objects that are isomorphic to a direct summand of $\CTson(M)$ for some $M\in\CM_k$.
\end{enumerate}
\end{definition}
By definition, the translation functors induce functors $\CTson\colon
\CM_k\to\CM_k^s$, $\CTsout\colon\CM_k^s\to\CM_k$ and $\CT^s\colon\CM_k\to\CM_k$.

\subsection{The functor $\Psi$}\label{sec-funcpsi}

Recall that $\CH_k$ is an $\hS_k$-linear category, while $\CM_k$ is an $S_k$-linear category. In order to construct a functor $\Psi\colon\CH_k\to\CM_k$, we need the homomorphism $\hS_k\to S_k$ of algebras that is induced by the projection $\ol{\cdot}\colon\hX=X\oplus\DZ\delta \to X$ onto the first direct summand. 

In Section \ref{sec-Loc} we defined the extension $\hS_k^\emptyset$ of $\hS_k$ by inverting all affine roots $\alpha_n$. Since $\ol\alpha_n=\alpha$ we get an induced map $\hS_k^\emptyset\to S_k^\emptyset$. Analogously, we get a map  $\hS_k^\beta\to S_k^\beta$ for each $\beta\in R^+$. For an $\hS_k$-module $M$ set $\ol M:=M\otimes_{\hS_k} S_k$. If $M$ was an $\hS_k^{\beta}$-module (an $\hS_k^\emptyset$-module), then $\ol M$ is an $S_k^\beta$-module (an $S_k^\emptyset$-module, resp.).

In Section \ref{sec-Fil} we constructed for each $M\in\CH_k$ a generic decomposition, i.e.~a canonical
decomposition $M^\emptyset=\bigoplus_{w\in\hCW} M^{\emptyset,w}$. Now
we replace $\hCW$, as an index set, by the set $\SA$ using the
identification $\hCW\xrightarrow{\sim}\SA$, $w\mapsto A_w=w(A_e)$.  Recall that by ``$\succeq$'' we denote the generic Bruhat order on $\hCW$ that we now consider as a partial order on $\SA$. The constructions
in Section \ref{sec-Fil} hence give us a functorial filtration on
$\CH_k$ by the partially ordered set $(\SA,\succeq)$, i.e.~for each
$M\in\CH_k$ and $A\in\SA$ we have the  submodule $M_{\succeq A}=M_{\{D\in\SA\mid D\succeq A\}}\subset
M$. Likewise,
we replace $\hCW^s$ by $\SA^s$ for each $s\in\hCS$, so we get a
functorial filtration on $\CH_k^s$ indexed by the partially ordered set
$(\SA^s,\succeq)$.

Now we define the functor $\Psi\colon\CH_k\to\CK_k$. 
For $M\in\CH_k$, $A\in\SA$ and $\beta\in R^+$ define 
\begin{eqnarray*}
\Psi M({A}) & = & \ol{M^{\emptyset,A}}, \\
\Psi M({A},\beta) & = & \im\left(\ol{M^\beta_{\succeq{A}}}\to \ol{M^{\emptyset,A}}\oplus \ol{M^{\emptyset,\beta\uparrow{A}}}\right).
\end{eqnarray*}
Here the map $\ol{M^\beta_{\succeq{A}}}\to \ol{M^{\emptyset,A}}\oplus \ol{M^{\emptyset,\beta\uparrow{A}}}$ is induced by the map 
$$
M^\beta_{\succeq A}\subset \bigoplus_{B\succeq A} M^{\emptyset, B}\to M^{\emptyset, A}\oplus M^{\emptyset, \beta\uparrow A},
$$
where the last map is the projection along the decomposition.

For $s\in\hCS$ we define  $\Psi^s\colon\CH_k^s\to\CK_k^s$ as follows.  For $N\in\CH_k^s$, $B\in\SA^s$ and $\beta\in R^+$ set 
\begin{eqnarray*}
\Psi^s N({B}) & = & \ol{N^{\emptyset,B}}, \\
\Psi^s N({B},\beta) & = & 
\begin{cases} 
\im\left(\ol{N^{\beta}_{\succeq{B}}}\to \ol{N^{\emptyset,B}}\oplus\ol{N^{\emptyset,\beta\uparrow{B}}}\right), & \text{ if $\beta\uparrow{B}\ne{B}$}, \\
\im\left(\ol{N^\beta_{\succeq{B}}}\to \ol{N^{\emptyset,B}}\right), & \text{ if $\beta\uparrow{B}={B}$}.
\end{cases}
\end{eqnarray*}
Again the maps that occur in the definition above are to be understood as the projections along the canonical decomposition.
The following theorem provides the most important step in the proof of Theorem \ref{theorem-MainTh}. 
\begin{theorem}\label{MTheo} 
\begin{enumerate}
\item For $M\in\CH_k$ we have $\Psi M\in \CM_k$.
\item For  $M\in\CH_k$ we have $\rk\, \Psi M (A)=\rk \,M^{\emptyset, A}$ for each alcove $A$.
\item The character map $h_{\succeq,\delta}\colon\CH_k   \dashrightarrow \bM$ factors over the functor $\Psi$ and we get a commutative diagram

\centerline{
\xymatrix{
\CH_k \ar@{-->}[d]_-{h_{\varle,l}} \ar[r]^-{\Psi} & \Psi(\CH_k) \ar@{-->}[d]^-{h_{\succeq,\delta}}\\
 \bH \ar[r]^-{A_e\ast\cdot} & \bM.
}
} 
\item Let $s\in\hCS$. For $N\in\CH_k^s$ we have $\Psi^s N\in\CM_k^s$ and the following diagrams commute:

\centerline{
\xymatrix{
\CH_k \ar[d]_-{\son} \ar[r]^-{\Psi} & \CM_k \ar[d]^-{\CTson}\\
 \CH_k^s \ar[r]^-{\Psi^s} & \CM_k^s,
}\quad  
\xymatrix{
\CH_k^s \ar[d]_-{\sout} \ar[r]^-{\Psi^s} & \CM_k^s \ar[d]^-{\CTsout}\\
 \CH_k \ar[r]^-{\Psi} & \CM_k.
}
}
\noindent
\end{enumerate}
\end{theorem}
\begin{remark}
In (2),  $\rk$ denotes the rank of an $S_k^\emptyset$- or $\hS_k^\emptyset$-module, resp.
\end{remark}
\begin{proof}
One immediately checks that $\Psi(B_e)\cong P_0$. So in order to prove claim (1) and claim (4) we only have to prove that $\Psi$ and $\Psi^s$ intertwine the translation functors. This we postpone for a moment. 

The claim (2) follows immediately from the definition. Assume that (1) and (4) are proven. Recall that in Section \ref{subsec-subquot} we defined  for $M\in\CH_k$ the subquotient $M_{[A,\succeq]}=\im(M_{\succeq
  A}\to M^{\emptyset,A})$. Now
$$
\ol{M^\beta_{[A,\succeq]}}=\im( \Psi(M)(A,\beta)\to \Psi(M)(A)).
$$
Hence we can recover $\ol{M^\beta_{[A,\succeq]}}$ from
the object $\Psi(M)$. Now recall that $M_{[A,\succeq]}$ is a graded free
$\hS_k$-module (by Lemma \ref{lemma-subquotfree}), so $\ol{M_{[A,\succeq]}}$ is a graded free $S_k$-module of
the same rank. But we can recover the latter module from the
collection  $\{\ol{M^\beta_{[A,\succeq]}}\}_{\beta\in R^+}$ (as $S_k=\bigcap_{\beta\in R^+} S_k^\beta$, cf.~\cite[Lemma 9.1]{MR1272539}), hence
from $\Psi(M)$, and we can deduce the first part of (3). The second part is then a consequence of Theorem \ref{charmaps}. 

Now we prove that $\Psi$ and $\Psi^s$ intertwine the  translation functors. Fix $M\in\CH_k$, $N\in\CH_k^s$, $A\in\SA$ and $B\in \SA^s$. By Lemma \ref{lemma-gendecomptrans} we have 
\begin{eqnarray*}
(\on M)^{\emptyset,B} & = & M^{\emptyset, B_-}\oplus M^{\emptyset, B_+}, \\
(\out N)^{\emptyset, A} & =& N^{\emptyset, \ol A}.
\end{eqnarray*}
Hence
\begin{eqnarray*} 
\Psi^s(\on M)(B) & = & \CTon(\Psi M)(B), \\
\Psi(\out N)(A) & = & \CTout(\Psi^s N)(A),
\end{eqnarray*}
which proves already the generic part of the  equations $\Psi^s\circ\on=\CTon\circ\Psi$ and $\Psi\circ\out=\CTout\circ\Psi^s$. 
In the following we fix $\beta\in R^+$  and determine  $\Psi\out N(A,\beta)$ and $\Psi^s\on M(B,\beta)$. We start with $\Psi^s\on M(B,\beta)$.

{\em The case $\beta\uparrow B\ne B$}:  We want to determine 
$$
\im\left((\on M)^\beta_{\succeq B}\to (\on M)^{\emptyset, B}\oplus (\on M)^{\emptyset, \beta\uparrow B} \right).
$$
Since the preimage of the set $\{C\in \SA^s\mid C\succeq B\}$ under the map $\SA\to \SA^s$ is $\{D\in\SA\mid D\succeq B_-\}$ we have 
$$
(\on M)^\beta_{\succeq B}=M^\beta_{\succeq B_-}.
$$
We have already shown that  
$$
(\on M)^{\emptyset, B}=M^{\emptyset, B_-}\oplus M^{\emptyset, B_+}
$$ 
and 
$$(\on M)^{\emptyset, \beta\uparrow B}=M^{\emptyset,(\beta\uparrow B)_-}\oplus M^{\emptyset,(\beta\uparrow B)_+}=M^{\emptyset,\beta\uparrow B_-}\oplus M^{\emptyset,\beta\uparrow B_+}
$$ 
(for the last identity we used Lemma \ref{wallcomb}). Hence we want to determine 
$$
\im\left(M^\beta_{\succeq B_-}\to M^{\emptyset, B_-}\oplus M^{\emptyset, \beta\uparrow B_-}\oplus M^{\emptyset, B_+}\oplus M^{\emptyset,\beta\uparrow B_+} \right).
$$

By Lemma \ref{lemma-decompreg} we have a canonical decomposition
$$
M^\beta  = \bigoplus_{\Theta\in \hCW^\beta\backslash\SA}M^{\beta,\Theta},
$$
where $M^{\beta,\Theta}\subset M^{\beta}$ is the direct summand supported on $\Theta$. The alcoves $B_-$ and $B_+$ lie in different $\hCW^\beta$-orbits, and hence the image above splits into the direct sum of 
$$
\im\left(M^\beta_{\succeq B_-}\to M^{\emptyset, B_-}\oplus M^{\emptyset, \beta\uparrow B_-} \right)
$$ 
and 
$$
\im\left(M^\beta_{\succeq B_+}\to M^{\emptyset, B_+}\oplus M^{\emptyset, \beta\uparrow B_+} \right).
$$
Here we used the fact that $B_+$ is the lowest alcove in the intersection of $\{D\in\SA\mid D\succeq B_-\}$ with the $\hCW^\beta$-orbit of $B_+$.

After inducing from $\hS_k$ to $S_k$, the first image is $\Psi M(B_-,\beta)$ and the second is $\Psi M(B_+,\beta)$, and we deduce
\begin{eqnarray*}
\Psi^s\on M(B,\beta) & = & \Psi M(B_-,\beta)\oplus \Psi M(B_+,\beta) \\
& = & \CTon \Psi M(B,\beta).
\end{eqnarray*}

{\em The case $\beta\uparrow B=B$}: This means that $\beta\uparrow B_-=B_+$. We want to determine
$$
\im\left((\on M)^\beta_{\succeq B}\to (\on M)^{\emptyset, B}\right).
$$
As before we have that $(\on M)^\beta_{\succeq B}=M^\beta_{\succeq B_-}$ and $ (\on M)^{\emptyset, B}=M^{\emptyset, B_-}\oplus M^{\emptyset, B_+}$. So we have
$$
\im\left((\on M)^\beta_{\succeq B}\to (\on M)^{\emptyset, B}\right)= \im\left( M^{\beta}_{\succeq B_-}\to M^{\emptyset, B_-}\oplus M^{\emptyset, B_+}\right).
$$
Hence $\Psi^s\on M(B,\beta) =  \Psi M(B_-,\beta)  =  \CTon \Psi M(B,\beta)$.

This finishes part of the proof: By now we showed that the functors $\Psi^s\circ \on$ and $\CTon \circ \Psi$ are isomorphic. Next we have to determine $\Psi\out N(A,\beta)$.

{\em The case $\beta\uparrow \ol A\ne \ol A$}: In this case the alcoves $A$ and $As$ lie in different $\hCW^\beta$-orbits. Let us denote by $\Theta\in\hCW^\beta\backslash \SA$ the orbit of $A$.
We want to determine
\begin{align*}
\im\left((\out N)^\beta_{\succeq A} \to (\out N)^{\emptyset, A}\oplus (\out N)^{\emptyset, \beta\uparrow A}\right) \\
=\im\left((\out N)^{\beta,\Theta}_{\succeq A} \to (\out N)^{\emptyset, A}\oplus (\out N)^{\emptyset, \beta\uparrow A}\right).
\end{align*}
By Lemma \ref{loctransone} we have $(\out N)^{\beta,\Theta}_{\succeq A}=N^{\beta,\ol\Theta}_{\succeq \ol A}$, where $\ol\Theta$ is the image of $\Theta$ in $\hCW^\beta\backslash \SA^s$. We already showed $(\out N)^{\emptyset, A}=N^{\emptyset,\ol A}$ and $(\out N)^{\emptyset, \beta\uparrow A}=N^{\emptyset,\ol{\beta\uparrow A}}$. We have $\beta\uparrow\ol{A}=\ol{\beta\uparrow A}$ by Lemma \ref{wallcomb}, hence we can deduce
\begin{align*}
\im\left((\out N)^{\beta,\Theta}_{\succeq A}\to (\out N)^{\emptyset, A}\oplus (\out N)^{\emptyset, \beta\uparrow A}\right) & =  \im\left(N^{\beta,\ol\Theta}_{\succeq \ol A}\to N^{\emptyset, \ol A}\oplus N^{\emptyset,\beta\uparrow\ol A}\right)\\
& = \im\left(N^{\beta}_{\succeq \ol A}\to N^{\emptyset, \ol A}\oplus N^{\emptyset,\beta\uparrow\ol A}\right),
\end{align*}
so $\Psi\out N(A,\beta)=\Psi^s N(\ol A,\beta)=\CTout\Psi^s N(A,\beta)$. 

{\em The case $\beta\uparrow \ol A=\ol A$ and $A=\ol A_-$}: In this case, $\ol A=\ol{\beta\uparrow A}=\beta\uparrow\ol A$. We want to determine
$$
 \im\left((\out N)^\beta_{\succeq A}\to (\out N)^{\emptyset, A}\oplus (\out N)^{\emptyset, \beta\uparrow A}\right).
$$
The set $\{D\in\SA\mid D\succeq A\}$ is  $s$-invariant in the present case, so we can apply Lemma \ref{loctransone} and we get $(\out N)^\beta_{\succeq A}=\out(N^\beta_{\succeq \ol A})$. By Lemma \ref{lemma-decompreg} we have $\out(N^{\emptyset,\ol A})=(\out N)^{\emptyset, A}\oplus (\out N)^{\emptyset, \beta\uparrow A}$, hence
\begin{align*}
\im\left((\out N)^\beta_{\succeq A}  \to  (\out N)^{\emptyset, A}\oplus (\out N)^{\emptyset, \beta\uparrow A}\right)  
 &=  \im\left(\out(N_{\succeq\ol A}^\beta)\to \out(N^{\emptyset, \ol A})\right)\\
& =  \out\left(\im\left(N_{\succeq\ol A}^\beta\to N^{\emptyset, \ol A}\right)\right)
\end{align*}
since $\out$ is exact. Now
$$
 \out\left(\im\left(N_{\succeq\ol A}^\beta\to N^{\emptyset, \ol A}\right)\right)
 =  \left\{(\beta_n x+y,y)\mid x,y\in \im\left(N_{\succeq\ol A}^\beta\to N^{\emptyset, \ol A}\right)\right\},
$$
where $n$ is such that $s_{\beta,n}(A)=\beta\uparrow A$.
We deduce that
\begin{eqnarray*}
\Psi\out N(A,\beta) & = & \left\{(\beta x+y,y)\mid x,y\in \Psi^s N(\ol A,\beta)\right\} \\
& = & \CTout \Psi^s N( A,\beta).
\end{eqnarray*}

{\em The case $\beta\uparrow \ol A=\ol A$ and $A=\ol A_+$}: We want to determine
$$
 \im\left(\ol{(\out N)^\beta_{\succeq A}}\to \ol{(\out N)^{\emptyset,A}}\oplus\ol{(\out N)^{\emptyset,\beta\uparrow A}}\right).
$$
By Lemma \ref{lemma-decompreg} we have a decomposition
$$
(\out N)^\beta=\bigoplus_{\Theta\in\hCW^{\beta}\backslash \SA} (\out N)^{\beta,\Theta}.
$$
Denote by $\Theta$ the $\hCW^\beta$-orbit of $A$. It also contains $\beta\uparrow A$, so 
\begin{align*}
\im\left((\out N)^\beta_{\succeq A}\to (\out N)^{\emptyset, A}\oplus (\out N)^{\emptyset, \beta\uparrow A}\right)\quad\quad \\
 =\im\left((\out N)^{\beta,\Theta}_{\succeq A}\to (\out N)^{\emptyset, A}\oplus (\out N)^{\emptyset, \beta\uparrow A}\right).
\end{align*}

For notational convenience we set  $N^\prime:=N^{\beta,\ol\Theta}$, where $\ol \Theta$ is the image of $\Theta$ inside $\hCW^\beta\backslash\SA^s$. Then 
$$
(\out N)^{\beta,\Theta}=\out(N^{\beta,\ol\Theta})=\out N^\prime.
$$
Hence we want to determine 
$$
\im\left((\out N^{\prime})_{\succeq A}\to (\out {N^{\prime}})^{\emptyset, A}\oplus (\out {N^{\prime}})^{\emptyset, \beta\uparrow A}\right).
$$
Let us consider the inclusion
$$
(\out N^{\prime})_{\succeq  A}/(\out N^{\prime})_{\succeq \beta\uparrow A}\subset (\out N^{\prime})_{\succeq \beta\downarrow  A}/(\out N^{\prime})_{\succeq \beta\uparrow A}.\eqno{(\ast)}
$$
The right hand side is supported on $\{\beta\downarrow A, A\}$, and the left hand side is its subspace of elements supported on $\{A\}$.

Since $A=\ol A_+$, the sets $\{\succeq \beta\uparrow A\}$ and $\{\succeq \beta\downarrow A\}$ are $s$-invariant. Hence, by Lemma \ref{loctransone},
\begin{eqnarray*}
(\out N^\prime)_{\succeq \beta\downarrow A}  & = &  \out(N^\prime_{\succeq\ol{\beta\downarrow A}}),\\
(\out N^\prime)_{\succeq \beta\uparrow A}  & =  & \out(N^\prime_{\succeq\ol{\beta\uparrow A}}). 
\end{eqnarray*}
Using the exactness of the functor $\out$ we identify the inclusion $(\ast)$  with
\begin{eqnarray*}
(\out N^{\prime})_{\succeq  A}/(\out N^{\prime})_{\succeq \beta\uparrow A} & \subset &  \out\left(N^{\prime}_{\succeq\ol{\beta\downarrow A}}/N^\prime_{\succeq\ol{\beta\uparrow A}}\right) \\
& = & \hCZ\otimes_{\hCZ^{s}} N^\prime_{\succeq\ol{\beta\downarrow A}}/ N^\prime_{\succeq\ol{\beta\uparrow A}}.
\end{eqnarray*}

Now $N^\prime_{\succeq\ol{\beta\downarrow A}}/ N^\prime_{\succeq\ol{\beta\uparrow A}}$ is supported on $\{\ol A\}$, hence
\begin{eqnarray*}
\hCZ\otimes_{\hCZ^{s}} N^\prime_{\succeq\ol{\beta\downarrow A}}/ N^\prime_{\succeq\ol{\beta\uparrow A}} & = &  \hCZ(\{A,\beta\downarrow A\})\otimes_{\hS_k}N^\prime_{\succeq\ol{\beta\downarrow A}}/ N^\prime_{\succeq\ol{\beta\uparrow A}} \\
& = & \left\{(\beta_n x+y,y)\mid x,y\in N^\prime_{\succeq\ol{\beta\downarrow A}}/ N^\prime_{\succeq\ol{\beta\uparrow A}}\right\}.
\end{eqnarray*}
Here $n$ is such that $s_{\beta,n}(A)=\beta\downarrow A$. Recall that the inclusion $(\ast)$ is the inclusion of the subspace of elements supported on $\{A\}$. We conclude
$$
(\out N^{\prime})_{\succeq  A}/(\out N^{\prime})_{\succeq \beta\uparrow A}  =  (\beta_n,0)\cdot (\out N^{\prime})_{\succeq \beta\downarrow  A}/(\out N^{\prime})_{\succeq \beta\uparrow A}.
$$

 Let $z=(z_C)_{C\in\SA}\in\hCZ$ be an element with $z_{{\beta\downarrow A}}=0$, $z_{A}\equiv \beta\mod \delta$ and $z_{{\beta\uparrow A}}\equiv 0\mod \delta$. Such an element exists by Lemma \ref{dinZ}. From the last identity we deduce
$$
\ol{(\out N^\prime)_{\succeq A}}=\ol{(\out N^\prime)_{\succeq \beta\uparrow A}}+\ol{z\cdot (\out N^\prime)_{\succeq \beta\downarrow A}}.$$

We want to determine
\begin{eqnarray*}
 \im\left(\ol{(\out N^\prime)_{\succeq A}}\to \ol{(\out N^\prime)^{\emptyset,A}}\oplus\ol{(\out N^\prime)^{\emptyset,\beta\uparrow A}}\right).
\end{eqnarray*}
Now the image of $\ol{(\out N^\prime)_{\succeq \beta\uparrow A}}$ inside $\ol{(\out N^\prime)^{\emptyset,A}}\oplus\ol{(\out N^\prime)^{\emptyset,\beta\uparrow A}}$ is contained in the second summand, while the image of $\ol{z\cdot (\out N^\prime)_{\succeq \beta\downarrow A}}$ is contained in the first, since $z_{\beta\uparrow A}\equiv 0\mod\delta$. 

We have 
\begin{eqnarray*}
\im\left(\ol{(\out N^\prime)_{\succeq \beta\uparrow A}}\to \ol{(\out N^\prime)^{\emptyset,\beta\uparrow A}}\right) 
& =&  \im\left(\ol{N^\prime_{\succeq \ol{\beta\uparrow A}}}\to \ol{{N^\prime}^{\emptyset,\ol {\beta\uparrow A}}}\right) \\
& = & \im\left(\ol{N^\beta_{\succeq \ol{\beta\uparrow A}}}\to \ol{{N}^{\emptyset,\ol {\beta\uparrow A}}}\right) \\
& = & \Psi^s N(\ol{\beta\uparrow A},\beta)
\end{eqnarray*}
and
\begin{eqnarray*}
\im\left(\ol{z\cdot (\out N^\prime)_{\succeq \beta\downarrow A}}\to \ol{(\out N^\prime)^{\emptyset,A}} \right) &= & \beta\cdot\im\left( \ol{N^\prime_{\succeq \ol{A}}}\to \ol{{N^\prime}^{\emptyset,\ol{A}}}\right) \\
& = & \beta\cdot\im\left( \ol{N^\beta_{\succeq \ol{A}}}\to \ol{{N}^{\emptyset,\ol{A}}}\right) \\
& = & \beta\cdot\Psi^s N(\ol A,\beta).
\end{eqnarray*}
Hence we showed that 
\begin{eqnarray*}
\Psi\out N(A,\beta) & = &  \beta\cdot\Psi^s N(\ol A,\beta)\oplus \Psi^sN(\ol{\beta\uparrow A},\beta) \\
& = & \CTout \Psi^s N(A,\beta).
\end{eqnarray*}

We have now considered all possible cases and hence have constructed an isomorphism $\Psi\circ\out\cong \CTout\circ\Psi^s$ of functors. This also finishes the proof of our theorem.
\end{proof}

\subsection{The objects $M(A,\beta)\subset M(A)\oplus M(\beta\uparrow A)$}

In this section we provide a result which is needed in the proof of
Theorem \ref{twtrans}. 

Fix $\beta\in R^+$. 
Choose $A\in\SF$ and suppose that $\beta\uparrow A\ne A$. Let $\CK_k^\beta(A)$ be the category whose objects are given by the data $(M, M(A), M(\beta\uparrow A))$, where $M(A)$ and $M(\beta\uparrow A)$ are $S_k^\emptyset$-modules and $M$ is an $S_k^\beta$-submodule of $M(A)\oplus M(\beta\uparrow A)$. Let the morphisms in $\CK_k^\beta(A)$ be the obvious ones. 

Let $V^\beta_A\in\CK_k^\beta(A)$ be given by the inclusion $S_k^\beta\subset S_k^\emptyset\oplus 0$, let $V^\beta_{\beta\uparrow A}$ be given by the inclusion $S_k^\beta\subset 0\oplus S_k^\emptyset$, and let $P^\beta_A$ be given by the inclusion $\{(\beta x+y, y)\mid  x,y\in S_k^\beta\}\subset S_k^\emptyset\oplus S_k^\emptyset$.

If $\beta\uparrow A=A$ we define the category $\CK_k^{\beta}(A)$
likewise, with objects given by the data $(M,M(A))$. We have a
standard object $V^\beta_A\in\CK_k^{\beta}(A)$ given by the inclusion $S_k^\beta\subset S_k^\emptyset$.

For $M\in\CM_k$, $A\in\SA$ and $\beta\in R^+$ we can consider $M(A,\beta)\subset M(A)\oplus M(\beta\uparrow A)$ as an object in $\CK_k^\beta(A)$. If $s\in \hCS$, the same holds for $M\in\CM_k^s$ and $A\in\SA^s$. 

\begin{lemma}\label{L4}  Suppose $M\in\CM_k$ and $A\in\SA$. Then
  $M(A,\beta)$ is isomorphic to a direct sum of copies of $V^\beta_A$, $V^\beta_{\beta\uparrow A}$ and $P^\beta_A$. A similar statement holds for objects in $\CM_k^s$. 
\end{lemma}

\begin{proof} The statement is readily verified for $M=P_0$. If the lemma is true for $M$, then it is also true for each direct summand of $M$. Hence it is enough to show that if the lemma holds for $M$ and if $s\in\hCS$, then it also holds for $\CTson(M)$ ($\CTsout(M)$, resp.). Now this is easily verified after a quick look at the definition of translation functors. 
\end{proof}

\begin{lemma}\label{L3} Fix $M\in\CM_k$ and  $A\in\SA$. Choose invertible elements
  $a,b\in S_k^{\beta}$ such that $a\equiv b\mod\beta$. Then
  $(a,1)M({A},\beta)=(b,1)M({A},\beta)$. A similar statement holds for objects in $\CM_k^s$.
\end{lemma}
\begin{proof} By Lemma \ref{L4}, $M(A,\beta)$ is isomorphic to a direct sum of copies of $V^\beta_A$, $V^\beta_{\beta\uparrow A}$ and $P^\beta_A$. For $V^\beta_A$ and $V^\beta_{\beta\uparrow A}$ the claim is clear. For the case of $P_A^\beta$ we  observe that the inclusions
\begin{align*}
(a,1)\cdot \{(\beta x+y, y)\mid  x,y\in S_k^\beta\} &\subset S_k^\emptyset\oplus S_k^\emptyset \\ 
(b,1)\cdot \{(\beta x+y, y)\mid  x,y\in S_k^\beta\} &\subset S_k^\emptyset\oplus S_k^\emptyset
\end{align*}
coincide, since the left hand spaces are   generated by $(a,1)$ and $(\beta,0)$, and  by $(b,1)$ and $(\beta,0)$, resp., and since $a\equiv b\mod \beta$.
\end{proof}

\section{The translation functors of Andersen, Jantzen and Soergel}\label{sec-AJStrans}

The definition of the translation functors on $\CK$ that is given by Andersen, Jantzen and Soergel in \cite{MR1272539} is different from ours. In this section  we will show that both versions lead to equivalent categories of special objects. For this we construct  auto-equivalences ${\gamma}\colon \CK\to\CK$ and ${\gamma}^s\colon \CK^s\to\CK^s$ for each $s\in\hCS$ that intertwine the  sets of translation  functors.

First we recall the definition of Andersen, Jantzen and Soergel. Let $R^-=-R^+$ be the system of negative roots and fix $s\in\hCS$. For $A\in\SA$ and $\beta\in R^+$ we define an element $a_A^\beta\in S_k^{\emptyset}$ by 
$$
a_A^\beta:=
\begin{cases} -\alpha_{\ol{A}}, & \text{if $s_\beta(\alpha_{\ol{A}})\in R^-$, $A=\ol{A}_+$}, \\
 \alpha_{\ol{A}}^{-1}, & \text{if $s_\beta(\alpha_{\ol{A}})\in R^-$, $A=\ol{A}_-$},\\
1, & \text{if $s_\beta(\alpha_{\ol{A}})\in R^+$}.
\end{cases}
$$
(Recall that $\alpha_{\ol A}$ is the positive root associated to the
hyperplane that contains $\ol A$.)

In \cite{MR1272539} the  translation functors
on the wall $\CTon^\prime=\CTon^{\prime s}\colon\CK_k\to\CK_k^s$ and
out of the wall $\CTout^\prime=\CTout^{\prime s}\colon\CK_k^s\to\CK_k$  were defined as follows. For $M\in\CK_k$, $B\in\SA^s$ and $\beta\in R^+$ set
\begin{eqnarray*}
\CTon^\prime M({B}) & :=  & M({B}_-)\oplus M({B}_+),
\\
\CTon^\prime M({B},\beta) &:=  &
\begin{cases}
((a_{B_-}^\beta)^{-1},1) M({B}_-,\beta), & \text{if $\beta\uparrow B=B$}, \\
((a_{B_-}^\beta)^{-1},1) M({B}_-,\beta) \\
\quad \oplus ((a_{B_+}^\beta)^{-1},1) M({B}_+,\beta),  & \text{if $\beta\uparrow B\ne B$}.
\end{cases} 
\end{eqnarray*}
For  $N\in\CK_k^s$, $A\in\SA$ and $\beta\in R^+$ set
\begin{eqnarray*}
\CTout^\prime N({A}) & :=  &  N(\ol{A}),
\\
\CTout^\prime N({A} ,\beta) &:=  &
\begin{cases} 
\left\{\left. (x+a^\beta_{A}y, y)\right|  x,y\in N(\ol{A},\beta) \right\}, & \text{if $\beta\uparrow \ol{A}=\ol{A}$}\\
&\quad\text{and $A=\ol{A}_-$},  \\
 N(\ol{A},\beta)\oplus   N(\ol{\beta\uparrow {A}},\beta),
& \text{if $\beta\uparrow \ol{A}=\ol{A}$}\\
&\quad \text{ and $A=\ol{A}_+$},   \\
\left(a_{A}^\beta,1\right)N(\ol{A},\beta),
 & \text{if $\beta\uparrow \ol{A}\ne \ol{A}$}.
\end{cases}
\end{eqnarray*}
Set $\CT^{\prime s}:=\CTsout\circ\CTson$ and define $\CM_k^\prime\subset \CK_k$ in
analogy to $\CM_k$ (see Definition \ref{defM}), i.e.~by replacing the
translation functors $\CT^s$ by ${\CT}^{\prime s}$. The main result of this section is the following.
\begin{theorem}\label{AJScat}
There is an equivalence $\gamma\colon\CM_k\xrightarrow{\sim}\CM_k^\prime$ that
intertwines the translation functors on both sides, i.e.~for all $s\in\hCS$ the
following diagram commutes naturally:

\centerline{
\xymatrix{
\CM_k \ar[d]_-{\CT^s}\ar[r]^-{\gamma} & \CM_k^\prime  \ar[d]^-{{\CT}^{\prime s}}\\
\CM_k \ar[r]^-{\gamma} & \CM_k^\prime.
}
} 
\end{theorem}

The proof of the theorem will occupy the remainder of this section.

\subsection{Adding constants}

Let ${\SO}\subset\SF$ be an orbit of facets, and choose for each ${A}\in {\SO}$ and each $\beta\in R^+$ invertible scalars ${\gamma}_{{A}}^{\beta,-}, {\gamma}_{{A}}^{\beta,+}\in S_k^{\emptyset}$ if $\beta\uparrow{A}\ne{A}$, and an invertible scalar ${\gamma}_{{A}}^\beta\in S_k^{\emptyset}$ if $\beta\uparrow {A}={A}$. Denote by ${\gamma_{\SO}}\colon\CK_k({\SO})\to\CK_k({\SO})$ the following functor. For ${A}\in{\SO}$, $\beta\in R^+$  and $M\in \CK_k({\SO})$ define
\begin{eqnarray*}
{\gamma_\SO}M({A}) &:=& M({A}), \\
{\gamma_\SO}M({A},\beta) &:=& 
\begin{cases} ({\gamma}_{{A}}^{\beta,-}, {\gamma}_{{A}}^{\beta,+})\cdot M({A},\beta),  &  \text{if $\beta\uparrow {A}\neq {A}$}, \\
{\gamma}_{{A}}^\beta\cdot M({A},\beta), & \text{if $\beta\uparrow {A}={A}$}.
\end{cases}
\end{eqnarray*}
If $f=(f_{A})_{{A}\in{\SO}}\colon M\to N$ is a morphism in $\CK_k({\SO})$, then $f_{A}\oplus f_{\beta\uparrow A}$ ($f_A$, resp.) maps $M({A},\beta)$ into $N({A},\beta)$ if $\beta\uparrow A\ne A$ (if $\beta\uparrow A=A$, resp.) for each ${A}\in{\SO}$, $\beta\in R^+$. Hence  $f_{A}\oplus f_{\beta\uparrow A}$  ($f_A$, resp.) also maps $({\gamma}_{{A}}^{\beta,-}, {\gamma}_{{A}}^{\beta,+})\cdot M({A},\beta)$ (${\gamma}_{A}^\beta\cdot M(A,\beta)$, resp.) into $({\gamma}_{{A}}^{\beta,-}, {\gamma}_{{A}}^{\beta,+})\cdot N({A},\beta)$ (${\gamma}_{A}^\beta\cdot M(A,\beta)$, resp.), hence the collection $(f_{A})_{{A}\in{\SO}}$ also yields a morphism ${\gamma_\SO}f\colon {\gamma_\SO}M\to {\gamma_\SO}N$. So our choice of scalars indeed gives a functor ${\gamma_\SO}\colon\CK_k({\SO})\to\CK_k({\SO})$.

\begin{proposition} The functor ${\gamma_\SO}\colon\CK_k({\SO})\to\CK_k({\SO})$ is an equivalence of categories.
\end{proposition}
\begin{proof} The inverse functor ${\gamma_\SO}^{-1}$ is given by the inverse scalars.
\end{proof}

Now suppose we have chosen scalars as above for each alcove and each wall. This gives us functors $\gamma=\gamma_\SA\colon\CK_k\to\CK_k$ and $\gamma^s=\gamma_{\SA^s}\colon\CK_k^s\to\CK_k^s$ for each $s\in\hCS$. 
For each choice of scalars we obviously have $\gamma P_0\cong P_0$.
Define 
\begin{eqnarray*}
\CTon^{\gamma}  =\CTon^{\gamma s}  :=   {\gamma}^{s,-1}\circ \CTon^{\prime s}\circ {\gamma}& \colon & \CK_k\to\CK_k^s, \\
\CTout^{\gamma} =\CTout^{\gamma s} :=    {\gamma}^{-1}\circ \CTout^{\prime s}\circ {\gamma}^s & \colon & \CK_k^s\to\CK_k. 
\end{eqnarray*}
Theorem \ref{AJScat} follows from the next statement.

\begin{proposition}\label{twtrans} One can choose scalars as above for each alcove and each wall such that for all $s\in\hCS$ we have $\CTout^{\gamma s}\cong \CTsout$ and $\CTon^{\gamma s}\cong \CTson$ when restricted to $\CM_k$ and $\CM_k^s$, resp.
\end{proposition}

\subsection{A choice of constants}

For alcoves $A$ and $A^\prime$ with $A\preceq A^\prime$ define $H(A,A^\prime)$ as the set of all
reflection hyperplanes that separate $A$ and $A^\prime$, i.e.~as the
set of all $H_{\alpha,n}$ with $A\subset H_{\alpha,n}^-$ and
$A^\prime\subset H_{\alpha,n}^+$. 
Let $\beta\in R^+$. We will now define a non-zero constant
$d_{F}^\beta\in S_k^{\emptyset}$ for each facet $F$ that is an alcove or
a wall. Suppose first that $F$ is an alcove $A\in\SA$. Let
$\alpha_1,\dots,\alpha_n\in R^+$ be the positive roots associated to
the hyperplanes in $H(A,\beta\uparrow A)$, and set
$$
d_{A}^\beta:=\prod_{i=1,\dots,n\atop s_{\beta}(\alpha_i)\in R^-} \alpha_i^{-1}.
$$
For a wall ${B}\in\ScW$ set
$$
d_{B}^\beta:=
\begin{cases}
\alpha_B\cdot d_{B_-}^\beta, & \text{if $s_\beta(\alpha_B)\in R^-$,} \\
d_{B_-}^\beta, & \text{if $s_\beta(\alpha_B)\in R^+$}.
\end{cases}
$$
\begin{lemma}\label{Lemcon} Let $A\in\SA$ and $B\in\ScW$.
\begin{enumerate}
\item $d_{A}^\beta\in\beta^{-1} S_k^\beta$, and $\beta\cdot d_A^{\beta}$ is invertible in $S_k^\beta$.
\item If $B$ is of type $\beta$ (i.e.~$\beta\uparrow B=B$), then $d_{B_-}^\beta=\beta^{-1}$ and $d_B^{\beta}=1$.
\item If $B$ is not of type $\beta$ (i.e.~$\beta\uparrow B\ne B$), then 
$$
d_{B_+}^\beta=
\begin{cases} 
\alpha_B\cdot (-s_\beta(\alpha_B))\cdot d_{B_-}^\beta, & \text{if $s_\beta(\alpha_B)\in R^-$}, \\
d_{B_-}^\beta, & \text{if $s_\beta(\alpha_B)\in R^+$.}
\end{cases}
$$
\end{enumerate}
\end{lemma}
\begin{proof}
The alcoves $A$ and $\beta\uparrow A$ are separated by exactly one hyperplane of type $\beta$ and it follows that $d_{A}^\beta\in \beta^{-1} S_k^{\beta}$ and that $\beta\cdot d_A^{\beta}$ is invertible in $S_k^\beta$. Hence (1).

Only the hyperplane containing $B$ separates $B_-$ and $B_+$, and if it is of type
$\beta$, then $\beta\uparrow B_-=B_+$ and we deduce
$d_{B_-}^\beta=\beta^{-1}$ and $d_B^\beta=1$ directly from the definitions. Hence we proved claim (2).

Suppose  $B$ is not of type $\beta$. Then $\{\beta\uparrow B_-,\beta\uparrow B_+\}=\{(\beta\uparrow B)_-,(\beta\uparrow B)_+\}$. Denote by $H_1$ the reflection
hyperplane separating $B_-$ and $B_+$ and by $H_2$ the hyperplane separating $(\beta\uparrow B)_-$ and $(\beta\uparrow B)_+$. 
For each reflection hyperplane $H\ne H_1$ we
have $B_-\subset H^-$ if and only if $B_+\subset H^-$, and the analogous result holds  for $(\beta\uparrow B)_-$ and $(\beta\uparrow B)_+$ .

If $H_1=H_2$, then $B_-$ and $\beta\uparrow B_-$ lie on the same side of $H_1$ and we deduce $H(B_-,\beta\uparrow B_-)=
H(B_+,\beta\uparrow B_+)$, so $d_{B_-}^\beta=d_{B_+}^\beta$. Moreover, $s_\beta$ stabilizes $H_1$, so
$s_\beta(\alpha_B)=\alpha_B$ (since $\alpha_B\ne \beta$), so $s_\beta(\alpha_B)\in R^+$, hence $d_{B_+}^\beta=d_{B_-}^\beta$ is
what we claimed in (3).

If $H_1\ne H_2$, then $H(B_-,\beta\uparrow B_-)=H(B_+,\beta\uparrow B_+)\cup\{H_1,H_2\}$. By definition $H_1$ is of type $\alpha_B$. Then the type of $H_2$ is the positive root in $\{\pm s_\beta(\alpha_B)\}$. Suppose that  $s_\beta(\alpha_B)\in R^+$. Then for the calculation of $d_{B_-}^\beta$ we do not have to consider $\alpha_B$, nor $s_\beta(\alpha_B)$ (as $s_\beta(s_\beta(\alpha_B))=\alpha_B\in R^+$). Hence $d_{B_+}^\beta=d_{B_-}^\beta$ in this case. But if $s_\beta(\alpha_B)\in R^-$, then, as $s_\beta(-s_\beta(\alpha_B))=-\alpha_B\in R^-$, we have to consider both $\alpha_B$ and $-s_\beta(\alpha_B)$ for the calculation of $d_{B_-}^\beta$, hence $d_{B_-}^\beta=\alpha_B^{-1}\cdot(-s_\beta(\alpha_B))^{-1}\cdot d_{B_+}^\beta$.
\end{proof}

\begin{proof}[Proof of Proposition \ref{twtrans}] Let ${\gamma}\colon \CK_k\to\CK_k$ and ${\gamma}^s\colon\CK_k^s\to\CK_k^s$ for $s\in\hCS$ be the functors associated to the following constants. For $A\in\SA$, $B\in\ScW$ and $\beta\in R^+$ define 

\begin{align*}
\gamma_A^{\beta,+} & := 1, & {\gamma}_{{A}}^{\beta,-} & :=d_A^\beta,   &\\
{\gamma}_{{B}}^{\beta,+}  &:=  1, &{\gamma}_{{B}}^{\beta,-} &:=d_{{B}}^\beta, & \text{ if $\beta\uparrow{B}\ne  {B}$}, \\
{\gamma}_{{B}}^\beta & :=1, &&&\text{ if $\beta\uparrow{B}=  {B}$}.
\end{align*} 

Now fix $s\in\hCS$ and choose $M\in\CM_k$,  $N\in\CM_k^s$,  $A\in\SA$ and $B\in\SA^s$. Then 
\begin{center}
\begin{tabular}{rcccl}
$\CTon^\gamma M(B)$&$ =$&$ M(B_-)\oplus M(B_+)$&$= $&$\CTon M(B)$, \\
$\CTout^{\gamma} N(A) $&$=$&$ N(\ol{A})$&$= $&$\CTout N(A) $.
\end{tabular}
\end{center}
In order to prove the proposition it is enough to show that the subspaces associated to $\beta\in R^+$ coincide.

{\em The case $\beta\uparrow B=B$}. Then $\alpha_B=\beta$, $a_{B_-}^\beta=\beta^{-1}$, ${\gamma}_{B_-}^{\beta,-}=\beta^{-1}$, ${\gamma}^{\beta,+}_{B_-}={\gamma}_B^{\beta}=1$. Hence
\begin{eqnarray*}
\CTon^{\gamma} M(B,\beta) & = & \left(\frac{ {\gamma}^{\beta,-}_{B_-}}{a_{B_-}^\beta\cdot{\gamma}^\beta_B},\frac{{\gamma}^{\beta,+}_{B_-}}{{\gamma}^\beta_B}\right) M(B_-,\beta)\\
& = &  M(B_-,\beta) =  \CTon M(B,\beta).
\end{eqnarray*}

{\em The case $\beta\uparrow B\ne B$ and $s_{\beta}(\alpha_B)\in R^-$}. Then $a_{B_-}^\beta=\alpha_B^{-1}$,  $a_{B_+}^\beta=-\alpha_B$, ${\gamma}_{B_-}^{\beta,-}=d_{B_-}^\beta$, ${\gamma}_{B_+}^{\beta,-}=d_{B_+}^\beta$, ${\gamma}_B^{\beta,-}=d_B^\beta=\alpha_B\cdot d_{B_-}^{\beta}$ and ${\gamma}_{B_-}^{\beta,+}={\gamma}_{B_+}^{\beta,+}={\gamma}_B^{\beta,+}=1$. Moreover, $d_{B_+}^\beta=\alpha_B\cdot(-s_\beta(\alpha_B))\cdot d_{B_-}^\beta$ by Lemma \ref{Lemcon}. Hence
\begin{eqnarray*}
\CTon^{\gamma} M(B,\beta) & = & \left(\frac{ {\gamma}^{\beta,-}_{B_-}}{a_{B_-}^\beta{\gamma}_B^{\beta,-}},\frac{{\gamma}^{\beta,+}_{B_-}}{{\gamma}^{\beta,+}_B}\right) M(B_-,\beta)\oplus \left(\frac{ {\gamma}^{\beta,-}_{B_+}}{a^\beta_{B_+} {\gamma}_B^{\beta,-}},\frac{{\gamma}^{\beta,+}_{B_+}}{{\gamma}^{\beta,+}_B}\right) M(B_+,\beta) \\
& = &  \left(\frac{d_{B_-}^\beta }{\alpha_B^{-1}\alpha_B d_{B_-}^\beta},1\right) M(B_-,\beta)\oplus \left(\frac{ d_{B_+}^\beta}{-\alpha_B \alpha_B d_{B_-}^\beta},1\right) M(B_+,\beta) \\
& = & M(B_-,\beta)\oplus \left(\frac{s_\beta(\alpha_B)}{\alpha_B},1\right)  M(B_+,\beta).
\end{eqnarray*}
Now $s_{\beta}(\alpha_B)\equiv \alpha_B\mod\beta$, hence $s_{\beta}(\alpha_B)\cdot \alpha_B^{-1}\equiv 1\mod\beta$. By Lemma \ref{L3}, $\left(\frac{s_\beta(\alpha_B)}{\alpha_B},1\right)  M(B_+,\beta)=M(B_+,\beta)$, hence $\CTon^{\gamma} M(B,\beta)=\CTon M(B,\beta)$.

{\em The case $\beta\uparrow B\ne B$ and $s_{\beta}(\alpha_B)\in R^+$}. Then $a_{B_-}^\beta=a_{B_+}^\beta=1$, ${\gamma}_{B_-}^{\beta,-}=d_{B_-}^\beta$, ${\gamma}_{B_+}^{\beta,-}=d_{B_+}^\beta$, ${\gamma}_B^{\beta,-}=d_B^\beta= d_{B_-}^{\beta}$ and ${\gamma}_{B_-}^{\beta,+}={\gamma}_{B_+}^{\beta,+}={\gamma}_B^{\beta,+}=1$. Moreover, $d_{B_+}^\beta=d_{B_-}^\beta$ by Lemma \ref{Lemcon}. Hence
\begin{eqnarray*}
\CTon^{\gamma} M(B,\beta) & = & \left(\frac{ {\gamma}^{\beta,-}_{B_-}}{a_{B_-}^\beta{\gamma}_B^{\beta,-}},\frac{{\gamma}^{\beta,+}_{B_-}}{{\gamma}^{\beta,+}_B}\right) M(B_-,\beta)\oplus \left(\frac{ {\gamma}^{\beta,-}_{B_+}}{a^\beta_{B_+} {\gamma}_B^{\beta,-}},\frac{{\gamma}^{\beta,+}_{B_+}}{{\gamma}^{\beta,+}_B}\right) M(B_+,\beta) \\
& = &  \left(\frac{d_{B_-}^\beta }{d_{B_-}^\beta},1\right) M(B_-,\beta)\oplus \left(\frac{ d_{B_+}^\beta}{d_{B_-}^\beta},1\right) M(B_+,\beta) \\
& = & M(B_-,\beta)\oplus M(B_+,\beta)=\CTon M(B,\beta).
\end{eqnarray*}
Hence we proved the proposition for $\CTon$.

{\em The case $\beta\uparrow \ol{A}=\ol{A}$ and $A=\ol{A}_-$}. Then $\alpha_{\ol{A}}=\beta$, $a_{A}^\beta=\beta^{-1}$, ${\gamma}_{A}^{\beta,-}=\beta^{-1}$, ${\gamma}_{A}^{\beta,+}=1$ and ${\gamma}_{\ol{A}}^\beta=1$ by Lemma \ref{Lemcon}. Hence
\begin{eqnarray*}
\CTout^{\gamma} N(A,\beta) & = &  \left.\left\{\left(\frac{x+a_{A}^\beta y}{\gamma_A^{\beta,-}},\frac{ y}{\gamma_A^{\beta,+}}\right)\right| \, x,y\in \gamma_{\ol{A}}^\beta\cdot N(\ol{A},\beta)\right\} \\
& = & \left\{(\beta x+y,y)\mid x,y\in N(\ol{A},\beta) \right\}\\
& = & \CTout N(A,\beta).
\end{eqnarray*}

{\em The case $\beta\uparrow \ol{A}=\ol{A}$ and $A=\ol{A}_+$}. Then ${\gamma}_{A}^{\beta,-}=d_A^\beta$, ${\gamma}_A^{\beta,+}=1$ and ${\gamma}_{\ol{A}}^\beta={\gamma}_{\ol{\beta\uparrow A}}^\beta=1$ by Lemma \ref{Lemcon}. Hence
\begin{eqnarray*}
\CTout^{\gamma} N(A,\beta) & = &  \frac{{\gamma}_{\ol{A}}^\beta}{{\gamma}_A^{\beta,-}} N(\ol{A},\beta) \oplus \frac{{\gamma}_{\ol{\beta\uparrow A}}^\beta}{{\gamma}_A^{\beta,+}} N(\ol{\beta\uparrow A},\beta) \\
& = & (d_A^\beta)^{-1}\cdot N(\ol{A},\beta) \oplus N(\ol{\beta\uparrow A},\beta).
\end{eqnarray*}
Now $\beta\cdot d_A^\beta$ is invertible in $S_k^\beta$ by Lemma \ref{Lemcon}, hence $(d_A^\beta)^{-1}\cdot N(\ol{A},\beta)=\beta\cdot N(\ol{A},\beta)$, so $\CTout^{\gamma} N(A,\beta)=\CTout N(A,\beta)$.

{\em The case $\beta\uparrow \ol{A}\ne \ol{A}$ and $s_\beta(\alpha_{\ol{A}})\in R^+$}. Then $a_{A}^\beta=1$, ${\gamma}_A^{\beta,-}=d_A^\beta$, ${\gamma}_{\ol{A}}^{\beta,-}=d_{\ol{A}}^\beta=d_{\ol{A}_-}^\beta$ and ${\gamma}_A^{\beta,+}={\gamma}_{\ol{A}}^{\beta,+}=1$. Moreover,  $d_{\ol{A}_+}^\beta=d_{\ol{A}_-}^\beta$, hence in either case ($A=\ol A_+$ or $A=\ol A_-$) we have $d_A^\beta=d_{\ol A_-}^\beta$. Then

\begin{eqnarray*}
\CTout^{\gamma} N(A,\beta) & = &  \left(\frac{a_A^\beta\cdot{\gamma}_{\ol{A}}^{\beta,-}}{{\gamma}_A^{\beta,-}},\frac{{\gamma}_{\ol{A}}^{\beta,+}}{{\gamma}_A^{\beta,+}}\right)  N(\ol{A},\beta) \\
& = &  \left(\frac{d_{\ol{A}_-}^{\beta}}{d_A^{\beta}},1\right)  N(\ol{A},\beta) \\
& = &  \CTout N(A,\beta).
\end{eqnarray*}

{\em The case $\beta\uparrow \ol{A}\ne \ol{A}$ and $s_\beta(\alpha_{\ol{A}})\in R^-$}. Then  $a_{\ol{A}_-}^\beta=\alpha_{\ol{A}}^{-1}$,  $a_{\ol{A}_+}^\beta=-\alpha_{\ol{A}}$, ${\gamma}_A^{\beta,-}=d_A^\beta$, ${\gamma}_{\ol{A}}^{\beta,-}=d_{\ol{A}}^\beta=\alpha_{\ol{A}}\cdot d_{\ol{A}_-}^\beta$ and ${\gamma}_A^{\beta,+}={\gamma}_{\ol{A}}^{\beta,+}=1$. Moreover, $d_{\ol{A}_+}^\beta=\alpha_{\ol{A}}\cdot(-s_\beta(\alpha_{\ol{A}}))\cdot d_{\ol{A}_-}^\beta$ by Lemma \ref{Lemcon}. If $A=\ol{A}_-$, then 

\begin{eqnarray*}
\CTout^{\gamma} N(A,\beta) & = &  \left(\frac{a_A^\beta\cdot{\gamma}_{\ol{A}}^{\beta,-}}{{\gamma}_A^{\beta,-}},\frac{{\gamma}_{\ol{A}}^{\beta,+}}{{\gamma}_A^{\beta,+}}\right)  N(\ol{A},\beta) \\
& = &  \left(\frac{\alpha_{\ol{A}}^{-1}\cdot\alpha_{\ol{A}}\cdot d_{\ol{A}_-}^\beta}{d_A^{\beta}},1\right) N(\ol{A},\beta) \\
& = & N(\ol{A},\beta) = \CTout N(A,\beta).
\end{eqnarray*}
If $A=\ol{A}_+$, then 

\begin{eqnarray*}
\CTout^{\gamma} N(A,\beta) & = &  \left(\frac{a_A^\beta\cdot{\gamma}_{\ol{A}}^{\beta,-}}{{\gamma}_A^{\beta,-}},\frac{{\gamma}_{\ol{A}}^{\beta,+}}{{\gamma}_A^{\beta,+}}\right)  N(\ol{A},\beta) \\
 & = &  \left(\frac{-\alpha_{\ol{A}}\cdot\alpha_{\ol{A}}\cdot d_{\ol{A}_-}^{\beta}}{d_{\ol{A}_+}^\beta},1\right)  N(\ol{A},\beta) \\
& = &  \left(\frac{\alpha_{\ol{A}}}{s_\beta(\alpha_{\ol{A}})},1\right)  N(\ol{A},\beta).
\end{eqnarray*}
Now $s_{\beta}(\alpha_{\ol{A}})\equiv \alpha_{\ol{A}}\mod\beta$, hence $s_{\beta}(\alpha_{\ol{A}})^{-1}\cdot \alpha_{\ol{A}}\equiv 1\mod\beta$. By Lemma \ref{L3}, $(s_\beta(\alpha_{\ol{A}})^{-1}\cdot{\alpha_{\ol{A}}},1)  N(\ol{A},\beta)=N(\ol{A},\beta)$,  hence $\CTout^{\gamma} N(A,\beta)=\CTout N(A,\beta)$.
\end{proof}

\section{Sheaves on the affine flag variety}\label{sec-She}

In this section we introduce the affine flag variety $\hFl$ associated
to the  connected, simply connected complex algebraic group with root datum $R$. It carries an action of an extended torus $\hT$. We define for a field $k$ the category $\CI_k$ of ``special $\hT$-
equivariant sheaves'' of $k$-vector spaces on $\hFl$. (In the following we call a ``sheaf'' any element in the derived category of sheaves on a topological space.) To each $w\in\hCW$ one associates a Schubert cell $\CO_w\subset\hFl$ and its closure $\ol{\CO_w}\subset\hFl$, the corresponding Schubert variety. Note that each Schubert variety is a finite dimensional complex projective variety. We 
show that each special sheaf is supported inside  a Schubert variety. 

For the applications to modular representation theory, only a subcategory $\CI_k^{\circ}$ of $\CI_k$ with finitely many indecomposable isomorphism classes (up to shifts) plays a role. We use the decomposition theorem to show that,
if $k$ is a field with zero or big enough characteristic, the
indecomposable objects in $\CI_k^\circ$ are, up to a shift, the equivariant intersection
cohomology sheaves $\IC_{\hT,y}$ on the Schubert varieties
$\ol{\CO_y}$, where  $y$ comes from a finite subset $\hCW^\circ$ of $\hCW$. By a classical theorem of Kazhdan and
Lusztig, the ``characters'' of these sheaves are known.

Then we consider the equivariant hypercohomology functor $\Hyp^\ast_{\hT}$ and show
that it can be considered as a functor from $\CI_k$ to $\CH_k$ and
that the characters of sheaves yield characters of objects in $\CH_k$.

In the following section we state some theorems about the structure of
affine flag varieties. For details we refer to Kumar's book 
\cite{MR1923198}.

\subsection{The affine flag variety}

Let $G$ be a connected simply connected complex algebraic
group. Let $T\subset B\subset G$ be a maximal torus and a Borel
subgroup and suppose that $R$ is the root system of $G$ and that
$R^+\subset R$ is the set of roots of $B$. 

Denote by $G((t)):=\{f\colon \Spec\, \DC((t))\to G\}$ the associated
loop group, and by $G[[t]]:=\{f\colon \Spec\, \DC[[t]]\to G\}\subset G((t))$ the subgroup of loops that extend to zero. Let $I\subset
G[[t]]$ be the Iwahori subgroup, i.e.~the preimage of $B$ under the
evaluation map $G[[t]]\to G$, $t\mapsto 0$. The set-theoretic quotient
$\hFl=G((t))/I$ carries a natural structure of an ind-variety and is
called  the
affine flag variety. 

To a simple affine reflection $s\in\hCS$ corresponds a minimal
parabolic subgroup $P_s\subset G((t))$ that contains $I$. The
set theoretic quotient $\hFl_s:=G((t))/P_s$ carries again an ind-variety
structure and is called a partial affine flag variety. The natural map
$\pi_s\colon \hFl\to\hFl_s$  is an ind-proper map of ind-varieties.

The set of $I$-orbits in $\hFl$ can  canonically be identified with the affine Weyl
group $\hCW$. Denote the orbit corresponding to $w\in\hCW$
by $\CO_w\subset\hFl$. Its Zariski closure $\ol{\CO_w}$ is a finite
dimensional, projective variety and is called a Schubert variety. It
is the union of the orbits $\CO_y$ for all $y\le w$. The
set of $I$-orbits in $\hFl_s$ can analogously be identified with
$\hCW^s$, and for each $\ol w\in\hCW^s$ we denote the corresponding
orbit by 
$\CO_{\ol w}$. Its closure $\ol{\CO_{\ol w}}$ is called a partial
Schubert variety.  We have $\pi_s(\CO_w)={\CO_{\ol w}}$ and $\pi_s^{-1}({\CO_{\ol w}})=\CO_w\cup\CO_{ws}$ for all $w\in\hCW$.

Let $\hT=T\times \DC^\times$ be the extended torus, and define a
$\hT$-action on $G((t))$ by letting the first factor act by left
multiplication and the second by rotating the loops. Then $\hT$ also
acts on $\hFl$ and on $\hFl_s$ for all $s\in\hCS$. Each $\CO_w$ contains a unique $\hT$-fixed point $\pt_w\in\CO_w$. Denote by $i_w\colon \{\pt_w\}\to\hFl$ its inclusion.

\subsection{The equivariant cohomology of Schubert varieties}

Let $w\in\hCW$, and denote by $i\colon \ol{\CO_w}^{\hT}\to\ol{\CO_w}$
the inclusion of the set of $\hT$-fixed points. Then
$\ol{\CO_w}^{\hT}$ can  be identified  with the
set $\{\le w\}:=\{x\in\hCW\mid x\le w\}$.  In particular, it is finite. Now identify the character lattice $\Hom(\hT,\DC^\times)$ of the extended torus with the extended character lattice $\hX=X\oplus \DZ$. Then  $H^\ast_{\hT}(\pt,\DZ)=S(\hX)=\hS_\DZ$. Consider the ``localization map'' on equivariant topology,
$$
i^\ast\colon H^\ast_{\hT}(\ol{\CO_w},\DZ)\to
H^\ast_{\hT}(\ol{\CO_w}^{\hT},\DZ)=\bigoplus_{x\le w}\hS_\DZ.
$$
The following statement can
be deduced from the results in Chapter XI of \cite{MR1923198}.

\begin{theorem} Let $w\in\hCW$. 
\begin{enumerate}
\item The localization map is injective and induces an isomorphism 
$$
H_{\hT}^\ast(\ol{\CO_w},\DZ)=\left\{(z_x)\in\bigoplus_{x\le w} \hS_\DZ\left|
\,
\begin{matrix} z_x\equiv z_{tx}\mod\alpha_t \\ 
\text{for all $t\in\hCT$, $x\le w$  with $tx\le w$}
\end{matrix}\right.\right\}.
$$
\item Let $s\in\hCS$ and suppose $ws<w$. The pull-back map $\pi_s^\ast$ on
  equivariant cohomology induces an identification
$$
H_{\hT}^\ast(\ol{\CO_{\ol w}},\DZ)=H_{\hT}^\ast(\ol{\CO_w},\DZ)^s.
$$
\end{enumerate}
\end{theorem}
(In part (2) we denote by $H_{\hT}^\ast(\ol{\CO_w},\DZ)^s\subset H_{\hT}^\ast(\ol{\CO_w},\DZ)$ the subspace of $\sigma_s$-invariant elements, i.e.~the set of $(z_x)$ with $z_x=z_{xs}$ for all $x\le w$.)

Now fix a field $k$ and let $\hCZ_k$ be the structure algebra over the field $k$ introduced in Section  \ref{subsec-Struc}. Let $\hCZ_k(\{\le w\})$ be its finite version corresponding to  the set of vertices $\{x\in \hCW\mid x\le w\}$.
A quick look at the definition of the structure algebra  gives the following result:
\begin{corollary} \label{Zco} For each field $k$  we have natural identifications
$H_{\hT}^\ast(\ol{\CO_w},k)=\hCZ_k(\{\le w\})$ for all $w\in\hCW$ and $H_{\hT}^\ast(\ol{\CO_{\ol w}},k)= \hCZ_k(\{\le w\})^s$ for all $w\in\hCW$ with $ws<w$. 
\end{corollary}

\subsection{Special sheaves on $\hFl$}

Let again $k$ be a field  of characteristic $\ne 2$. For a topological space $Y$ which is acted upon continuously by $\hT$ we denote by $\dercat_{\hT}(Y,k)$ the $\hT$-equivariant derived category of
sheaves of $k$-modules with cohomology bounded from below. If $f\colon Y\to Y^\prime$ is a continuous map between such spaces that commutes with the $\hT$-actions, we denote by $f^\ast\colon\dercat_{\hT}(Y^\prime,k)\to\dercat_{\hT}(Y,k)$ and by $f_\ast\colon\dercat_{\hT}(Y,k)\to\dercat_{\hT}(Y^\prime,k)$ the equivariant pull-back and push-forward functors.

Denote by
$[n]\colon\dercat_{\hT}(\hFl,k)\to\dercat_{\hT}(\hFl,k)$ the shift
functor for $n\in\DZ$. Note that the support of a sheaf $F\in\dercat_{\hT}(\hFl,k)$, denoted by $\supp\, F$, is defined as the support of the underlying ordinary sheaf in $\dercat(\hFl,k)$, i.e.~the union of the supports of the cohomology sheaves of the latter.

Let $k_{e}\in  \dercat_{\hT}(\CO_e,k)$ be the constant equivariant sheaf of rank one on the point orbit $\CO_e$, and set $F_{e}:=i_{e\ast} k_{e}\in \dercat_{\hT}(\hFl,k)$.

\begin{definition} 
 The category of {\em special (equivariant) sheaves}  is the full
 subcategory $\CI_k$ of $\dercat_{\hT}(\hFl,k)$ that consists of all objects that are isomorphic to a direct summand of a direct sum of sheaves of the form $\pi_s^\ast\pi_{s\ast}\cdots \pi_t^\ast\pi_{t\ast}F_{e}[n]$  for arbitrary sequences $s,\dots,t \in\hCS$ and $n\in\DZ$.
 \end{definition}

In \cite{FW} we give a more intrinsic definition of this category as the category of {\em parity sheaves} on the affine flag manifold.
\begin{lemma} The support of each object of $\CI_k$ is contained in a
  Schubert variety, i.e.~for each $F\in\CI_k$ there exists $w\in\hCW$
  with $\supp\, F\subset\ol{\CO_w}$.
\end{lemma}
\begin{proof} This is certainly true for $F_{e}$. If it is true for $F$,
  then also for each isomorphic object, for $\pi_s^\ast\pi_{s\ast} F$, for $F[n]$ and for each  direct summand of $F$. If it is true for $F$ and $G$, then also for
  $F\oplus G$. The lemma follows. 
\end{proof}

\subsection{The hypercohomology of special sheaves}

 Let $F\in\CI_k$ and suppose that the support of $F$ is contained in $\ol{\CO_w}$. Then
the hypercohomology $\Hyp^\ast_{\hT}(F)$ is naturally a module over the
equivariant cohomology $H_{\hT}^\ast(\ol{\CO_w},k)$. The natural map
$\hCZ_k\to \hCZ_k(\{\le w\})\cong H_{\hT}^\ast(\ol{\CO_w},k)$ hence induces a $\hCZ_k$-module
structure on $\Hyp^\ast_{\hT}(F)$, so we can consider $\Hyp_{\hT}^\ast$
as a functor from $\CI_k$ to $\hCZ_k\catmod$.

\begin{proposition} Choose $F\in\CI_k$ and $s\in\hCS$. Then  there is a natural isomorphism 
$$
\hCZ_k\otimes_{\hCZ_k^s}\Hyp^\ast_\hT(F)\cong\Hyp^\ast_\hT(\pi_s^\ast\pi_{s\ast} F)
$$
of $\hCZ_k$-modules.
\end{proposition}
\begin{proof} Suppose that the support
  of $F$ is contained in $\ol{\CO_w}$ for $w\in\hCW$. Without loss of
  generality we can assume
  that $ws<w$.
 Then $\pi_s^\ast\pi_{s\ast} F$ is supported on $\ol{\CO_w}$ as well. As in \cite{MR1784005} one shows that we have an isomorphism
$$
H^\ast_{\hT}(\ol{\CO_w})\otimes_{H^\ast_{\hT}(\ol{\CO_{\ol w}})}\Hyp_{\hT}^\ast(F)=\Hyp^\ast_{\hT}(\pi_s^\ast\pi_{s\ast} F)
$$
of $H^\ast_{\hT}(\ol{\CO_w})$-modules. Now the statement follows from Corollary \ref{Zco}. 
\end{proof}

\begin{theorem} \label{speSheH} Equivariant hypercohomology gives a functor
$\Hyp^\ast_{\hT}\colon\CI_k\to\CH_k$ such that the following diagram commutes naturally for all $s\in\hCS$:

\centerline{
\xymatrix{
\CI_k \ar[d]_-{\pi_s^\ast\pi_{s\ast}}\ar[rr]^-{\Hyp_{\hT}^\ast} && \CH_k  \ar[d]^-{\theta^s}\\
\CI_k \ar[rr]^-{\Hyp_{\hT}^\ast} && \CH_k.
}
} 
\noindent
Moreover, we have 
$\rk\,\Hyp_{\hT}^\ast(i^\ast_x F)=\rk\, \Hyp^\ast_{\hT}( F)^{\emptyset, x}$
for each $F\in\CI_k$ and $x\in\hCW$. (By $\rk$ we denote the  rank of a free $\hS_k$-module or a free
$\hS^\emptyset_k$-module.)
\end{theorem}
\begin{proof}
The first part of the theorem follows from the preceding proposition
and the fact that $\Hyp^\ast_{\hT}(F_{e})=B_{e}$ as a $\hCZ_k$-module. 
Consider the localization  map
$
\Hyp^\ast_{\hT}(F)\to\bigoplus_{x\in\hCW}\Hyp^\ast_{\hT}(i^\ast_x F)
$.
Each $\Hyp^\ast_{\hT}(i^\ast_x F)$ is naturally an $\hS_k$-module, and
the right hand side of the map above carries the obvious action of
$\prod_{x\in\hCW}\hS$. Moreover, the map itself is compatible with the
actions of $\hCZ_k$ on the left hand side
and of $\prod_{x\in\hCW}\hS_k$ on the right, i.e.~it is a
$\hCZ_k$-module map. It induces an isomorphism
$$
\Hyp^\ast_{\hT}(F)\otimes_{\hS_k}\hS_k^\emptyset\stackrel{\sim}\to\bigoplus_{x\in\hCW}\Hyp^\ast_{\hT}(i^\ast_x F)\otimes_{\hS_k}\hS_k^\emptyset.
$$
and the second claim follows.
\end{proof}

\subsection{Equivariant intersection cohomology sheaves}

Let 
$\Pi^{-}\subset\SA$ be the set of alcoves that are contained in the strip
$H_{\alpha,0}^-\cap H_{\alpha,-1}^+$ for each simple root $\alpha$ (the set $-\Pi^-$ is also called the ``fundamental box''). Let $\hw_0\in\hCW$ be the element that corresponds to the smallest alcove $A_{\hw_0}$ in $\Pi^{-}$. Denote by $\SA^\circ\subset \SA$ the set of alcoves that correspond to the set $\hCW^{\circ}:=\{w\in\hCW\mid w\le \hw_0\}$.

\begin{definition}
\begin{enumerate}
\item Denote by $\CI_k^\circ\subset\CI_k$ the full
  subcategory that consists of direct sums of direct summands of
  objects of the form $\pi_t^\ast\pi_{t\ast}\circ\cdots\circ\pi_s^\ast\pi_{s\ast}
  F_{e}[n]$ for $s,\cdots, t\in\hCS$ such that $w=s\cdots t$ is a
  reduced expression for $w\in\hCW^\circ$.
\item Analogously, denote by $\CH_k^\circ\subset\CH_k$ the full subcategory that consists of direct sums of direct summands of
  objects of the form $\theta^t\circ\cdots\circ\theta^s B_{e}[n]$ for $s,\cdots, t\in\hCS$ such that $w=s\cdots t$ is a
  reduced expression for $w\in\hCW^\circ$.
\end{enumerate}
\end{definition}
Since by Theorem \ref{speSheH} equivariant hypercohomology commutes with the translation
functors on $\CI_k$ and $\CH_k$ we obtain an induced functor $\Hyp_{\hT}^\ast\colon\CI_k^\circ\to\CH_k^\circ$.

Now suppose that $k$ is a field and denote by $\IC_{\hT,y}\in\dercat_{\hT}(\hFl,k)$ the equivariant intersection cohomology complex on the Schubert variety $\ol{\CO_y}$ with coefficients in $k$. 
Choose $y\in\hCW^{\circ}$ and fix a reduced expression $y=s\cdots t$. If $\ch k=0$, then the decomposition theorem (cf.~\cite{MR751966}) together with some orbit combinatorics shows that 
$$
\pi_t^\ast\pi_{t\ast}\circ\cdots\circ\pi_s^\ast\pi_{s\ast}
  F_{e}\cong\bigoplus_i \IC_{\hT,y_i}[n_i]
 $$
 for some $y_i\in \hCW^\circ$ and $n_i\in\DZ$. Moreover, there is exactly one $i$ with $y_i=y$. By an equivariant version of a theorem of Kazhdan and Lusztig (cf.~\cite{MR573434}) we have
$\rk\,\Hyp^\ast_{\hT}(i^\ast_x\IC_{\hT,y})=h_{x,y}(1)$, again under the assumption $\ch k=0$.

Keep the reduced expression $s\cdots t$ fixed. Since the object
$\pi_t^\ast\pi_{t\ast}\circ\cdots\circ\pi_s^\ast\pi_{s\ast}F_{e}$ as
well as the intersection cohomology sheaves on Schubert varieties can
be defined over $\DZ$, we deduce that the above decomposition, as well as the character formula of Kazhdan and Lusztig, also holds if $\ch k$ is big enough, i.e.~bigger than a certain number $N$ depending on the reduced expression. 

Now in the definition of $\CI_k^\circ$ only finitely many reduced
expressions $s\cdots t$ occur. So we can deduce the following. 
\begin{theorem}\label{ICpos2} Suppose $\ch k$ is either zero or big enough, i.e.~bigger than a certain number $N$ depending on the root system $R$. Then the following holds.
\begin{enumerate}
\item  $\CI_k^\circ$ is the full subcategory of $\dercat_{\hT}(\hFl,k)$ that consists of objects isomorphic to a direct sum of shifted equivariant intersection cohomology sheaves $\IC_{\hT,y}$ with $y\in\hCW^\circ$. 
\item If $y\in\hCW^\circ$ and $y=s\cdots t$ is a reduced expression, then $\IC_{\hT,y}$ occurs, up to a shift, as the unique indecomposable direct summand in $\pi_t^\ast\pi_{t\ast}\cdots\pi_s^\ast\pi_{s\ast} F_{e}$ that is supported on $\ol{\CO_{y}}$. 
\item For $y\in\hCW^\circ$ and $x\in\hCW$ we have $\rk\, \Hyp_{\hT}^\ast(i_x^\ast \IC_{\hT,y})=h_{x,y}(1)$. 
\end{enumerate}
\end{theorem}

\section{Modular representations}\label{sec-mod}

Let $k$ be an algebraically closed field of characteristic $p>h$ and
let $\fg^\vee$ be the simple $k$-Lie algebra with root system
$R^\vee$. Let $\fh^\vee\subset\fb^\vee\subset\fg^\vee$ be a Cartan and
a Borel subalgebra.
In the following we recall the definition and the main properties of a certain category
$\CC_k$ of $X^\vee$-graded restricted representations of $\fg^\vee$ that
 is equivalent to the category of rational representations of the subgroup scheme $G_1^\vee T^\vee\subset G^\vee$ that is generated by the (first)
Frobenius kernel 
$G_1^\vee\subset G^\vee$ and the torus $T^\vee$. Our basic references
for the following are the books \cite{MR2015057} and \cite{MR1272539}.

\subsection{Restricted representations of $\fg^\vee$} 
Denote by $U(\fl)$ the universal enveloping algebra of a Lie algebra $\fl$. 
Recall that if $\fl$ is the Lie algebra of an algebraic group over a
field of characteristic $p>0$, then $\fl$ is a $p$-Lie algebra,
i.e.~it is endowed with a $p$-th power map $D\mapsto D^{[p]}$. For
each such Lie algebra let $U^{res}(\fl)$ be the restricted enveloping
algebra, i.e.~the quotient of
$U(\fl)$ by the two-sided ideal generated by all $D^p-D^{[p]}$ with
$D\in \fl$ (here $D^{p}$ denotes the $p$-th power of $D$ in $U(\fl)$).

Set $U^{res}:=U^{res}(\fg^\vee)$. Then $U^{res}$ is an $X^\vee$-graded Lie algebra, where the grading comes from the adjoint action of the torus $T^\vee$. Let $U^{res}\catmod_{X^\vee}$ be the category of $X^\vee$-graded $U^{res}$-modules $M=\bigoplus_{\nu\in X^\vee}M_\nu$.  For $\lambda\in \Hom(T^\vee,k^\times)$ we denote by $\ol \lambda\in\Hom(\fh^\vee,k)$ its differential. 
As in  \cite{MR1272539,MR1357204} we define the full subcategory $\CC_k$ of
$U^{res}\catmod_{X^\vee}$ by the following condition on its objects:
$$
\CC_k:=\left\{M\in U^{res}\catmod_{X^\vee}\left|
\,
\begin{matrix}
\text{$M$ is finitely generated,}\\
Hm=\ol\nu(H)m \\
\forall H\in\fh^\vee, \nu\in {X^\vee}, m\in M_\nu
\end{matrix}
\right.\right\}.
$$
It is known that this category is equivalent to the category of finite dimensional
rational representations of $G_1^\vee
T^\vee$ (cf.~\cite{MR2015057}). 

For each $\lambda\in {X^\vee}$ there is a standard module 
$$
Z(\lambda):=U^{res}(\fg^\vee)\otimes_{U(\fb^\vee)}k_\lambda
$$
in $\CC_k$, the  {\em baby Verma module} with highest weight
$\lambda$, which has a unique simple quotient $L(\lambda)$. The
$L(\lambda)$'s form a system of representatives for the simple objects
in $\CC_k$, and we denote by $\left[Z(\lambda):L(\mu)\right]$ the
multiplicity of $L(\mu)$ in a Jordan-Hölder series of
$Z(\lambda)$. This is a finite number, and for fixed $\lambda$ it is non-zero only for finitely many $\mu$.

\subsection{Deformed projective objects}\label{epo}

There is a projective cover $P(\mu)$ of each simple object
$L(\mu)$ in $\CC_k$. Each $P(\mu)$ admits a Verma flag, i.e.~a filtration
such that the subquotients are isomorphic to various $Z(\lambda)$. The corresponding multiplicity $(P(\mu): Z(\lambda))$ is independent of the filtration and we have the Brauer-Humphreys reciprocity formula $$
\left(P(\mu):Z(\lambda)\right)=\left[ Z(\lambda):L(\mu)\right].
$$

Recall that we defined the algebra
$S_k=S(X\otimes_\DZ k)$ that we can now identify with
$S(\fh^\vee)$. The  deformed version $\tCC_k$ of $\CC_k$ carries an 
action of the completion ${\tilde S_k}$ of $S_k$ at the maximal ideal
generated by $\fh^\vee\subset S_k$. We denote by $\tau\colon S_k\to\tilde
S_k$ the canonical map.  Following \cite[Section 1.2]{MR1272539}, we let $U$ be the quotient of
$U(\fg^\vee)$ by the twosided ideal generated by all $D^{[p]}-D^p$
with $D\in\fg^\vee_{\alpha^\vee}$ for some $\alpha^\vee\in R^\vee$
(note that $U^{res}$ is a quotient of $U$). We consider  $U\otimes_k
{\tilde S_k}$ as an $X^\vee$-graded ${\tilde S_k}$-algebra and define
$U\otimes_k {\tilde S_k}\catmod_{X^\vee}$ as the category of
$X^\vee$-graded $U\otimes_k {\tilde S_k}$-modules. Then 
$\tCC_k$ is the full subcategory of $U\otimes_k \tilde
S_k\catmod_{X^\vee}$ whose objects satisfy the following condition: 
$$
\tCC_k:=\left\{M\in U\otimes_k {\tilde S_k}\catmod_{X^\vee}\left|
\,
\begin{matrix}
\text{$M$ is finitely generated,}\\
(H\otimes 1)m=(\ol{\nu}(H)\otimes 1+1\otimes H)m \\
\forall H\in\fh^\vee, \nu\in {X^\vee}, m\in M_\nu
\end{matrix}
\right.\right\}.
$$

For $\lambda\in X^\vee$ let $\tilde S_{k,\lambda}$ be the $\tilde
S_k$-module that is free of rank one  and on which
$U(\fb^\vee)$ acts via the map $U(\fb^\vee)\to U(\fh^\vee)=S_k\stackrel{\lambda+\tau}\to\tilde
S_k$. Then we can construct the deformed standard module
$$
\tZ(\lambda):=U\otimes_{U(\fb^\vee)}\tilde S_{k,\lambda}.
$$
It is an object of  $\tCC_k$. For each $\mu\in X^\vee$
there is a deformed indecomposable projective object
$\tP(\mu)$ in $\tCC_k$ that maps surjectively to
$\tZ(\mu)$. Again these projectives admit a finite filtration with
subquotients isomorphic to deformed standard modules, and for the
multiplicities we have 
$$
(\tP(\mu):\tZ(\lambda))=(P(\mu):Z(\lambda)).
$$

\subsection{Translation functors and special representations}

Define the ``$\cdot_p$''-action  of $\hCW$ on $X^\vee$ as follows:  for $w\in\hCW$ and $\lambda\in X^\vee$ set 
$$
w\cdot_p \lambda:=pw(p^{-1}(\lambda+\rho))-\rho,
$$
where $\rho=1/2\cdot\sum_{\alpha\in R^{+,\vee}}\alpha$.

The translation principle implies that in order to know the 
Jordan-H\"older multiplicities of all standard modules it is
sufficient to calculate the Jordan-Hölder multiplicities of the
modules $Z(x\cdot_p 0)$ for $x\in \hCW$.  The linkage principle states
that $L(\mu)$ occurs as a subquotient of $Z(x\cdot_p 0)$ only if
$\mu=y\cdot_p 0$ for some $y\in\hCW$. By the above, the number
$[Z(x\cdot_p 0):L(y\cdot_p0)]$ equals the number  $(\tP(y\cdot_p 0):\tZ(x\cdot_p 0))$  for all
$x,y\in\hCW$. Moreover, an inherent $pX^\vee$-symmetry shows that it
is sufficient to assume that $y\in \hCW^{res,-}\subset\hCW^{\circ}$,
where $\hCW^{res,-}$ is the set of elements in $\hCW$ that correspond
to alcoves in the anti-fundamental box $\Pi^{-}$.

Let us denote by $\tCC_{k,0}^{V}\subset\tCC_k$ the full
subcategory of objects in the principal block that admit a {\em Verma
  flag}, i.e.~a filtration with subquotients isomorphic to standard
modules of the form $\tZ(w\cdot_p 0)$ for various $w\in\hCW$.
For $M\in \tCC_{k,0}^{V}$ we denote by $\supp\, M$ its support,
i.e.~the multiset of $w\in\hCW$ such that $\tZ(w\cdot_p0)$
occurs in a Verma flag of $M$ (so we count the $w$ in $\supp\, M$ with multiplicity).

For each $s\in\hCS$ there is a {\em translation functor} $\theta^s\colon
\tCC_{k,0}^{V}\to\tCC_{k,0}^{V}$. It has the following properties.
\begin{proposition}\label{proptrans}
\begin{enumerate}
\item $\theta^s$ is an exact and self-adjoint functor.
\item $\supp\, \theta^s M=\supp\, M\,\dot\cup\,(\supp\, M) s$.
\item If $w=s\cdots t$ is a reduced expression, then $\tP(w\cdot_p 0)$ occurs as a direct summand in
  $\theta^t\cdots\theta^s \tP(0)$ with multiplicity one.
\end{enumerate}
\end{proposition}

\begin{definition} The category of {\em special deformed modular representations} is the full subcategory $\CR_k$ of $\tCC_{k,0}^{V}$ that consists of all objects that are isomorphic to a direct summand of a direct sum of objects of the form $\theta^s\cdots\theta^t \tZ(0)$ for an arbitrary sequence $s,\dots, t\in\hCS$.
\end{definition}

\subsection{A construction of deformed projectives}
Denote by $w_0\in\CW$ the longest element in the finite
Weyl group. Note that it is the smallest element (in the Bruhat order)
in $\hCW^{res,-}$. Let $y=s_1\cdots s_n$ be a reduced expression for
$y\in\hCW^{res,-}$. We say that this is a reduced expression {\em
  through $w_0$}, if $w_0=s_1\cdots s_{l(w_0)}$. The following proposition shows that all the deformed objects that we are interested in actually appear in $\CR_k$.

\begin{proposition}\label{PasdS} Let $y\in\hCW^{res,-}$ and choose a reduced expression $y=s\cdots t$ through $w_0$. Then $\tP(y\cdot_p 0)$ is the unique indecomposable direct summand in $\theta^t\cdots\theta^s \tZ(0)$ that contains $y$ in its support.
\end{proposition}
\begin{proof} It is enough to prove the claim  for
  $y=w_0$, since the general case follows from this by Proposition
  \ref{proptrans}, (3). 

  So let $w_0=s\cdots t$ be a reduced expression. Consider a
  surjection $\tP(0)\to \tZ(0)$. Proposition
  \ref{proptrans} shows that we get a surjection
  $\theta^t\cdots\theta^s
  \tP(0)\to\theta^t\cdots\theta^s  \tZ(0)$ and
  that the module on the left hand side
  contains exactly one direct summand
  which is isomorphic to $\tP(w_0\cdot_p 0)$.

  Since $s,\dots, t\in\CW$ we can deduce the following from Proposition \ref{proptrans}:
   $$
 \left(\theta^t\cdots\theta^s  \tZ(0):\tZ(w\cdot_p 0)\right) = 
 \begin{cases} 
\left(\theta^t\cdots\theta^s
  \tP(0):\tZ(w\cdot_p 0)\right), & \text{if $w\in\CW$,} \\
  0, & \text{if $w\not\in \CW$.}
  \end{cases}
  $$
   Each direct summand $P$ of $\theta^t\cdots\theta^s
  \tP(0)$ has a Verma flag. By Corollary 9.6 of
  \cite{MR1272539} we can find a Verma flag  that starts with
  all subquotients $\tZ(w\cdot_p0)$ such that $w\not\preceq
  e$. From the following Lemma \ref{supportorder} we deduce that this Verma flag
  ends with all $\tZ(w\cdot_p 0)$ with $w\in\CW$. From the
  above multiplicity result  we deduce that the
  restriction of the map $\theta^t\cdots\theta^s
  \tP(0)\to \theta^t\cdots\theta^s
  \tZ(0)$ to $P$ is the quotient map modulo the biggest
  submodule admitting a Verma flag with subquotients $\tZ(w\cdot_p 0)$ and $w\not\in \CW$. 

  Hence the direct sum decomposition of $\theta^t\cdots\theta^s
  \tP(0)$ induces a direct sum decomposition in $\theta^t\cdots\theta^s
  \tZ(0)$. Since the support of
  $\tP(w_0\cdot_p 0)$ coincides with $\CW$ (even
  with multiplicities), the proposition is proved. 
\end{proof}

\begin{lemma}\label{supportorder} Choose $s,\dots,t\in\CS$ such that $w_0=s\cdots t$ is a reduced expression. If $w\in \supp\, \theta^t\cdots\theta^s
  \tP(0)$ is such that $w\preceq e$, then $w\in\CW$.
\end{lemma}
\begin{proof} Let us denote by $\Lambda^\prime\subset\hCW$ 
the support of $\tP(0)$ and by $\Lambda\subset\hCW$ the
support of $\theta^t\cdots\theta^s
  \tP(0)$ , both taken without multiplicities. Then $\Lambda=\Lambda^\prime\cdot \CW$. 

  The set $\Lambda$ is  partially ordered by the
  generic Bruhat order. By definition, there is $\lambda\in\DN
  R^{\vee,+}$ such that the map $ \Lambda\to\hCW$,
  $w\mapsto t_\lambda w$, preserves the order (here $\hCW$ carries the
  ordinary Bruhat order). Moreover, this map is equivariant with
  respect to the multiplication action of $\CW$ on the right.

  Now $\CW\subset\hCW$ is a parabolic subgroup, which implies that
  there is an induced (Bruhat) order on $\hCW/\CW$ such that $\hCW\to\hCW/\CW$
  is a map between partially ordered sets. Together with the results in the
  preceding paragraph we deduce that there is a
  partial order on $\Lambda/\CW$ such that the  map $\Lambda\to\Lambda/\CW$
  respects the orders (explicitely, we identify $\Lambda/\CW$ with its image under
  the composition
  $\Lambda\stackrel{t_\lambda\cdot}\to\hCW\to\hCW/\CW$ and take
the induced order). Let us fix this order on $\Lambda/\CW$.

  The maps $\Lambda^\prime\to\Lambda\to\Lambda/\CW$ respect the
  partial orders and the composition is surjective. Since $e$ is the smallest element in
  $\Lambda^\prime$, its image $\ol e$ in $\Lambda/\CW$ is the smallest
  element in $\Lambda/\CW$. Hence each element in $\Lambda$
  that is smaller than $e$ must be mapped to $\ol e$, hence lies in
  the $\CW$-orbit of $e$.
\end{proof}

\subsection{The quantum group case}
In \cite{Lus90} Lusztig constructs the {\em restricted quantum group}
$\bf u$ associated to the root system $R$, its set of positive roots
$R^+\subset R$, an integer $l$ which is prime to the entries of the
Cartan matrix of $R$, and a primitive $l$-th root of unity
$\zeta$. The quantum group is a finite dimensional $\DZ R$-graded
algebra over the field $k=\DQ(\zeta)$. 

There is a complete analog of the above definitions and structures for
the representation theory of $\bf u$. In particular, one can define
standard modules $Z(\lambda)$, simple modules $L(\lambda)$ and
projective objects $P(\lambda)$ (in a certain category). The standard and
the projective objects
 admit
deformed versions and we can define a category $\CR_{\DQ(\zeta)}$, that describes the representation theory of
$\bf u$. For an overview, see \cite{MR1403973,
  MR1357204}).

\subsection{The Andersen-Jantzen-Soergel equivalence}

Recall the definition of $\CM_k^\prime$ in Section
\ref{sec-AJStrans}. In order to compare the combinatorial category to
the category $\CR_k$ we need to extend scalars. So let
$\tCM_k^\prime$ be the combinatorial category defined in
exactly the same way as $\CM_k^\prime$, but with the algebra $S_k$
replaced by $\tilde S_k$. Then there is an obvious extension of scalars
functor $\cdot\otimes_{S_k} \tilde S_k\colon\CM_k^\prime\to\tCM_k^\prime$ that naturally commutes with the translation functors.

The following is one of the main results in \cite{MR1272539}.

\begin{theorem}\label{AJSequiv}  Suppose that $k=\DQ(\zeta)$ or $\ch k=p>h$. Then there is an equivalence 
$$
\DV\colon\CR_k\stackrel{\sim}\to\tCM_k^\prime
$$
such that  $\DV\circ \theta^s\cong \CT^\prime_s\circ \DV$ for all $s\in\hCS$,
and such that for all $M\in\CR_k$ and $x\in\hCW$  we have
$$
(M:\tZ(x\cdot_p 0))=\rk\, \DV(M)(A_x).
$$
\end{theorem}

The following partly proves the conjectures of Lusztig.

\begin{theorem}\label{MTheo1} Suppose that $k=\DQ(\zeta)$ or that $\ch
  k=p$  is large enough. For $x\in\hCW$ and
  $y\in\hCW^{res,-}$ we have 
$$
[Z(x\cdot_p 0):L(y\cdot_p 0)]=p_{w_0x,w_0y}(1).
$$
\end{theorem}

\begin{proof} By the results in Section \ref{epo} the claim is equivalent to 
$(\tP(y\cdot_p0):\tZ(x\cdot_p 0))= p_{w_0x,w_0y}(1)$. 
Consider the composition
$$
\Phi\colon\CI_k\xrightarrow{\Hyp_{\hT}^\ast}\CH_k\xrightarrow{\Psi}\CM_k\cong\CM_k^\prime\xrightarrow{\cdot\otimes_{S_k}\tilde
  S_k}\tCM_k^\prime\xrightarrow{\DV^{-1}}\CR_k.
$$
Using Theorem \ref{speSheH}, Theorem \ref{MTheo} and Theorem \ref{AJSequiv} we deduce that
$\Phi\circ\pi_s^\ast\pi_{s\ast}\cong \theta^s\circ\Phi$
and that
$$
\left(\Phi(F): \tZ(x\cdot_p0)\right) = \rk\, \Hyp_\hT^\ast(i_x^\ast F)
$$
for all $F\in \CI_k$.

Let $y\in\hCW^{res,-}$ and choose a reduced expression $y=s\cdots t$
through $w_0$. By Proposition \ref{PasdS} the object $\tP(y\cdot_p 0)$ occurs as the unique indecomposable direct summand of $\theta^t\circ\cdots\circ\theta^s \tZ(0)$ that contains $y$ in its support. Suppose that $k=\DQ(\zeta)$ or that $\ch k=p$ is large enough such that the assumptions in Theorem
\ref{ICpos2} hold. Then, on the geometric side, the sheaf $\pi_t^\ast\pi_{t\ast}\cdots\pi_s^\ast\pi_{s\ast}(F_{e})$ decomposes into a direct sum of shifted equivariant intersection cohomology complexes on Schubert varieties, and $\IC_{\hT,y}$ occurs with multiplicity one. Moreover, $\rk \, \Hyp_\hT^\ast(i_x^\ast \IC_{\hT,y})=h_{x,y}(1)$.  From the above we deduce that $\tP(y\cdot_p0)$ is a direct summand of $\Phi(\IC_{\hT,y})$, hence
$$
(\tP(y\cdot_p0):\tZ(x\cdot_p 0))\le h_{x,y}(1).
$$ 
By a result of Lusztig in \cite{MR591724} we have
$h_{x,y}(1)=p_{w_0x,w_0y}(1)$.

Now suppose that the theorem is proven for $y\in\hCW^{res,-}$ and choose $s\in\hCS$ with $ys\in \hCW^{res,-}$ and $y<ys$, if possible. 
By \cite{MR1272539}, pp.~ 238/239, the decomposition of $\theta^s
\tP(y\cdot_p 0)$ into indecomposables is {\em at most} as the
Lusztig conjecture predicts, i.e.~we have
$$
(\tP(ys\cdot_p 0):\tZ(x\cdot_p 0))\ge p_{w_0x,w_0ys}(1).
$$
Hence we deduce
$(\tP(ys\cdot_p 0):\tZ(x\cdot_p 0))= p_{w_0x,w_0ys}(1)$,
which is the claim for $ys$.
\end{proof}

\section{Braden--MacPherson sheaves on moment graphs}\label{sec-momgra}

In this final section we give an alternative construction of the
objects in $\CH_k^\circ$ as spaces of global sections of certain sheaves on a moment graph. This
allows us to apply the results in  \cite{fiebig:math0607501} and
\cite{fiebig:math0811.1674} in the representation theoretic context.

\subsection{GKM-pairs}\label{subsec-GKM}
Apart from $\ch\, k\ne 2$ and $\ne 3$ in case $R$ is of type $G_2$, the main restriction $k$ is the  GKM-property.

\begin{definition} Let $\hCG^\prime\subset \hCG$ be a finite full subgraph. We say that $(\hCG^\prime, k)$ is a  {\em GKM-pair} if the labels at two different edges of $\hCG^\prime$ that meet at a common vertex are linearly independent in $\hX\otimes_\DZ k$. 
\end{definition}

For the application to Lusztig's conjecture we recover the usual restriction on $k$. Let $\hCG^\circ$ be the full subgraph of $\hCG$ with vertices in $\hCW^\circ=\{w\in\hCW\mid w\le\hw_0\}$. 

\begin{lemma} Suppose that $\ch k=0$ or $\ch k=p\ge h$. Then $(\hCG^\circ,k)$ is a GKM-pair.
\end{lemma}
\begin{proof} In the characteristic $0$ case the pairwise linear independence
  of the positive affine roots implies the statement. So suppose that
  $\ch k$ is positive and that there are $\alpha,\beta\in R^+$,
  $m,n\in\DZ$ such that $\ol{\alpha_n}=\ol{\alpha-n\delta}$ and
  $\ol{\beta_m}=\ol{\beta-m\delta}$ are linearly dependent in
  $\hX\otimes_\DZ k$. This implies that  $\ol{\alpha}=\alpha\otimes 1$
  and $\ol{\beta}=\beta\otimes 1$ are linearly dependent in $X\otimes_\DZ k$. 

Now $\ch k\ge h$ implies that if $\ol{\alpha}$ and $\ol{\beta}$ are linearly dependent, then $\alpha=\beta$. Then $\ol{\alpha-n\delta}$ and $\ol{\alpha-m\delta}$ are linearly dependent if and only if $n-m$ is divisible by $p$.

Recall that we defined $\SA^\circ\subset \SA$ as the set of alcoves that correspond to the set $\hCW^{\circ}$. We claim that for each positive root $\gamma$, the set $\SA^{\circ}$ is contained in $H_{\gamma,h}^-\cap H_{\gamma,-h}^+$. In order to prove this we choose an alcove $A\in\SA^\circ$. Our claim is $\CW$-invariant, hence we can assume that $A$ is contained in the anti-dominant cone, i.e.~$A\subset H_{\nu,0}^-$ for all $\nu \in R^+$. In particular, $A\subset H_{\gamma, h}^-$. 

In order to finish the proof of our claim above it is enough to show that the alcove $A_{\hw_0}$ corresponding to the largest element $\hw_0\in\hCW^{\circ}$ (with respect to the Bruhat order), is contained in $H_{\gamma, -h}^+$, since $A$ is larger (with respect to the generic order) than $A_{\hw_0}$. But this is a direct consequence of  the definition of the Coxeter number.

We finish the proof of the lemma. Suppose that there is $w\in\hCW^\circ$ and
$\alpha,\beta\in R^+$, $n,m\in\DZ$ with $s_{\alpha,n}w$, $s_{\beta,m}w\in \hCW^\circ$, such that $\ol{\alpha_n}$ and $\ol{\beta_m}$ are linearly dependent.  As we have shown above, this already implies that $\alpha=\beta$ and that $p$ divides $n-m$. We have to show that $m=n$ under the assumption $p\ge h$.

Since each alcove in $\SA^\circ$ is contained in $H_{\alpha,-h}^+\cup H_{\alpha,h}^-$ the assumption $w$, $s_{\alpha,n}w$, $s_{\alpha,m}w\in \hCW^\circ$  imply that
$|m-n|< h$, from which we deduce $m=n$. 
\end{proof}

\subsection{Braden--MacPherson sheaves}

Let $x$ be a vertex of $\hCG$ and $\SF$ a sheaf on
$\hCG$. Denote by $\hCG_{>
  x}$ the sub-moment graph with set of vertices $\{y\in\CV\mid y >
x\}$. Restriction of $\SF$ to $\hCG_{> x}$ gives a sheaf $\SF_{>
  x}$ on $\hCG_{> x}$. Define $\CV_{\delta x}\subset\CV$ as the set of vertices $y$ with $y>x$ and that are connected to $x$ by an edge. Accordingly, let $\CE_{\delta x}=\{E\colon x\llinie y\mid y\in\CV_{\delta x}\}\subset\CE$ be the corresponding set of edges. Denote by $\SF^{\delta x}$ the image of the
map
$$
\Gamma(\SF_{> x})\subset \prod_{y> x}\SF^y\stackrel{p}\to
\prod_{y\in\CV_{\delta
    x}}\SF^{y}\stackrel{\rho}\to\prod_{E\in\CE_{\delta x}} \SF^E,
$$
where $p$ is the projection along the decomposition and
$\rho=\prod_{y\in\CV_{\delta x}}\rho_{y,E}$. Define  
$$
\rho_{x,\delta x}:=(\rho_{x,E})_{E\in\CE_{\delta x}}^T\colon \SF^x\to\prod_{E\in\CE_{\delta x}}\SF^E.
$$ 

\begin{theorem}[\cite{MR1871967}, cf.~also \cite{Fie05a}] There is an up to isomorphism unique $k$-sheaf $\SB_w$ on
  $\hCG$ with the following properties:
\begin{enumerate}
\item $\SB_w^w\cong \hS_k$.
\item If $x\stackrel{\alpha}\llinie y$ is an edge and $x<y$, then   the map $\rho_{y,E}\colon \SB_w^y\to\SB_w^E$ is surjective with kernel $\alpha\cdot\SB_w^y$.
\item For any $x\in\CV$, the image of $\rho_{x,\delta x}$ is $\SB_w^{\delta x}$, and  $\rho_{x,\delta x}\colon \SB_w^x\to\SB_w^{\delta x}$ is a projective cover in the category of graded $\hS_k$-modules. 
\end{enumerate}
We call the sheaf $\SB$ the {\em Braden--MacPherson sheaf} on $\hCG$. It is called the {\em canonical sheaf} in \cite{MR1871967}.
\end{theorem}
 We set $B_w:=\Gamma(\SB_w)\in\hCZ_k\catmod$. 
 
 \begin{proposition} \label{prop-globsecBM} Suppose that $\ch k=0$ or $\ch k=p\ge h$. Then each indecomposable object in  $\CH_k^\circ$ is, up to a shift in degree,  isomorphic to the $\hCZ_k$-module $B_x$ for some vertex $x$ of $\hCG^\circ$. 
\end{proposition}

\begin{proof} We have to show that for each reduced expression $w=s\cdots t$ for $w\in\hCW^\circ$, the object $\theta^t\circ\dots\circ\theta^s B_e$ is isomorphic to a direct sum of shifted copies of various $B_x$ with $x\in\hCW^\circ$. For this it is enough to prove that $\theta^s B_w$ decomposes in such a way for each $w\in\hCW^\circ$, $s\in\hCS$ with $ws\in\hCW^\circ$ and $w<ws$. 

In the following we explain some arguments developed in \cite{Fie05a,Fie05b}. Note that the following heavily depends on the GKM-property. 
In \cite{Fie05a} we associated to each subgraph $\hCG^\prime$  the full
subcategory $\CV(\hCG^\prime)_k$ of $\hCZ_k\catmod$ that consists of
modules that {\em admit a Verma flag} and we defined on
$\CV(\hCG^\prime)_k$ a non-standard exact structure.  We showed in \cite{Fie05b} that the indecomposable projective objects in $\CV(\hCG^\prime)_k$ are up to  shifts represented by the set $\{B_x\}$, where $x$ is a vertex of $\hCG^\prime$. 

We also showed  in \cite{Fie05b} that, if $w<ws$,  the functor $\theta^s\colon \hCZ_k(\le ws)\catmod\to\hCZ_k(\le ws)\catmod$ preserves $\CV(\hCG_{\le ws})_k$ and is exact and self-adjoint, hence it also preserves the subcategory of projective objects. In particular, we deduce that $\theta^s B_w$ is projective in $\CV(\hCG_{\le ws})_k$ and hence decomposes into a direct sum of shifted copies of some $B_x$.
\end{proof}

\subsection{A conjecture}\label{subsec-MomGraconj}

Let $w\in\hCW$ be arbitrary. We can then consider the full sub-moment graph $\hCG_{\le w}$ of $\hCG$ that contains all vertices that are smaller or equal to $w$.

\begin{conjecture}\label{conj-MomGra} Let $w\in\hCW$ and suppose that $(\hCG_{\le w},k)$ is a GKM-pair. Then for all $x\in\hCW$ we have
$$
\rk\, \SB_w^x=h_{x,w}(1).
$$
\end{conjecture}

By the results in the previous sections, this conjecture implies
Lusztig's modular and quantum conjectures. By \cite{Fie05a}, it also implies the Kazhdan--Lusztig conjecture on the simple highest weight characters of Kac--Moody algebras.

\subsection{The smooth locus of a moment graph}

In  \cite{Fie05a} we determined the ``smooth locus'' of moment graphs for which the intersection sheaf is self-dual. In our situation we obtain a proof of the multiplicity one case of Conjecture \ref{conj-MomGra}:  

\begin{theorem}\label{Mone} Let $w\in\hCW$ and suppose that $(\hCG_{\le w},k)$ is a GKM-pair. Then we have, for all  $x\in\hCW$,
$$
\rk\, \SB_w^x=1\text{ if and only if } h_{x,w}(1)=1.
$$
\end{theorem}

We obtain, as a special case, part (2) of Theorem \ref{theorem-MainTh} from the introduction:

\begin{corollary} Suppose that $p>h$. For $y\in\hCW^{res,-}$ and $x\in \hCW$  we have 
$$
[Z(x\cdot_p 0):L(y\cdot_p 0)]=1 \text{ if and only if } p_{w_0x,w_0y}(1)=1.
$$
\end{corollary}
\begin{proof} As in the proof of  Theorem \ref{MTheo1} we see that our claim is equivalent to 
$$
\left(\tP(y\cdot_p 0):\tZ(x\cdot_p 0)\right)=1 \text{ if and only if } p_{w_0x,w_0y}(1)=1.
$$
It is known that $\left(\tP(y\cdot_p 0):\tZ(x\cdot_p 0)\right)\ge
p_{w_0x,w_0y}(1)$ in any case. Since for $x\in\hCW$, $y\in \hCW^{\circ}$ with $x\le y$  one has $p_{w_0x,w_0y}(1)\ge 1$, we deduce that $\left(\tP(y\cdot_p
  0):\tZ(x\cdot_p 0)\right)=1$ implies  $p_{w_0x,w_0y}(1)=1$.

Now assume that $p_{w_0x,w_0y}(1)=h_{x,y}(1)=1$.
We consider the composition 
$$
\Phi^\prime\colon\CH_k\xrightarrow{\Psi}\CM_k\cong\CM_k^\prime\xrightarrow{\cdot\otimes_{S_k}
  \tilde S_k}\tCM_k^\prime\xrightarrow{\DV^{-1}}\CR_k.
$$
In a similar way as in the proof of Theorem \ref{MTheo1} we deduce that $\tP(y\cdot_p0)$ is a direct summand of $\Phi^\prime(B_y)$, and Theorem \ref{Mone} then implies
$$
\left(\tP(y\cdot_p 0):\tZ(x\cdot_p 0)\right)\le 1.
$$
The reverse inequality is stated above, hence $\left(\tP(y\cdot_p
  0):\tZ(x\cdot_p 0)\right)=1$.
\end{proof}

\subsection{An upper bound on the exceptional primes}\label{subsec-upbound}
Let $\bs=(s_1,\dots,s_{l})$ be a sequence in $\hCS$ of length $l$.  Let us
define the corresponding Bott--Samelson element $\ul{H}(\bs)$ in the
affine Hecke algebra by 
$$
\ul{H}({\bs}):=\ul{H}_{s_1}\cdots\ul{H}_{s_l}.
$$
 Let the
polynomials $a_x\in\DZ[v]$ be defined by 
$$
\ul{H}({\bs})=\sum_{x\in \hCW} a_x \tilde T_x.
$$
Set $r_x=a_x(1)$ and set
$$
r=r(\bs):=\max_{x}\{r_x\}.
$$
Let $d_x=\left(\frac{d}{dv}a_x\right)(1)$ be the sum of the exponents
of $a_x$, and set
$$
d=d(\bs):=\max_{x}\{d_x\}.
$$

We denote by $\hR^+_{\bs}\subset\hR^+$ the subset of all positive roots that appear
as a label on the subgraph of $\hCG$ that contains all vertices that are smaller or equal to a subword  of $s_1\cdots
s_l$.  We define the
{\em height} $\height(\alpha)$ of a positive affine root
$\alpha\in \hR^+$ as the number $n$ such that $\alpha$ can be written
as a sum of $n$ elements of $\widehat\Pi$ and we set
$$
N=N(\bs):=\max_{\alpha\in\hR^+_{\bs}}\{\height(\alpha)\}.
$$

Now we associate to $\bs$ the number
$U(\bs):=r!(r!(r-1)!N^{l+2d})^r$ and set, for $w\in\hCW$, 
$$
U(w):=\min_{\bs} U(\bs),
$$
where the minimum is taken over all sequences $\bs$ such that $s_1\cdots s_l$ is a reduced expressions for $w$. Note
that $U(w^\prime)\le U(w)$ for $w^\prime \le w$. The main result in \cite{fiebig:math0811.1674} is the following:

\begin{theorem} Let $w\in\hCW$ and suppose that $\ch k=p>U(w)$. Then Conjecture \ref{conj-MomGra} holds for the moment graph $\hCG_{\le w}$. 
\end{theorem}

We now obtain part (3) of Theorem \ref{theorem-MainTh} from the introduction:
\begin{corollary} Suppose that $p> U(\hw_0)$. For $x\in\hCW$ and
  $y\in\hCW^{res,-}$ we have 
$$
[Z(x\cdot_p 0):L(y\cdot_p 0)]=p_{w_0x,w_0y}(1).
$$
\end{corollary}

\end{document}